\documentclass[11pt]{article}
\usepackage{amsmath, amssymb, amsfonts, amsxtra, latexsym, amscd, eucal, amsthm, graphicx, epsfig, color, enumerate}      % use if color is used in text
\usepackage{tikz}
\usepackage{subcaption}

\usepackage{hyperref}   % use for hypertext links, including those to external documents and URLs
\theoremstyle{theorem}
\newtheorem{Def}{Definition}[section]

\newtheorem{Lem}[Def]{Lemma}
\newtheorem{Thm}[Def]{Theorem}
\newtheorem{Cor}[Def]{Corollary}
\theoremstyle{definition}

\newtheorem{Rem}[Def]{Remark}

\newcommand{\bR}{\mathbb{R}}

\newcommand{\mf}{\mathcal{F}}

\newcommand{\bN}{\mathbb{N}}

\newcommand{\pr}{\mathbb{P}}

\newcommand{\Om}{\Omega}

\usepackage[top=2cm, bottom=2cm, left=2cm, right=2cm]{geometry}

\allowdisplaybreaks

\begin{document}

\title{Strong convergence rate of Euler-Maruyama approximations in temporal-spatial H\"older-norms for L\'evy-driven stochastic differential equations}
\author{ Vu Thi Hue\footnote{Faculty of Mathematics and Informatics, Hanoi University of Science and Technology, 1 Dai Co Viet, Bach Mai, Hanoi, Vietnam. Email: hue.vuthi@hust.edu.vn} \quad 
Ngoc Khue Tran\footnote{Corresponding author. Faculty of Mathematics and Informatics, Hanoi University of Science and Technology, 1 Dai Co Viet, Bach Mai, Hanoi, Vietnam. Email: khue.tranngoc@hust.edu.vn} \quad 
Hoang-Long Ngo\footnote{   Department of Mathematics and Informatics, School of Mathematics and Computer Science, Hanoi  National  University  of  Education, 136  Xuan  Thuy, Cau  Giay, Hanoi, Vietnam. Email: ngolong@hnue.edu.vn}}  
\maketitle

\begin{abstract} 
We study the error between the exact solution and its Euler-Maruyama approximation in temporal-spatial H\"older-norms for L\'evy-driven stochastic differential equations. 
\end{abstract} 



\section{Introduction} 

On a complete probability space $(\Om, \mf, \pr)$, we consider the process $X=(X_t)_{t\geq 0}$ solution to the following one-dimensional stochastic differential equation (SDE) with jumps
\begin{equation} \label{eqn1}
dX_t=\mu(X_t)dt+\sigma(X_t)dW_t + \gamma \left(X_{t-}\right) dZ_t, 
\end{equation}
where $W=(W_t)_{t\geq 0}$ is a standard Brownian motion, and $Z=(Z_t)_{t\geq 0}$ is a centered pure jump L\'evy process whose L\'evy measure $\nu$ satisfies $\int_{\mathbb{R}}\min\{1, z^2\} \nu(dz)<+\infty$. Two processes $W$ and $Z$ are assumed to be independent.  The natural filtration $(\mf_t)_{t\geq 0}$ is generated by two processes $W$ and $Z$ and $X_0$ is $\mathcal{F}_0$-measurable. Let $\mathcal{B}(\mathbb{R}_{+}\times \mathbb{R}_0)$ be the Borel $\sigma$-algebra on $\mathbb{R}_{+}\times \mathbb{R}_0$, where $\mathbb{R}_0:=\mathbb{R}\setminus\{0\}$.  The L\'evy-It\^o decomposition of $Z$ is defined by
$$
Z_t=\int_{0}^{t}\int_{\mathbb{R}_0}z(N(ds,dz)-\nu(dz)ds), 
$$
for any $t\geq 0$, where $N(dt,dz)$ is a Poisson random measure on the measurable space $(\mathbb{R}_{+}\times \mathbb{R}_0,\mathcal{B}(\mathbb{R}_{+}\times \mathbb{R}_0))$  associated  with the intensity measure $\nu(dz)dt$. That is, 
$$
N([0,t] \times A):=\sum_{0\leq u\leq t}{\bf 1}_{\{\Delta Z_u \in A \}},
\text{ for any } t > 0, \text{ and  } A \in \mathcal{B}(\mathbb{R}_0).$$
Here, the jump size of $Z$ at instant $u$ is defined as $\Delta Z_u:=Z_u-Z_{u-}:=Z_u-\lim_{r \uparrow u}Z_{r}$ for any $u>0$, $\Delta Z_0:=0$.
The compensated Poisson random measure associated with $N(dt,dz)$ is denoted by $\widetilde{N}(dt,dz):=N(dt,dz)-\nu(dz)dt$.

Since stochastic differential equations of the form \eqref{eqn1} rarely admit explicit or directly simulatable solutions, their practical implementation requires the construction of suitable numerical approximations on discrete time grids (see \cite{KP92}). Among the available methods, the Euler-Maruyama scheme is one of the most widely used approaches. To introduce this scheme, let \(\mathbb{S}\) denote the collection of all discrete partitions of the interval \([0,T]\), namely,
\begin{align*}
 \mathbb{S}=\left \{
\delta: [0, T]\to [0, T]: 
\begin{aligned}
&\exists\, n\in\bN,\ t_0,t_1,\dots,t_n\in [0,T] \text{ such that } 0=t_0<t_1<\cdots<t_n=T, \\
&\delta([t_0,t_1])=\{t_0\},\ \delta((t_1,t_2])=\{t_1\},\ \ldots,\ 
\delta((t_{n-1},t_n])=\{t_{n-1}\}
\end{aligned}
\right\}.
\end{align*}
For every \(\delta\in\mathbb{S}\), \(s\in[0,T]\), and \(x\in\mathbb{R}\), we define the Euler-Maruyama approximation associated with the grid \(\delta\), denoted by
\[
(X^{\delta,x}_{s,t})_{t\in[s,T]}:[s,T]\times\Omega\to\mathbb{R},
\]
as the \((\mathcal{F}_t)_{t\in[s,T]}\)-adapted c\`adl\`ag process satisfying, for all \(t\in[s,T]\), almost surely,
\begin{align}
\label{y01}
 X^{\delta,x}_{s,t}
&=
x+\int_{s}^{t}\mu\bigl(X^{\delta,x}_{s,\max\{s,\delta(r)\}}\bigr)\,dr
+\int_{s}^{t}\sigma\bigl(X^{\delta,x}_{s,\max\{s,\delta(r)\}}\bigr)\,dW_r
\notag\\
&\quad
+\int_{s}^{t}\gamma\bigl(X^{\delta,x}_{s,\max\{s,\delta(r-)\}}\bigr)\,dZ_r
\notag\\
&=
x+\int_{s}^{t}\mu\bigl(X^{\delta,x}_{s,\max\{s,\delta(r)\}}\bigr)\,dr
+\int_{s}^{t}\sigma\bigl(X^{\delta,x}_{s,\max\{s,\delta(r)\}}\bigr)\,dW_r
\notag\\
&\quad
+\int_{s}^{t}\int_{\mathbb{R}_0}
\gamma\bigl(X^{\delta,x}_{s,\max\{s,\delta(r-)\}}\bigr)
\,z\,\widetilde{N}(dr,dz).
\end{align}

Classical error analyses for numerical schemes are typically formulated in terms of \(L^p\)-errors evaluated at fixed time points. More precisely, one studies the convergence of
\[
\mathbb{E}\bigl[|X_t-X^{\delta,x}_{0,t}|^p \,\big|\, X_0=x\bigr]
\]
as the mesh size \(|\delta|\to0\), for some \(t\in[0,T]\) and \(p>0\), where
\[
|\delta|
=
\max\Bigl\{
|s-t|:
s,t\in\delta([0,T]),\ s<t,\ (s,t)\cap\delta([0,T])=\emptyset
\Bigr\}.
\]
Although such estimates provide important information on pointwise convergence, they do not adequately capture the dependence of the solution on the initial condition nor its global regularity properties.

Recently, the study of convergence rates in temporal-spatial H\"older norms has attracted considerable attention in the numerical analysis of stochastic differential equations,
%Let $\widetilde{\mathbb{S}} = \mathbb{S}\cup\{\iota \}$, let $|.|: \widetilde{\mathbb{S}} \to [0,T]$ satisfy that $|\iota| =0$ and 
%The analysis of convergence rates in temporal-spatial H\"older norms has emerged as an interesting topic in the numerical approximation of stochastic differential equations (SDEs). Classical error estimates for numerical schemes typically focus on \(L^p\)-errors evaluated at fixed time points (see, e.g., \cite{KP92}), and thus fail to capture the dependence of the solution on the initial condition as well as its global regularity properties. In contrast, 
since they allow for a uniform quantification of approximation errors over the underlying domain, thereby providing a more comprehensive description of global accuracy. This framework is particularly well suited for the investigation of spatial regularity of approximation processes, which plays a crucial role in the stability and efficiency of subsequent numerical procedures. In \cite{HN22}, the authors  have established strong convergence rates in temporal-spatial H\"older norms under suitable regularity assumptions on the coefficients, typically requiring bounded first- and second-order derivatives. These results have important implications for the development of multigrid and multilevel approximation techniques, which significantly outperform classical interpolation methods on fine grids in high-dimensional settings. In particular, spatial regularity estimates for Euler-Maruyama schemes constitute a key ingredient in the convergence analysis of multilevel Picard methods for nonlinear parabolic partial differential equations \cite{NNW25}. Moreover, such estimates contribute to the refinement of multilevel Monte Carlo methods, leading to near-optimal computational complexity for the approximation of global solutions to associated integral equations \cite{Giles2008}. They are also of fundamental importance in financial engineering, where backward stochastic differential equations are studied via representations in terms of forward processes, and where global regularity is essential for the robustness of numerical implementations \cite{HJKN23}.

The objective of the present paper is to extend the strong convergence analysis of the Euler-Maruyama scheme in temporal-spatial H\"older norms, developed in \cite{HN22}, to stochastic differential equations driven jointly by Brownian motion and pure-jump L\'evy processes. The presence of jumps introduces substantial analytical difficulties due to the discontinuity of sample paths and the distinct scaling behaviour of jump increments. To overcome these difficulties, we establish a new version of the Lyapunov estimate in Lemma~2.2 of \cite{CHJ22}, together with an extension of the Gronwall-Lyapunov inequality from Theorem~2.4 of \cite{HHM21} adapted to the jump setting. In particular, unlike Brownian increments $W_{t+s}-W_{t}$, whose \(p\)-moments are of order \(s^{p/2}\), the increments $Z_{t+s}-Z_{t}$ satisfy moment estimates of order \(s\) for all \(p\geq2\). Consequently, the temporal-spatial H\"older convergence analysis in the present framework differs substantially from that in \cite{HN22}, and requires a refined treatment of the jump contributions in both the temporal and spatial estimates.

The paper is organized as follows. We present some preliminary results in Section \ref{Preli}. The main result on the estimate of error in temporal-spatial H\"older norms is considered in Section \ref{mainresult}.

\section{Some preliminary results} \label{Preli}

The integral equation \eqref{eqn1} now can be written as 
\begin{equation*} 
X_t=X_0+\int_0^t \mu(X_s)ds+\int_0^t \sigma(X_s)dW_s + \int_0^t\int_{\mathbb{R}_0} \gamma \left(X_{s-}\right)z\widetilde{N}(d s, d z), \quad t \geq 0.
\end{equation*}	

% Numerical experiments and simulations are provided in Section \ref{sec:num}. 

For the sake of convenience, we denote $ m_p:= \int_{\bR_0}\vert z\vert^p\nu(d z)$ for $p>0$. We denote by $\mathbb{E}$ the expectation with respect to $\pr$ and by $C^2(\mathbb{R}, \mathbb{R})$ the set of functions with first and second bounded derivatives. Let $C^{1,2}([0, T]\times \mathbb{R}, [0, \infty))$ be the class of functions which are continuously differentiable once with respect to the time variable and twice with respect to the spatial variable, and $L([0, T], [0, \infty))$  denotes the set of measurable functions from $[0, T]$ to $[0, \infty)$. For all $x, y \in \mathbb{R}$, $x \land y: =\min \{x, y\}$ and $x \vee y: =\max \{x, y\}$.

As usual, positive constants will be denoted by $C_p$ whose values may vary from one line to the next.

Next, we introduce some preliminary results on  Gronwall inequalities, Lyapunov estimate, Gronwall-Lyapunov inequality and Burkholder-Davis-Gundy's inequality with jumps, which are needed for the proof of the main results.

%\begin{Lem} (\cite[Lemma 2.1]{HN22})
%  Let $a, c_{\star} \in [0, \infty), T\in \mathbb{R}, t_0\in (-\infty, T)$, let $x: [t_0, T] \to [0, \infty), \delta: [t_0,T] \to [t_0, T]$ be measurable, and assume for all $t\in [t_0, T]$ that $\delta (t)\leq t$ and $x(t) \leq a+ \int_{t_0}^t c_{\star}x(\delta(s)) ds <\infty$. Then for all $t\in [t_0, T]$, it holds that $x(t)\leq a e^{c_{\star} (t-t_0)}$.   
%\end{Lem}	

\begin{Lem}
     (\cite[Corollary 2.2]{HN22})
\label{Cor2.1}
Let $t_0 < T$, and let $a, x : [t_0, T] \to [0,\infty)$ and
$\delta : [t_0, T] \to [t_0, T]$ be measurable functions such that
$\delta(t) \le t$ for all $t \in [t_0, T]$. Assume that, for some
constants $c_\star > 0$ and $p \ge 1$,
\[
x(t) \le a(t) + \left( \int_{t_0}^t |c_\star x(\delta(s))|^p \, ds \right)^{1/p}
< \infty, \quad t \in [t_0, T].
\]
Then, for all $t \in [t_0, T]$, it holds that
\[
x(t) \le 2^{1-\frac{1}{p}}
\left( \sup_{s \in [t_0, t]} a(s) \right)
\exp\left( \frac{2^{p-1} c_\star^p (t - t_0)}{p} \right).
\]
\end{Lem}

Next, we extend the result on the Lyapunov estimate given in Lemma 2.2 of \cite{CHJ22} in the presence of jumps.
\begin{Lem} (A Lyapunov estimate) \label{Lem2.2}
 Let $V\in C^{1,2}([0, T]\times \mathbb{R}, [0, \infty)), \alpha \in L([0, T], [0, \infty))$ with $\int_0^T \alpha (t) dt<\infty,$ let $\tau: \Omega\to [0, T]$ be a stopping time, let the coefficients $\mu, \sigma, \gamma$ be continuous functions and let $X: [0, T] \times \Omega \to \mathbb{R}$ be the c\`adl\`ag and adapted  solution process to the equation \eqref{eqn1} with sample paths satisfying
\begin{align}
    \label{ConLem2.3}
        & \dfrac{\partial V}{\partial t} (t \land \tau, X_{t \land \tau}) + \dfrac{\partial V}{\partial x} (t \land \tau, X_{t \land \tau}) \mu (X_{t \land \tau})+ \dfrac{1}{2} \dfrac{\partial^2 V}{\partial x^2}(t \land \tau, X_{t \land \tau}) \sigma^2 (X_{t \land \tau}) \notag \\
        & \quad + \int_{\bR_0} \left[ V(t \land \tau, X_{t \land \tau} + \gamma (X_{t \land \tau})z ) - V(t \land \tau, X_{t \land \tau}) - \dfrac{\partial V}{\partial x}(t \land \tau, X_{t \land \tau}) \gamma (X_{t\land \tau}) z \right] \nu (dz) \notag \\
        & \leq \alpha (t \land \tau) V(t \land \tau, X_{t \land \tau}), \quad \mathbb{P} \text{-a.s. for all } t \in [0,T].
    \end{align}
    Then,
\begin{align*}
    \mathbb{E}\left[V(\tau, X_{\tau})\right] \leq \exp \left( \int_0^T \alpha(s) ds \right) \mathbb{E} [V(0, X_0)] \in [0,\infty].
\end{align*}
\end{Lem}
\begin{proof}
 First, without loss of generality, we assume that $\mathbb{E}[V(0, X_0)] < \infty.$ Since $X$ is the c\`adl\`ag stochastic process and the coefficients $\mu, \sigma, \gamma$ are continuous functions, we have
  \begin{align*}
    \int_0^T \left(|\mu (X_s)|+ \sigma^2(X_s) + m_2 \gamma ^2(X_s)\right) ds  < \infty  \quad \mathbb{P} \text{-a.s.},
    \end{align*} 
 where recall that $ m_2= \int_{\bR_0}z^2\nu(d z)$.
For all $n\in \mathbb{N}$, we define stopping times $\rho_n: \Omega \to [0, T]$ by
  %\textcolor{red}{con thieu phan jumps nua phai k0 ? + gradient nhieu chieu hay dao ham rieng theo $x$}
\begin{align*}
    \rho_n:&= \inf \Bigg(  \{\tau\}\cup \Bigg\{ t\in [0, T]: \sup_{s\in [0, t]} V(s, X_s)+ \int_0^t \Bigg |\dfrac{\partial V}{\partial x}  (s, X_s) \sigma (X_s)\Bigg|^2 ds  \notag \\
    &  \qquad + \int_0^t \int_{\bR_0} \left| V(s, X_{s} + \gamma (X_{s })z ) - V(s, X_{s}) \right|^2 \nu (dz) ds  \geq n \Bigg\} \Bigg).
\end{align*}
Then, using It\^o's formula, we have for all $(t,n) \in [0,T]\times \mathbb{N}$,
%\textcolor{red}{can chu y $X_{s-}$ khi tich phan di kem voi do do $\widetilde{N}(ds,dz)$; da vua sua xong bieu thuc phia duoi. Tu day tro ve sau can ra soat sua lai tat ca, con rat nhieu o phia duoi}
\begin{align*}
   V(t\land \rho_n, X_{t\land \rho_n}) 
   & = V(0, X_0) 
   +\int_0^{t\land \rho_n} \Bigg( \dfrac{\partial V}{\partial t}  (s, X_s) +   \dfrac{\partial V}{\partial x}  (s, X_s) \mu (X_s) +  \dfrac{1}{2} \dfrac{\partial^2 V}{\partial x^2}(s, X_s) \sigma^2 (X_s)  \notag \\
   & \qquad +\int_{\bR_0} \left[ V(s, X_s + \gamma (X_{s })z ) - V(s, X_s) - \dfrac{\partial V}{\partial x}(s, X_s) \gamma (X_{s })z \right] \nu (dz)\Bigg) ds\notag \\
   & \qquad + \int_0^{t\land \rho_n}  \dfrac{\partial V}{\partial x}  (s, X_s) \sigma (X_s) dW_s \notag \\
     & \qquad + \int_0^{t\land \rho_n} \int_{\bR_0} \left[ V(s, X_{s-} + \gamma (X_{s- })z ) - V(s, X_{s-})  \right] \widetilde{N}(ds,dz). 
\end{align*}
Hence, using condition $\eqref{ConLem2.3}$, we get  
\begin{align*}
   V(t\land \rho_n, X_{t\land \rho_n}) & \leq V(0, X_0) + \int_0^{t\land \rho_n} \alpha(s)V(s, X_s)ds + \int_0^{t\land \rho_n} \dfrac{\partial V}{\partial x}  (s, X_s) \sigma (X_s) dW_s \notag \\
   & \qquad + \int_0^{t\land \rho_n} \int_{\bR_0} \left[ V(s, X_{s-} + \gamma (X_{s- })z ) - V(s, X_{s-})  \right] \widetilde{N}(ds,dz).  
\end{align*}
 Next, using the definition of $\rho_n$ and taking expectations of both sides, we obtain
 \begin{align*}
    \mathbb{E} \left[ V(t\land \rho_n, X_{t\land \rho_n})\right] & \leq  \mathbb{E} [V(0, X_0)] + \int_0^{t} \alpha(s)\mathbb{E}[{\bf 1}_{(s \leq \rho_n)} V(s, X_s)]ds \notag \\
    & \leq \mathbb{E}[V(0, X_0)] + \int_0^{t} \alpha(s)\mathbb{E}[V(s\land \rho_n, X_{s\land \rho_n})] ds,
 \end{align*}
for all $(t,n) \in [0,T]\times \mathbb{N}$. Note that  $\mathbb{E} \left[ V(t\land \rho_n, X_{t\land \rho_n})\right] \leq  \mathbb{E}[V(0, X_0)] + n \int_0^{t} \alpha(s) ds <\infty$ for all $(t,n) \in [0,T]\times \mathbb{N}$. Applying Gronwall's lemma, we get 
$$\mathbb{E} \left[ V(t\land \rho_n, X_{t\land \rho_n})\right] \leq \exp \left( \int_0^t \alpha(s) ds \right) \mathbb{E}[V(0, X_0)],$$ 
for all $(t,n) \in [0,T]\times \mathbb{N}$.
Consequently, taking $n\uparrow \infty$ and using Fatou’s lemma, we obtain
$$\mathbb{E} \left[ V(\tau, X_{\tau})\right] \leq \exp \left( \int_0^T \alpha(s) ds \right) \mathbb{E}[V(0, X_0)].$$
Thus, the result follows.
\end{proof}
Next, we extend the result on the stochastic Gronwall-Lyapunov inequality given in Theorem 2.4 of \cite{HHM21} in the presence of jumps.

\begin{Lem}

    \label{Lem2.4} (A Gronwall-Lyapunov inequality) 
Let $\tau $ be a stopping time,  $\widehat{\mu}, \widehat{\sigma}$ and $\widehat{\gamma}$ real-valued measurable functions such that 
\begin{equation} \label{eqn:hat}
\int_0^{\tau} \left(|\widehat{\mu} (\xi_s)| + |\widehat{\sigma}(\xi_s)|^2 + | \widehat{\gamma}(\xi_s)|^2 \right) ds <\infty, 
 \quad \mathbb{P}\text{-a.s.},
 \end{equation} 
 where  $\xi$ is an adapted  c\`adl\`ag process satisfying for all $t\in [0, T]$,   
 $$\xi_{t \wedge \tau } = \xi_0 + \int_0^t {\bf 1}_{[0, \tau]}(s)\widehat{\mu}(\xi _s)ds+ \int_0^t {\bf 1}_{[0, \tau]}(s)\widehat{\sigma}(\xi _s)dW_s + \int_0^t {\bf 1}_{[0, \tau]}(s)\widehat{\gamma}(\xi _{s-}) dZ_s, \quad  \mathbb{P} \text{-a.s.}.$$   
Let  $\alpha, \beta: [0, T] \to [0, \infty )$  be measurable functions  satisfying  $\int_0^{T}|\alpha_r|dr <\infty$. Suppose that for some function  $V=(V(s, x))_{s\in [0, T], x\in \mathbb{R}} \in C^{1,2}([0, T] \times \mathbb{R}, [0, \infty))$ it holds for any $s\in [0, u]$ with $u\in [0, T]$    that
\begin{align}
\label{Con2.4}
& \dfrac{\partial V}{\partial s} (s, \xi_s) + \dfrac{\partial V}{\partial x} (s, \xi_s) \widehat{\mu}(\xi_s)+\dfrac{1}{2}\frac{\partial^2 V}{\partial x^2}(s, \xi_s)\widehat{\sigma}^2(\xi_s)  \notag \\
& \quad + \int_{\bR_0} \left[ V(s, \xi_s+ \widehat{\gamma}(\xi_{s})z)-V(s, \xi_s) - \dfrac{\partial V}{\partial x} (s, \xi_s) \widehat{\gamma}( \xi_{s})z \right] \nu (dz)  \notag \\
& \leq \alpha(s) V(s, \xi_s)+  \beta(s) \quad \mathbb{P} \text{-a.s.}.
\end{align}
Then, for all $t \in [0, T]$, 
%\begin{align*}
% \mathbb{E} [V(u, \xi_{u})] \leq \exp \left( \int_0^{u} \alpha(r) dr \right) \left( \mathbb{E}[V(0, \xi_0)] +\int_0^{u}  \dfrac{  \beta(s)}{\exp \left( \int_0^{s} \alpha(r) dr \right)} ds \right).
%\end{align*}
\begin{align*}
\mathbb{E} \left[ \dfrac{V(\tau \wedge t, \xi_{\tau \wedge t})}{\exp \left( \int_0^{\tau \wedge t} \alpha(r)dr\right)}\right] \leq \mathbb{E}[V(0, \xi_0)]  + \mathbb{E} \left[ \int_0^{\tau \land t}  \dfrac{\beta(s)}{\exp \left( \int_0^{s} \alpha(r) dr \right)}ds \right]. 
\end{align*}
\end{Lem}

\begin{proof}
For each $n\geq 1$, let $\tau_n$ be a stopping time defined by 
\begin{align*}
   &\tau_n=\inf \left( \{\tau \}  \cup \left\{  s\in [0, T]: V(s, \xi_s)+ \int_0^s  \left|\dfrac{\partial V}{\partial x}(r, \xi_r) \widehat{\sigma}  (\xi_r)\right|^2 dr \right. \right. \notag \\
    &  \quad \quad \quad \quad  +\left. \left. \int_0^s \int_{\bR_0} \left| V(r, \xi_{r} + \widehat{\gamma} (\xi_{r })z ) - V(r, \xi_{r}) \right|^2 \nu (dz) dr  \geq n \right\} \right).
    \end{align*}
Applying It\^o's formula, we get for every $t\in [0, T],$
\begin{align*}
 &\dfrac{ V(\tau_n \wedge t, \xi_{\tau_n \wedge t}) }{\exp \left( \int_0^{\tau_n \wedge t} \alpha(r) dr \right)} = V(0, \xi_0)  + \int_0^{\tau_n \wedge t} \dfrac{1}{\exp \left( \int_0^{s} \alpha(r) dr \right)}\Bigg( \frac{\partial V}{\partial s}(s, \xi_s)- V(s, \xi_s) \alpha(s)  \notag \\
& \qquad +  \frac{\partial V}{\partial x}(s, \xi_s)\widehat{\mu}(\xi_s)+\frac{1}{2}\frac{\partial^2 V}{\partial x^2}(s, \xi_s)\widehat{\sigma}^2(\xi_s)  \notag \\ 
 & \qquad +  \int_{\bR_0} \left[ V(s, \xi_s + \widehat{\gamma} (\xi_{s})z ) - V(s, \xi_s) - \dfrac{\partial V}{\partial x}(s, \xi_s) \widehat{\gamma} (\xi_{s})z  \right] \nu (dz) \Bigg) ds \notag \\
& \qquad + \int_0^{\tau_n \wedge t}  \dfrac{\frac{\partial V}{\partial x}(s, \xi_s)\widehat{\sigma}(\xi_s)}{\exp \left( \int_0^{s} \alpha(r) dr \right)} dW_s + \int_0^{\tau_n \wedge t} \int_{\bR_0}  \dfrac{V(s, \xi_{s-} + \widehat{\gamma} (\xi_{s- })z ) - V(s, \xi_{s-}) }{\exp \left( \int_0^{s} \alpha(r) dr \right)}   \widetilde{N}(ds,dz). 
\end{align*}
Then taking expectations both sides and using condition \eqref{Con2.4}, we get
\begin{align*}
\mathbb{E}\left[\dfrac{ V(\tau_n \land t, \xi_{\tau_n \land t})}{\exp \left( \int_0^{\tau_n \land t} \alpha(r) dr \right)} \right] 
& \leq  \mathbb{E}[V(0, \xi_0)]  + \mathbb{E} \left[ \int_0^{\tau_n \land t}  \dfrac{\beta(s)}{\exp \left( \int_0^{s} \alpha(r) dr \right)} \right]ds.
%\notag\\ 
% &\le  \mathbb{E}[V(0, \xi_0)]  + \int_0^{ t} \mathbb{E} \left[\dfrac{{\bf 1}_{[0, \tau_n]}(s) \beta(s)}{\exp \left( \int_0^{s} \alpha(r) dr \right)} \right]ds.
\end{align*}
%Hence, replacing $t$ with $T$ yields 
%\begin{align*}  \mathbb{E}\left[\dfrac{ V(\tau_n \land T, \xi_{\tau_n \land T})}{\exp \left( \int_0^{\tau_n \land T} \alpha(r) dr \right)} \right]   \le  \mathbb{E}[V(0, \xi_0)]  +  \int_0^{ T} \mathbb{E} \left[\dfrac{{\bf 1}_{[0, \tau_n]}(s) \beta(s)}{\exp \left( \int_0^{s} \alpha(r) dr \right)} \right]ds. \end{align*}
Moreover, 
% the fact that $(V(s, \xi_s))_{s\in [0, T]}$, $(\frac{\partial V}{\partial x}(s, \xi_s))_{s\in [0, T]}$ and $(\frac{\partial V}{\partial x}(s, \xi_s +\theta_{(\xi_s, \widehat{\gamma}(\xi_x)z)}\widehat{\gamma}(\xi_s)z))_{s\in [0, T], \theta_{(\xi_s, \widehat{\gamma}(\xi_x)z)} \in [0,1]}, $ 
the fact that $\xi_s$ has c\`adl\`ag sample path, it follows from \eqref{eqn:hat} and the fact  $V \in C^{1,2}([0, T] \times \mathbb{R}, [0, \infty))$ that 
$\mathbb{P}( \lim_{n\to \infty} \tau_n = \tau)=1.$
 % Taking $n\uparrow \infty$, using Fatou’s lemma and and the monotone convergence theorem together with $\mathbb{E}[|1+V(0, \xi_0)|^p] < \infty$
  Thanks to Fatou's lemma, we have 
 \begin{align*}
    \mathbb{E} \left[ \dfrac{V(\tau \wedge t, \xi_{\tau \wedge t})}{\exp \left( \int_0^{\tau \wedge t} \alpha(r)dr\right)}\right] 
%     & = \mathbb{E} \left[ \dfrac{V(u, \xi_{u-})}{\exp \left( \int_0^{u} \alpha(r)dr\right)}\right] 
&= \mathbb{E} \left[ \liminf_{n\to \infty} \dfrac{V(\tau_n \wedge t, \xi_{\tau_n \wedge t})}{\exp \left( \int_0^{\tau_n \wedge t} \alpha(r)dr\right)} \right] \notag\\
     & \leq  \liminf_{n\to \infty} \mathbb{E} \left[\dfrac{ V(\tau_n \wedge t, \xi_{\tau_n \wedge t})}{\exp \left( \int_0^{\tau_n \wedge t} \alpha(r)dr\right)} \right] \notag\\
  & \leq   \liminf_{n\to \infty}  \left( \mathbb{E}[V(0, \xi_0)]  + \mathbb{E} \left[ \int_0^{\tau_n \land t}  \dfrac{\beta(s)}{\exp \left( \int_0^{s} \alpha(r) dr \right)}ds \right] \right). \notag \\
   & =  \mathbb{E}[V(0, \xi_0)]  + \mathbb{E} \left[ \int_0^{\tau \land t}  \dfrac{\beta(s)}{\exp \left( \int_0^{s} \alpha(r) dr \right)}ds \right]. 
 %  & \leq  \liminf_{(0,1] \ni \epsilon \to 0}  \left(\epsilon + \mathbb{E}[V(0, \xi_0)]  + \int_0^T \dfrac{  {\bf 1}_{[0, u]}(s)\beta(s)}{\exp \left( \int_0^{s} \alpha(r) dr \right)} ds\right) \notag\\
%   & = \mathbb{E}[V(0, \xi_0)]  +  \int_0^u \dfrac{ \beta(s)}{\exp \left( \int_0^{s} \alpha(r) dr \right)} ds. 
\end{align*}
%Therefore, for all $u \in [0, T]$,
%\begin{align*}
% \mathbb{E} [V(u, \xi_{u})] \leq \exp \left( \int_0^{u} \alpha(r) dr \right) \left( \mathbb{E}[V(0, \xi_0)] +\int_0^{u}  \dfrac{   \beta(s)}{\exp \left( \int_0^{s} \alpha(r) dr \right)}  ds \right).
%\end{align*}
This finishes the proof.
\end{proof}

We recall Kunita's inequality and the Burkholder-Davis-Gundy's inequality with jumps.
\begin{Lem}[\cite{A59}, Theorem 4.4.23  and \cite{Zhu}, Proposition 2.2]
	\label{BDGjump}  Let $\mathcal{P}$ be the progressive $\sigma$-algebra on $\mathbb{R}_+\times\Omega$ and $\mathcal{B}(\mathbb{R}_0)$ be the Borel $\sigma$-algebra of \;$\mathbb{R}_0$. Assume that $g$ is a $\mathcal{P}\otimes \mathcal{B}(\mathbb{R}_0)$-measurable function such that $\int_0^T\int_{\mathbb{R}_0}\vert g(s,z) \vert^2\nu(dz)ds<\infty$ $\pr$-a.s. for all $T\geq 0$. Then, for any $p\geq 2$, there exists a constant $C_p >0$ such that
	\begin{align*}
		&\mathbb{E}\left[\sup_{0\leq t \leq T}\left\vert \int_0^t\int_{\mathbb{R}_0} g(s,z) \widetilde{N}(d s, d z)\right\vert^p\right]\\
		&\leq C_p \left(\mathbb{E}\left[\left(\int_0^T\int_{\mathbb{R}_0} \vert g(s,z) \vert^2 \nu(dz)ds\right)^{\frac{p}{2}}\right]+\mathbb{E}\left[\int_0^T\int_{\mathbb{R}_0} \vert g(s,z) \vert^p \nu(dz)ds\right]\right).
	\end{align*}
	Moreover, for any $1\leq p<2$, there exists a constant $C_p >0$ such that
	\begin{align*}
		\mathbb{E}\left[\sup_{0\leq t \leq T}\left\vert \int_0^t\int_{\mathbb{R}_0} g(s,z) \widetilde{N}(d s, d z)\right\vert^p\right]\leq C_p \, \mathbb{E}\left[\left(\int_0^T\int_{\mathbb{R}_0} \vert g(s,z) \vert^2 \nu(dz)ds\right)^{\frac{p}{2}}\right].
	\end{align*}	
\end{Lem}

\section{Strong convergence rate of Euler-Maruyama approximations in temporal-spatial H\"older norms}
\label{mainresult}
Throughout this section, we always assume that $\mu, \sigma, \gamma \in C^2(\mathbb{R}, \mathbb{R})$. Then, it follows from Lemma 3.3 in \cite{HN22} that there exist positive constants $b$ and $c$,   such that  
\begin{enumerate}
\item[\bf A1.] $\max_{\xi \in \{\mu, \sigma, \gamma\}} \left|(\xi(x)-\xi(y))-(\xi(\widetilde{x})-\xi(\widetilde{y}))\right|\leq c|(x-y)- (\widetilde{x}- \widetilde{y})|+ b \dfrac{|x-y|+|\widetilde{x}-\widetilde{y}|}{2}|x-\widetilde{x}|.$
\end{enumerate} 
% In the following, to easy our notation, we will suppose that $b,c \in [1; +\infty)$. 
%Assume that $p\geq 2$, $c > 1$, $\mu, \sigma, \gamma \in C(\mathbb{R}, \mathbb{R})$ and $V\in C^2(\mathbb{R}, [1, \infty))$. 

Let consider a function $V$ defined by 
$$V(x) =  2^{p} \left[ 1+ \left(|\mu(0)|+ |\sigma(0)|+ |\gamma(0)| \sqrt{m_2} \right)^2 + c^2 x^2 (1+\sqrt{m_2})^2 \right]^{p/2},$$
for some constant $p \geq 2$. 
It is straightforward to see that $V\in C^2(\mathbb{R}, [1, \infty))$ and there exist a positive constant $ \overline{c} $ such that for all $x, y, \widetilde{x}, \widetilde{y} \in \bR$,
\begin{enumerate}
\item[\bf A2.] $|\mu(0)|+|\sigma(0)|+c |x|+\left(|\gamma(0)|+c |x|\right) \sqrt{m_2} \leq \left(V(x)\right)^{1/p}.$
\item[\bf A3.] $|V'(x)| \leq \overline{c} \left(V(x)\right)^{1-1/p}.$
\item[\bf A4.] $|V''(x)| \leq \overline{c} \left(V(x)\right)^{1-2/p}.$
\item[\bf A5.]   
$\int_{\bR_0} \left[ V(y+\gamma(x)z)-V(y)-V'(y)\gamma(x)z\right] \nu (dz)  \le \dfrac{1}{2} \overline{c}  \left(V(x)+ V(y)\right).$
\end{enumerate} 
\begin{Rem} %\label{Lip}
It follows from {\bf A1} that for all $x,y \in \mathbb{R}$, 
$$|\mu(x)-\mu(y)| \vee |\sigma(x) - \sigma(y)| \vee |\gamma(x) - \gamma(y)| \leq c|x-y|,$$
and 
$$|\mu(x)|\leq |\mu(0)|+c|x|\leq \left(V(x)\right)^{1/p}, \quad  |\sigma(x)| \leq |\sigma(0)|+c|x|\leq \left(V(x)\right)^{1/p}, \quad |\gamma(x)|\sqrt{m_2} \leq \left(V(x)\right)^{1/p}.$$ 
\label{Rem1}
\end{Rem}

%\textcolor{black}{\it Definition of Euler-Maruyama approximations of the solution to equation \eqref{eqn1}:} 

Let $\iota :[0, T]\to [0, T]$ satisfy that $\iota(t)=t$, for all $t\in [0, T]$.  Let $\widetilde{\mathbb{S}} = \mathbb{S}\cup\{\iota \}$. We suppose that  $|\iota| =0$. 

\begin{Lem}
\label{Lem3.1} 
 For $\delta \in \widetilde{\mathbb{S}},$ $s\in [0,T],$ $t\in [s, T],$ $\widetilde{t} \in [t, T],$ $x, \widetilde{x}\in\mathbb{R}$, let $(X^{\delta,x}_{s,t})_{t\in[s,T]}:[s,T]  \times\Omega\to\mathbb{R} $ 
be an $(\mf_t)_{t\in[s,T]}$-adapted stochastic process with c\`adl\`ag  sample paths satisfying equation \eqref{y01}. Let $(a_r, b_r, g_r) = \theta (\psi_r)-\theta(\widetilde{\psi}_r)$ with $\theta = (\mu, \sigma, \gamma)$, $r\in [t, \widetilde{t}]$ and 
 \begin{align*}
     (\psi_r, \widetilde{\psi}_r) = \Big\{(X_r, Y_r), (X_{s,\max\{s,\delta(r)\}}^{\delta,x}, X_{s,\max\{s,\delta(r)\}}^{\delta,\widetilde{x}}), (X^{\iota ,x}_{s,r}, X_{s,\max\{s,\delta(r)\}}^{\delta,x})\Big\}.
     \end{align*}
     Then, for all $p\ge 2,$ there exists a positive constant $C_p=C_p(T, m_2, m_p)$ such that 
 \begin{align*}
 %\label{a_r}
 &\left[\mathbb{E} \left[\left| \int_t^{\widetilde{t}} a_r dr \right|^p \right]\right]^{1/p} + \left[\mathbb{E} \left[\left| \int_t^{\widetilde{t}} b_r dW_r \right|^p \right]\right]^{1/p} + \left[\mathbb{E}\left[ \left| \int_t^{\widetilde{t}} \int _{\bR_0} g_{r-} z \widetilde{N}(dr, dz) \right|^p \right]\right]^{1/p} \notag  \\
  & \leq C_p |\widetilde{t} -t|^{1/p} \sup_{r\in [t, \widetilde{t}]} \max \left\{ \left [\mathbb{E} \left[\left| a_r \right|^p\right]\right]^{1/p}, \left [\mathbb{E} \left[\left| b_r \right|^p\right]\right]^{1/p}, \left [\mathbb{E} \left[\left| g_r \right|^p\right]\right]^{1/p} \right\}.
  \end{align*}
  % \textcolor{red}{Hang so $C_p$ o dau ra ???}
\end{Lem}
 \begin{proof} 
 %\textcolor{red}{Su dung phuong trinh vi phan ngau nhien nao ?}
 Using H\"older's inequality and Burkholder-Davis-Gundy inequalities of the It\^o stochastic integral with respect to Brownian motion and compensated Poisson random measure (see Lemma \ref{BDGjump}), we get
\begin{align*} 
%\label{a_r}
 &\left[\mathbb{E} \left[\left| \int_t^{\widetilde{t}} a_r dr \right|^p \right]\right]^{1/p} + \left[\mathbb{E} \left[\left| \int_t^{\widetilde{t}} b_r dW_r \right|^p \right]\right]^{1/p} + \left[\mathbb{E}\left[ \left| \int_t^{\widetilde{t}} \int _{\bR_0} g_{r-} z \widetilde{N}(dr, dz) \right|^p \right]\right]^{1/p} \notag  \\
 & \leq \left(|\widetilde{t} -t|^{p-1} \int_t^{\widetilde{t}} \mathbb{E} [|a_r|^p] dr\right)^{1/p}  + \left( C_p |\widetilde{t} -t|^{\frac{p}{2}-1} \int_t^{\widetilde{t}} \mathbb{E} [|b_r|^p] dr\right)^{1/p}  \notag  \\
 & \quad + \left( C_p \left[|\widetilde{t} -t|^{\frac{p}{2}-1} m_2^{p/2} +m_p \right] \int_t^{\widetilde{t}} \mathbb{E} [|g_r|^p] dr\right)^{1/p}  \notag  \\
 & \leq |\widetilde{t} -t|\sup_{r\in [t, \widetilde{t}]} \left [\mathbb{E} \left[\left| a_r \right|^p\right]\right]^{1/p} + C_p|\widetilde{t} -t|^{1/2} \sup_{r\in [t, \widetilde{t}]} \left [\mathbb{E} \left[\left| b_r \right|^p\right]\right]^{1/p} \notag\\
 & \quad + C_p (|\widetilde{t}-t|^{p/2 -1} m_2^{p/2} + m_p)^{1/p}|\widetilde{t} -t|^{1/p} \sup_{r\in [t, \widetilde{t}]} \left [\mathbb{E} \left[\left| g_r \right|^p\right]\right]^{1/p} \notag \\
 & = |\widetilde{t} -t|^{1/p} \left\{ |\widetilde{t} -t|^{1-\frac{1}{p}} \sup_{r\in [t, \widetilde{t}]} \left [\mathbb{E} \left[\left| a_r \right|^p\right]\right]^{1/p} + C_p|\widetilde{t} -t|^{\frac{1}{2}-\frac{1}{p}} \sup_{r\in [t, \widetilde{t}]} \left [\mathbb{E} \left[\left| b_r \right|^p\right]\right]^{1/p} \right.\notag\\
 & \quad \left. + C_p (|\widetilde{t}-t|^{p/2 -1} m_2^{p/2} + m_p)^{1/p} \sup_{r\in [t, \widetilde{t}]} \left [\mathbb{E} \left[\left| g_r \right|^p\right]\right]^{1/p}\right\} \notag\\
& \leq |\widetilde{t} -t|^{1/p} \left\{ T^{1-\frac{1}{p}} \sup_{r\in [t, \widetilde{t}]} \left [\mathbb{E} \left[\left| a_r \right|^p\right]\right]^{1/p} + C_p T^{\frac{1}{2}-\frac{1}{p}} \sup_{r\in [t, \widetilde{t}]} \left [\mathbb{E} \left[\left| b_r \right|^p\right]\right]^{1/p} \right.\notag\\
 & \quad \left. + C_p (T^{p/2 -1} m_2^{p/2} + m_p)^{1/p} \sup_{r\in [t, \widetilde{t}]} \left [\mathbb{E} \left[\left| g_r \right|^p\right]\right]^{1/p}\right\} \notag\\
 & \leq C_p |\widetilde{t} -t|^{1/p} \sup_{r\in [t, \widetilde{t}]} \max \left\{ \left [\mathbb{E} \left[\left| a_r \right|^p\right]\right]^{1/p}, \left [\mathbb{E} \left[\left| b_r \right|^p\right]\right]^{1/p}, \left [\mathbb{E} \left[\left|g_r \right|^p\right]\right]^{1/p} \right\},
\end{align*}
for a positive constant $C_p=C_p(T, m_2, m_p)$. Thus, the result follows.
\end{proof}
\begin{Lem}
\label{Lem3.2}
For $\delta \in \widetilde{\mathbb{S}},$ $s, \widetilde{s} \in [0,T],$ $t\in [s, T],$  $x \in\mathbb{R}$, let $(X^{\delta,x}_{s,t})_{t\in[s,T]}:[s,T]  \times\Omega\to\mathbb{R} $ 
be an $(\mf_t)_{t\in[s,T]}$-adapted stochastic process with c\`adl\`ag  sample paths satisfying equation \eqref{y01}. Then, for all $\mathbb{X}=(\mathbb{X}_r)_{r\in [s, t]}, \mathbb{Y}=(\mathbb{Y}_r)_{r\in [s, t]} \in \{ (x)_{r\in [s, t]}, (X_{\widetilde{s},\max\{\widetilde{s},\delta(r)\}}^{\delta,x})_{r\in [s, t]}, (X_{\widetilde{s},\max\{\widetilde{s},\delta (r) \}}^{\iota,x})_{r\in [s, t]}  \}$ and $p\ge 2,$ there exists a positive constant $C_p=C_p(T,m_2,m_p)$ such that
\begin{align*}
&\left[\mathbb{E} \left[\left| \int_s^t (\mu(\mathbb{X}_r)-\mu(\mathbb{Y}_r)) dr \right|^p \right]\right]^{1/p} + \left[\mathbb{E} \left[\left| \int_s^t (\sigma(\mathbb{X}_r)-\sigma(\mathbb{Y}_r)) dW_r \right|^p \right]\right]^{1/p} \notag \\ 
& \quad + \left[\mathbb{E} \left[ \left| \int_s^t \int_{\bR_0} (\gamma(\mathbb{X}_{r-})-\gamma(\mathbb{Y}_{r-})) z \widetilde{N}(dr, dz) \right|^p \right]\right]^{1/p} \notag \\
 & \leq c  C_p  |t-s|^{1/p} \sup_{r\in [s, t]} \left[ \mathbb{E} \left[\left| \mathbb{X}_r-\mathbb{Y}_r \right|^p \right]\right]^{1/p},
\end{align*}
where the constant $c$ is given in {\bf A1}.
\end{Lem}

\begin{proof}
    Using H\"older's inequality, Burkholder-Davis-Gundy inequalities of the It\^o stochastic integral with respect to Brownian motion and compensated Poisson random measure (see Lemma \ref{BDGjump}), and the Lipschitz property of the coefficients $\mu$, $\sigma$, $\gamma$ (see  Remark $\ref{Rem1}$), we obtain
\begin{align*} 
%\label{X_r, Y_r}
 &\left[\mathbb{E} \left[\left| \int_s^t (\mu(\mathbb{X}_r)-\mu(\mathbb{Y}_r)) dr \right|^p \right]\right]^{1/p} + \left[\mathbb{E} \left[\left| \int_s^t (\sigma(\mathbb{X}_r)-\sigma(\mathbb{Y}_r)) dW_r \right|^p \right]\right]^{1/p}  \\
 & \quad + \left[\mathbb{E} \left[ \left| \int_s^t \int_{\bR_0} (\gamma(\mathbb{X}_{r-})-\gamma(\mathbb{Y}_{r-})) z \widetilde{N}(dr, dz) \right|^p \right]\right]^{1/p} \notag \\
 %& \leq \left[\sqrt{T} + 1+ \sqrt{m_2} \right] \left[ \int_s^t \max_{\xi \in \{ \mu, \sigma, \gamma\}} \mathbb{E} \left[\left( \xi(\mathbb{X}_r)-\xi(\mathbb{Y}_r)  \right)^2 \right] dr \right]^{1/2} \notag \\
 %& \leq c \left\{ \left[\mathbb{E} \left[\left| \int_s^t (\mathbb{X}_r-\mathbb{Y}_r) dr \right|^p \right]\right]^{1/p} + \left[\mathbb{E} \left[\left| \int_s^t (\mathbb{X}_r-\mathbb{Y}_r) dW_r \right|^p \right]\right]^{1/p} \right.  \notag \\
% & \quad \left. + \left[\mathbb{E} \left[ \left| \int_s^t \int_{\bR_0} (\mathbb{X}_r-\mathbb{Y}_r) z \widetilde{N}(dr, dz) \right|^p \right]\right]^{1/p} \right \} \notag \\
& \le \left( |t-s|^{p-1} \int_s^t \mathbb{E} \left[ \left| \mu(\mathbb{X}_r)-\mu(\mathbb{Y}_r) \right|^p\right] dr \right)^{1/p} + \left( C_{p} |t-s|^{\frac{p}{2}-1} \int_s^t \mathbb{E} \left[ \left| \sigma(\mathbb{X}_r)-\sigma(\mathbb{Y}_r) \right|^p\right] dr \right)^{1/p} \notag\\
& \quad  + \left( C_{p} (|t-s|^{\frac{p}{2}-1}m_2^{\frac{p}{2}}+m_p ) \int_s^t \mathbb{E} \left[ \left| \gamma( \mathbb{X}_r) - \gamma (\mathbb{Y}_r) \right|^p\right] dr \right)^{1/p} \notag\\
& \le c \left( |t-s|^{p-1} \int_s^t \mathbb{E} \left[ \left| \mathbb{X}_r-\mathbb{Y}_r \right|^p\right] dr \right)^{1/p} + c \left( C_{p} |t-s|^{\frac{p}{2}-1} \int_s^t \mathbb{E} \left[ \left| \mathbb{X}_r-\mathbb{Y}_r \right|^p\right] dr \right)^{1/p} \notag\\
& \quad  + c \left( C_{p} (|t-s|^{\frac{p}{2}-1}m_2^{\frac{p}{2}}+m_p ) \int_s^t \mathbb{E} \left[ \left|  \mathbb{X}_r - \mathbb{Y}_r \right|^p\right] dr \right)^{1/p} \notag\\
 & \le c \left\{ |t-s|\sup_{r\in [s, t]} \left[ \mathbb{E} \left[ \left| \mathbb{X}_r-\mathbb{Y}_r \right|^p \right]\right]^{1/p} + C_{p}^{1/p}|t-s|^{1/2} \sup_{r\in [s, t]}  \left[ \mathbb{E} \left[ \left| \mathbb{X}_r-\mathbb{Y}_r \right|^p \right]\right]^{1/p} \right. \notag \\
 & \quad \left. +  C_{p}^{1/p} (|t-s|^{\frac{p}{2}-1}m_2^{\frac{p}{2}}+m_p )^{1/p}|t-s|^{1/p} \sup_{r\in [s, t]} \left[ \mathbb{E} \left[ \left| \mathbb{X}_r-\mathbb{Y}_r \right|^p \right]\right]^{1/p} \right\} \notag\\
  &= c |t-s|^{1/p} \left\{ |t-s|^{1-\frac{1}{p}}\sup_{r\in [s, t]} \left[ \mathbb{E} \left[ \left| \mathbb{X}_r-\mathbb{Y}_r \right|^p \right]\right]^{1/p} + C_{p}^{1/p}|t-s|^{\frac{1}{2}-\frac{1}{p}} \sup_{r\in [s, t]}  \left[ \mathbb{E} \left[ \left| \mathbb{X}_r-\mathbb{Y}_r \right|^p \right]\right]^{1/p} \right. \notag \\
 & \quad \left. +   C_{p}^{1/p} (|t-s|^{\frac{p}{2}-1}m_2^{\frac{p}{2}}+m_p )^{1/p} \sup_{r\in [s, t]} \left[ \mathbb{E} \left[ \left| \mathbb{X}_r-\mathbb{Y}_r \right|^p \right]\right]^{1/p} \right\} \notag\\
  &\leq c |t-s|^{1/p} \left\{ T^{1-\frac{1}{p}}\sup_{r\in [s, t]} \left[ \mathbb{E} \left[ \left| \mathbb{X}_r-\mathbb{Y}_r \right|^p \right]\right]^{1/p} + C_{p}^{1/p} T^{\frac{1}{2}-\frac{1}{p}} \sup_{r\in [s, t]}  \left[ \mathbb{E} \left[ \left| \mathbb{X}_r-\mathbb{Y}_r \right|^p \right]\right]^{1/p} \right. \notag \\
 & \quad \left. +   C_{p}^{1/p} (T^{\frac{p}{2}-1}m_2^{\frac{p}{2}}+m_p )^{1/p} \sup_{r\in [s, t]} \left[ \mathbb{E} \left[ \left| \mathbb{X}_r-\mathbb{Y}_r \right|^p \right]\right]^{1/p} \right\} \notag\\
& \le c  C_p  |t-s|^{1/p} \sup_{r\in [s, t]} \left[ \mathbb{E} \left[ \left| \mathbb{X}_r-\mathbb{Y}_r \right|^p \right]\right]^{1/p},
\end{align*}
for a positive constant $C_p=C_p(T, m_2, m_p)$. This finishes the proof.
\end{proof}

\begin{Lem}
\label{Lem3.3}
For $s\in [0, T],$ $t, \widetilde{t} \in [s, T],$ $ \delta \in \widetilde{\mathbb{S}}$ and $ x\in \bR,$ let $(X^{\delta,x}_{s,t})_{t\in[s,T]}:[s,T]  \times\Omega\to\mathbb{R} $ 
be an $(\mf_t)_{t\in[s,T]}$-adapted stochastic process with c\`adl\`ag  sample paths satisfying equation \eqref{y01}. 
%Let $V\in C^2(\mathbb{R}, [1, \infty))$ satisfy statements {\bf A2-A5}. 
Then, 
\begin{enumerate}
   \item [\textup{(i)}] we have
\begin{align*}
\mathbb{E}\left[V(X^{\delta, x}_{s,t})\right] \leq e^{2.5 \overline{c}|t-s|}V(x).
\end{align*} 
% \begin{align*}
% \mathbb{E}\left[(V(X^{i, x}_{s,t}))^n\right] \leq e^{2.5 \overline{c}|t-s|}V(x).
% \end{align*} 
\item [\textup{(ii)}] for all $p\ge 2,$ there exists a positive constant $C_p=C_p(T, m_2, m_p)$ such that \begin{align*}
\left[\mathbb{E}\left[\left|X^{\delta, x}_{s,\widetilde{t}}-X^{\delta, x}_{s,t}\right|^p\right]\right]^{1/p} 
 \leq C_p |\widetilde{t} -t|^{1/p} (e^{2.5\overline{c}T}V(x))^{1/p}.
\end{align*} 
\end{enumerate}
Here, the constant $\overline{c}$ is given in {\bf A3, A4, A5}. 
 \end{Lem}
\begin{proof}
 $\textup{(i)}$ Using $\bf A3, A4, A5$ with $y=x$ and Remark $\ref{Rem1}$, we have
\begin{align*}
&V'(x)\mu (x) + \dfrac{1}{2}V''(x) \sigma^2(x) + \int_{\bR_0}\left[V(x+\gamma(x)z)-V(x)-V'(x)\gamma(x)z\right] \nu (dz)    \notag \\
& \leq \overline{c}\left(V(x)\right)^{1-\frac{1}{p}} \left(V(x)\right)^{\frac{1}{p}} + \dfrac{1}{2} \overline{c} \left(V(x)\right)^{1-\frac{2}{p}} \left(V(x)\right)^{\frac{2}{p}} + \overline{c} V(x) \\
%& \leq \overline{c} V(x) +\dfrac{1}{2} \overline{c} V(x)+ \dfrac{1}{2} \gamma^2(x) \overline{c} \int_{\bR_0}    \left(V(x+\theta(x, \gamma(x)z) \gamma(x)z)\right)^{1-\frac{2}{p}} z^2 \nu(dz) \notag \\
%& ??? \leq \overline{c} V(x) +\dfrac{1}{2} \overline{c} V(x)+\dfrac{1}{2} \gamma^2(x) \overline{c} \left(V(x)\right)^{1-\frac{2}{p}} \int_{\bR_0}  z^2 \nu(dz) \notag \\
& = \overline{c} V(x) +\dfrac{1}{2} \overline{c} V(x)+ \overline{c} V(x) = 2.5 \overline{c} V(x).
\end{align*}
 Then, applying Lemma $\ref{Lem2.2}$ with{ \color{black} $\alpha = 2.5 \overline{c} $, $X_{\tau}=X^{\iota, x}_{s, t}$, $\tau=t$ and $V(t,x)=V(x)$,} we obtain
\begin{align}
    \mathbb{E}\left[V(X^{\iota, x}_{s, t})\right] \leq e^{2.5 \overline{c}(t-s)}V(x).
    \label{i1}
\end{align}
% We consider another process $X_{t}=x+ \mu (x)t + \sigma (x)W_t+\gamma(x)Z_t$. With this process, we have 
Next, using again $\bf A3, A4, A5,$ Remark $\ref{Rem1},$ and the inequality $a^{\lambda} b^{1-\lambda} \le \lambda a+(1-\lambda)b$, valid for all $a, b \in [0, \infty)$, $\lambda \in (0,1)$, we have  
% \begin{cases} 
% & \xi_t=x+\mu(x)t+\sigma(x)W_t+c(t)Z_t,\\ 
% &  \widehat{\mu}(t)=\mu(x) \\
% & \widehat{\sigma}(t) = \sigma(x)\\
% & \widehat{\gamma}(t)= \gamma(x), 
% \end{cases} we have
\begin{align*}
   & V'(y) \mu(x)+\dfrac{1}{2} V''(y) \sigma^2 (x) + \int_{\bR_0} \left[ V(y+\gamma(x)z)-V(y)-V'(y)\gamma(x)z\right] \nu (dz) \notag \\
   & \le |V'(y)| \, |\mu(x)|+\dfrac{1}{2} \left|V''(y) \right|\sigma^2 (x) +\dfrac{1}{2}  \overline{c} (V(x) + V(y)) \notag \\
   %& \leq \overline{c} \left(V(y)\right)^{1-\frac{1}{p}} |\mu (x)| +\dfrac{1}{2} \overline{c} \left(V(y)\right)^{1-\frac{2}{p}}  \sigma^2 (x) +\dfrac{1}{2} \overline{c}(V(x) + V(y)) \notag \\
&  \leq \overline{c} \left(V(y)\right)^{1-\frac{1}{p}} \left(V(x)\right)^{\frac{1}{p}}  +\dfrac{1}{2} \overline{c} \left(V(y)\right)^{1-\frac{2}{p}}  \left(V(x)\right)^{\frac{2}{p}}  +\dfrac{1}{2}  \overline{c} (V(x) + V(y)) \notag \\
&  \leq \overline{c}  \left[ \left(1-\dfrac{1}{p} \right) V(y)+ \dfrac{1}{p} V(x)\right] +\dfrac{1}{2} \overline{c}  \left[ \left(1-\dfrac{2}{p} \right) V(y)+ \dfrac{2}{p} V(x)\right] +\dfrac{1}{2}  \overline{c} (V(x) + V(y)) \notag \\
&=  \left(2  \overline{c} -\dfrac{2\overline{c}}{p}\right) V(y)+ \left(\dfrac{2\overline{c}}{p} +\dfrac{1}{2}\overline{c} \right)V(x).
    \end{align*}
Then for $\xi_s=x+\mu(x)s+\sigma(x)W_s+\gamma(x)Z_s,$ where $ \widehat{\mu}(\xi_s):=\mu(x),$ $\widehat{\sigma}(\xi_s) := \sigma(x),$ $ \widehat{\gamma}(\xi_s) := \gamma(x)$, we have 
\begin{align}
& V'(\xi_s) \widehat{\mu}(\xi_s)+\dfrac{1}{2} V''(\xi_s) \widehat{\sigma}^2 (\xi_s) + \int_{\bR_0} \left[ V(\xi_s +\widehat{\gamma}(\xi_s)z)-V(\xi_s)-V'(\xi_s) \widehat{\gamma}(\xi_s)z\right] \nu (dz) \notag \\
  & = V'(\xi_s) \mu(x)+\dfrac{1}{2} V''(\xi_s) \sigma^2 (x) + \int_{\bR_0} \left[ V(\xi_s +\gamma(x)z)-V(\xi_s)-V'(\xi_s)\gamma(x)z\right] \nu (dz) \notag \\
 &\le  \left(2  \overline{c} -\dfrac{2\overline{c}}{p}\right) V(\xi_s)+ \left(\dfrac{2\overline{c}}{p} +\dfrac{1}{2}\overline{c} \right)V(x). 
\end{align}  
Therefore, applying Lemma $\ref{Lem2.4}$ for  $\xi_s$, and the inequality $1+a \le e^a,$ valid for all $a \in [0, \infty),$ we obtain
\begin{align}
  \mathbb{E}\left[ V(x+ \mu (x)s + \sigma (x)W_s+\gamma(x)Z_s)\right] 
 &= \mathbb{E}\left[ V (\xi_s) \right] \notag \\ 
 &\leq e^{\left(2 \overline{c} -\frac{2\overline{c}}{p}\right)s} \left( V(x)+ \int_0^s \dfrac{\left(\frac{2\overline{c}}{p} +\frac{1}{2}\overline{c} \right)V(x)}{ \exp \int_0^u \left(2 \overline{c} -\frac{2\overline{c}}{p}\right)dr}du \right) \notag \\
 & \leq e^{\left(2 \overline{c} -\frac{2\overline{c}}{p}\right)s} V(x) \left( 1+\left(\frac{2\overline{c}}{p} +\frac{1}{2} \overline{c} \right)s \right) \notag\\
 & \leq e^{\left(2 \overline{c} -\frac{2\overline{c}}{p}\right)s} e^{\left(\frac{2\overline{c}}{p} +\frac{1}{2} \overline{c} \right)s} V(x)= e^{2.5 \overline{c}s} V(x),
 \label{i2}
\end{align}
where we have used the fact that $\exp \int_0^u \left(2 \overline{c} -\frac{2\overline{c}}{p}\right)dr \ge 1.$

For the rest of this paper, we will repeatedly  use the tower property of conditional expectation and the disintegration theorem (see \cite[Lemma 2.2]{HJKNW20}) as follows
\begin{align*}
% \label{i3}
   &\mathbb{E} \left[ V(X_{s,t}^{\delta,x})\right] =\mathbb{E} \left[\mathbb{E} \left[V(X_{s,t}^{\delta,x})\big| \mathcal{F}_{\max\{s,\delta(t)\}}\right] \right] \notag \\ 
&=\mathbb{E}\biggl[\mathbb{E} \Bigl[V\Bigl(z+
\mu(z)(t- \max\{s,\delta(t)\})+
\sigma(z)(W_{t} -W_{\max\{s,\delta(t)\}} ) \notag \\
&\qquad  +
\gamma (z)(Z_{t} -Z_{ \max\{s,\delta(t)\}} )\Bigr)\Bigr|_{z=X_{s,\max\{s,\delta(t)\}}^{\delta,x} } \Big|\mathcal{F}_{\max\{s,\delta(t)\}}\Bigr]\biggr]\\
&=\mathbb{E}\left[\mathbb{E} \Bigl[V \Bigl(z+
\mu(z)(t- \max\{s,\delta(t)\})+
\sigma(z)W_{t- \max\{s,\delta(t)\}} + \gamma (z)Z_{t- \max\{s,\delta(t)\}}  \Bigr) \Bigr]\Bigr|_{ z=X_{s,\max\{s,\delta(t)\}}^{\delta,x} }\right].
\end{align*}
Next, using $\eqref{i2}$ and $\eqref{i1}$, we get
%Consequently, applying the tower property of conditional expectation, equation $\eqref{y01}$, and the stationary increments of the Brownian motion and L\'evy process, together with  $\eqref{i2}$ and $\eqref{i1}$, we get
\begin{align*}
% \label{i3}
   \mathbb{E} \left[ V(X_{s,t}^{\delta,x})\right] %=\mathbb{E} \left[\mathbb{E} \left[V(X_{s,t}^{\delta,x})\big| \mathcal{F}_{\max\{s,\delta(t)\}}\right] \right] \notag \\ 
%&=\mathbb{E}\biggl[\mathbb{E} \Bigl[V\Bigl(z+\mu(z)(t- \max\{s,\delta(t)\})+ \sigma(z)(W_{t} -W_{\max\{s,\delta(t)\}} ) \notag \\ &\qquad  + \gamma (z)(Z_{t} -Z_{ \max\{s,\delta(t)\}} )\Bigr)\Bigr|_{z=X_{s,\max\{s,\delta(t)\}}^{\delta,x} } \Big|\mathcal{F}_{\max\{s,\delta(t)\}}\Bigr]\biggr] \notag \\
%&=\mathbb{E}\left[\mathbb{E} \Bigl[V \Bigl(z+\mu(z)(t- \max\{s,\delta(t)\})+\sigma(z)W_{t- \max\{s,\delta(t)\}} + \gamma (z)Z_{t- \max\{s,\delta(t)\}}  \Bigr) \Bigr]\Bigr|_{ z=X_{s,\max\{s,\delta(t)\}}^{\delta,x} }\right] \notag \\
&\leq e^{2.5 \overline{c} (t-\max\{s,\delta(t)\})}  \mathbb{E} \left[ V\bigl(X_{s,\max\{s,\delta(t)\}}^{\delta,x} \bigr)\right] \notag \\
& \leq e^{2.5\overline{c} (t-\max\{s,\delta(t)\})}  e^{2.5 \overline{c} (\max\{s,\delta(t)\}-s)}  V(x) \notag \\
& = e^{2.5 \overline{c}(t-s)}V(x).
\end{align*}
% Using the first statement of $\bf A5$ with $y=x$ and (i), (ii) of Remark $\ref{Rem1}$, we have
% \begin{align*}
% &((V(x))^n)'\mu (x) + \dfrac{1}{2} ((V(x))^n)'' \sigma^2(x) + \int_{\bR_0}\left[(V(x+\gamma(x)z))^n-(V(x))^n-((V(x))^n)'\gamma(x)z\right] \nu (dz)    \notag \\
% & \leq n\overline{c}(V(x))^{n-1}\left(V(x)\right)^{1-\frac{1}{p}} \left(V(x)\right)^{\frac{1}{p}} + \dfrac{1}{2} n ((V(x))^{n-1} \overline{c} (V(x))^{1-\frac{2}{p}}+(n-1) \left(V(x)\right)^{2-\frac{2}{p}} (V(x))^{n-2}) \left(V(x)\right)^{\frac{2}{p}}\notag \\
% & \quad + \overline{c} (V(x))^n \\
% %& \leq \overline{c} V(x) +\dfrac{1}{2} \overline{c} V(x)+ \dfrac{1}{2} \gamma^2(x) \overline{c} \int_{\bR_0}    \left(V(x+\theta(x, \gamma(x)z) \gamma(x)z)\right)^{1-\frac{2}{p}} z^2 \nu(dz) \notag \\
% %& ??? \leq \overline{c} V(x) +\dfrac{1}{2} \overline{c} V(x)+\dfrac{1}{2} \gamma^2(x) \overline{c} \left(V(x)\right)^{1-\frac{2}{p}} \int_{\bR_0}  z^2 \nu(dz) \notag \\
% & = n\overline{c} (V(x))^n +\dfrac{n^2}{2} \overline{c} (V(x))^n+ \overline{c} (V(x))^n = \left( n+\dfrac{n^2}{2}+ 1\right) \overline{c} (V(x))^n.
% \end{align*}
%  Then, applying Lemma $\ref{Lem2.2}$ with{ \color{red} $\alpha = 2.5 \overline{c} $, $X_{\tau}=X^{\iota, x}_{s, t}$, $\tau=t$ and $V(t,x)=V(x)$,} we obtain
% \begin{align}
%     \mathbb{E}\left[V(X^{\iota, x}_{s, t})\right] \leq e^{2.5 \overline{c}(t-s)}V(x).
%     \label{i11}
% \end{align}
This shows $\textup{(i)}.$\\
$\textup{(ii)}$ Using equation $\eqref{y01}$, we write
  \begin{align*}
 &  X^{\delta, x}_{s,\widetilde{t}}-X^{\delta, x}_{s,t} \\
 &= \int_t^{\widetilde{t}} \mu(X_{s,\max\{s,\delta(r)\}}^{\delta,x}) dr + \int_t^{\widetilde{t}} \sigma(X_{s,\max\{s,\delta(r)\}}^{\delta,x}) dW_r + \int_t^{\widetilde{t}} \int _{\bR_0}\gamma (X_{s,\max\{s,\delta(r-)\}}^{\delta,x}) z \widetilde{N}(dr, dz).
  \end{align*} 
  Then, applying the triangle inequality, we get for $p \ge 2,$
\begin{align} 
\label{Lem3.3.1}
\left[\mathbb{E}\left[\left|X^{\delta, x}_{s,\widetilde{t}}-X^{\delta, x}_{s,t}\right|^p\right]\right]^{1/p} 
&\leq \left[\mathbb{E} \left[\left| \int_t^{\widetilde{t}} \mu(X_{s,\max\{s,\delta(r)\}}^{\delta,x}) dr \right|^p \right]\right]^{1/p} + \left[\mathbb{E} \left[\left| \int_t^{\widetilde{t}} \sigma(X_{s,\max\{s,\delta(r)\}}^{\delta,x}) dW_r \right|^p \right]\right]^{1/p} \notag \\
& \quad + \left[\mathbb{E} \left[\left| \int_t^{\widetilde{t}} \int _{\bR_0}\gamma (X_{s,\max\{s,\delta(r-)\}}^{\delta,x}) z \widetilde{N}(dr, dz) \right|^p\right]\right]^{1/p}.
\end{align}
Next, using H\"older inequality and Burkholder-Davis-Gundy inequalities of the It\^o stochastic integral with respect to Brownian motion and compensated Poisson random measure (see Lemma \ref{BDGjump}), and proceeding as in the proof of Lemma $\ref{Lem3.1}$, we get
%applying Lemma $\ref{Lem3.1}$ with $a_r=\mu(X_{s,\max\{s,\delta(r)\}}^{\delta,x}),$ $b_r= \sigma(X_{s,\max\{s,\delta(r)\}}^{\delta,x})$ and $g_r = \gamma (X_{s,\max\{s,\delta(r)\}}^{\delta,x})$, we get 
\begin{align}
\label{s,wt}
&\left[\mathbb{E}\left[\left|X^{\delta, x}_{s,\widetilde{t}}-X^{\delta, x}_{s,t}\right|^p\right]\right]^{1/p} \notag \\
& \leq |\widetilde{t} -t|^{1/p} \left\{ T^{1-\frac{1}{p}} \sup_{r\in [t, \widetilde{t}]}  \left[\mathbb{E} \left[\left| \mu (X_{s,\max\{s,\delta(r)\}}^{\delta,x}) \right|^p \right]\right]^{1/p} + C_p T^{\frac{1}{2}-\frac{1}{p}} \sup_{r\in [t, \widetilde{t}]}  \left[\mathbb{E} \left[\left| \sigma (X_{s,\max\{s,\delta(r)\}}^{\delta,x}) \right|^p \right]\right]^{1/p} \right.\notag\\
 & \quad \left. + C_p (T^{\frac{p}{2} -1} m_2^{p/2} + m_p)^{1/p} \sup_{r\in [t, \widetilde{t}]}  \left[\mathbb{E} \left[\left| \gamma (X_{s,\max\{s,\delta(r)\}}^{\delta,x}) \right|^p \right]\right]^{1/p} \right\}.
\end{align}
Thanks to Remark $\ref{Rem1}$, we have
\begin{align*}
&\left[\mathbb{E}\left[\left|X^{\delta, x}_{s,\widetilde{t}}-X^{\delta, x}_{s,t}\right|^p\right]\right]^{1/p} \notag \\
& \leq |\widetilde{t} -t|^{1/p} \left\{ T^{1-\frac{1}{p}} \sup_{r\in [t, \widetilde{t}]}  \left[\mathbb{E}\left[V(X_{s,\max\{s,\delta(r)\}}^{\delta,x})\right]\right]^{1/p}  + C_p T^{\frac{1}{2}-\frac{1}{p}} \sup_{r\in [t, \widetilde{t}]}  \left[\mathbb{E}\left[V(X_{s,\max\{s,\delta(r)\}}^{\delta,x})\right]\right]^{1/p}  \right.\notag\\
 & \quad \left. + C_p (T^{\frac{p}{2} -1} m_2^{p/2} + m_p)^{1/p} \sup_{r\in [t, \widetilde{t}]}  \left[\mathbb{E} \left[V(X_{s,\max\{s,\delta(r)\}}^{\delta,x})\right]\right]^{1/p} \right\} \notag\\
 %& \leq C_p |\widetilde{t} -t|^{1/p} \sup_{r\in [t, \widetilde{t}]} \left[\mathbb{E} \left[\left| |V(X_{s,\max\{s,\delta(r)\}}^{\delta,x})|^{1/p}  \right|^p \right]\right]^{1/p} \notag\\
& \le C_p |\widetilde{t} -t|^{1/p} \sup_{r\in [t, \widetilde{t}]} \left[\mathbb{E} \left[ V(X_{s,\max\{s,\delta(r)\}}^{\delta,x}) \right] \right]^{1/p},
\end{align*} 
for a positive constant $C_p=C_p(T,m_2, m_p)$.

Consequently, applying the result $\textup{(i)}$ above, we obtain that
\begin{align}
\label{Lem3.3.4}
\left[\mathbb{E}\left[\left|X^{\delta, x}_{s,\widetilde{t}}-X^{\delta, x}_{s,t}\right|^p\right]\right]^{1/p} 
&\leq C_p |\widetilde{t} -t|^{1/p} (e^{2.5 \overline{c}T}V(x))^{1/p},
%& = C_p |\widetilde{t} -t|^{1/p} e^{\overline{c}T} \sqrt{V(x)}.
% & \leq \left[\sqrt{T} + 1+ \sqrt{m_2} \right] |\widetilde{t} -t|^{1/2} e^{\overline{c}T} \sqrt{V(x)} \notag  \\
% & \leq \left[\sqrt{T} + 1+ \sqrt{m_2} \right] |\delta|^{1/2} e^{\overline{c}T} \sqrt{V(x)}.
\end{align} 
for a positive constant $C_p=C_p(T,m_2,m_p)$. This finishes the desired proof. 
\end{proof}

\begin{Lem}
\label{Lem3.4}
 For $s\in [0, T], t \in [s, T],$ $\delta \in \widetilde{\mathbb{S}}$ and $ x, \widetilde{x} \in \bR,$ let $(X^{\delta,x}_{s,t})_{t\in[s,T]}:[s,T]  \times\Omega\to\mathbb{R} $ 
be an $(\mf_t)_{t\in[s,T]}$-adapted stochastic process with c\`adl\`ag  sample paths satisfying equation \eqref{y01}. Then, for all $p\ge 2,$ there exists a positive constant $C_p$ such that \begin{align*}
\left[\mathbb{E}\left[\left|X^{\delta, x}_{s,t}-X^{\delta, \widetilde{x}}_{s,t}\right|^p\right]\right]^{1/p} \leq C_p |x-\widetilde{x}|.  
\end{align*}
%where the constant $c$ is given in {\bf A2} and the constant $C_p$ is given in Lemma $\ref{Lem3.1}$.
\end{Lem}
\begin{proof}
  Using the definition of $X^{\delta, x}_{s,t}, X^{\delta, \widetilde{x}}_{s,t}$, the triangle inequality and proceeding as in the proof of Lemma $\ref{Lem3.2}$, we obtain  
\begin{align*}
\left[\mathbb{E}\left[\left|X^{\delta, x}_{s,t}-X^{\delta, \widetilde{x}}_{s,t}\right|^p\right]\right]^{1/p} & \leq |x-\widetilde{x}| + \left[\mathbb{E} \left[\left| \int_s^t 
 \left (\mu(X_{s,\max\{s,\delta(r)\}}^{\delta,x}) -\mu(X_{s,\max\{s,\delta(r)\}}^{\delta,\widetilde{x}}) \right) dr \right|^p \right]\right]^{1/p} \notag\\
& \quad + \left[\mathbb{E} \left[\left| \int_s^t \left( \sigma(X_{s,\max\{s,\delta(r)\}}^{\delta,x}) -\sigma(X_{s,\max\{s,\delta(r)\}}^{\delta,\widetilde{x}})\right) dW_r \right|^p \right]\right]^{1/p}\notag \\
& \quad + \left[\mathbb{E} \left[\left| \int_s^t \int_{\bR_0} \left(\gamma (X_{s,\max\{s,\delta(r-)\}}^{\delta,x})-\gamma (X_{s,\max\{s,\delta(r-)\}}^{\delta,\widetilde{x}}) \right) z \widetilde{N}(dr, dz)  \right|^p \right]\right]^{1/p} \notag \\
& \leq |x-\widetilde{x}| +  c \left( |t-s|^{p-1} \int_s^t \mathbb{E} \left[ \left| X_{s,\max\{s,\delta(r)\}}^{\delta,x} -X_{s,\max\{s,\delta(r)\}}^{\delta,\widetilde{x}}
\right|^p\right] dr \right)^{1/p} \\
&\quad + c \left( C_{p} |t-s|^{\frac{p}{2}-1} \int_s^t \mathbb{E} \left[ \left| X_{s,\max\{s,\delta(r)\}}^{\delta,x} -X_{s,\max\{s,\delta(r)\}}^{\delta,\widetilde{x}} \right|^p\right] dr \right)^{1/p} \notag\\
& \quad  + c \left( C_{p} (|t-s|^{\frac{p}{2}-1}m_2^{\frac{p}{2}}+m_p )^{1/p} \int_s^t \mathbb{E} \left[ \left|  X_{s,\max\{s,\delta(r)\}}^{\delta,x} -X_{s,\max\{s,\delta(r)\}}^{\delta,\widetilde{x}} \right|^p\right] dr \right)^{1/p} \notag\\
& \leq |x-\widetilde{x}| + c C_p \left[ \int_s^t \mathbb{E} \left[\left| X_{s,\max\{s,\delta(r)\}}^{\delta,x} -X_{s,\max\{s,\delta(r)\}}^{\delta,\widetilde{x}}  \right|^p\right] dr \right]^{1/p},
\end{align*}
for a positive constant $C_p:= C_p (T, m_2, m_p).$

Then, applying Lemma $\ref{Cor2.1}$ for $x(t)= \left[\mathbb{E}\left[\left|X^{\delta, x}_{s,t}-X^{\delta, \widetilde{x}}_{s,t}\right|^p\right]\right]^{1/p} $, $a(t)=|x-\widetilde{x}|$ and $c_{\star}=c C_p,$ we have
\begin{align*}
\left[\mathbb{E}\left[\left|X^{\delta, x}_{s,t}-X^{\delta, \widetilde{x}}_{s,t}\right|^p\right]\right]^{1/p} \le 2^{1-\frac{1}{p}} |x-\widetilde{x}| \exp \left(\dfrac{2^{p-1} (c C_p)^p (t-s)}{p}\right) \leq C_p |x-\widetilde{x}| , 
\end{align*}
for a positive constant $C_p:= C_p (c, T, m_2, m_p).$ This finishes the desired proof.
\end{proof}

\begin{Lem}
\label{Lem3.5}
For $\delta \in \widetilde{\mathbb{S}},$ $x \in \bR,$ $s\in [0, T],$ $\widetilde{s} \in [s, T], \widehat{s} \in [\widetilde{s},T],$ $\delta([0, T])\cap (s, \widehat{s}) = \emptyset,$ let $(X^{\delta,x}_{s,t})_{t\in[s,T]}:[s,T]  \times\Omega\to\mathbb{R} $ 
be an $(\mf_t)_{t\in[s,T]}$-adapted stochastic process with c\`adl\`ag  sample paths satisfying equation \eqref{y01}. 
%Let $V\in C^2(\mathbb{R}, [1, \infty))$ satisfy statements {\bf A2-A5}. 
Then, for all $p \ge 2$, there exists a positive constant $C_p$ such that 
\begin{enumerate}
     \item [\textup{(i)}]
    \begin{align*}
        \left[\mathbb{E}\left|X^{\delta, x}_{\widetilde{s},\widehat{s}}-X^{\delta, x}_{s,\widehat{s}}\right|^p\right]^{1/p} 
 \leq C_p |\widetilde{s} -s|^{1/p} (e^{2.5\overline{c}T}V(x))^{1/p}. 
 %\label{iii5}
 \end{align*}
 \item [\textup{(ii)}]
 \begin{align*}
\left[\mathbb{E}\left[\left|\left(X^{\iota, x}_{s,\widehat{s}}-X^{\delta,x}_{s,\widehat{s}}\right) - \left(X^{\iota, x}_{\widetilde{s},\widehat{s}}-X^{\delta, x}_{\widetilde{s},\widehat{s}}\right)\right|^p\right]\right]^{1/p}  
   \leq C_p |\delta|^{1/p} |\widetilde{s} -s|^{1/p} (e^{2.5\overline{c}T}V(x))^{1/p}.
   %\label{vi1}
   \end{align*}
\end{enumerate}
\end{Lem}

\begin{proof}
%  If $\widetilde{s}$ is not a grid point, then there exists $\widehat{s}\in [\widetilde{s}, T]$ such that there is no grid points on $(s, \widehat{s})$.\\
% Let $\delta \in \widetilde{\mathbb{S}},$ $x \in \bR,$ $s\in [0, T],$ $\widetilde{s} \in [s, T], \widehat{s} \in [\widetilde{s},T],$ $\delta([0, T])\cap [s, \widehat{s}] = \emptyset $.
(i) Using the fact that there is no grid point on $(s, \widehat{s})$ and $(\widetilde{s}, \widehat{s}),$ and equation $\eqref{y01},$ we write
\begin{align*}
   X^{\delta, x}_{s,\widehat{s}} = x +\mu (x)(\widehat{s}-s)+\sigma(x)(W_{\widehat{s}}-W_s) +  \gamma(x)(Z_{\widehat{s}}-Z_{s}), \notag\\
   X^{\delta, x}_{\widetilde{s},\widehat{s}} = x +\mu (x)(\widehat{s}-\widetilde{s})+\sigma(x)(W_{\widehat{s}}-W_{\widetilde{s}}) +  \gamma(x)(Z_{\widehat{s}}-Z_{\widetilde{s}}).
\end{align*}
Therefore, proceeding as in the proof of (ii) of Lemma $\ref{Lem3.3}$, we get 
\begin{align*}
&\left[\mathbb{E}\left|X^{\delta, x}_{\widetilde{s},\widehat{s}}-X^{\delta, x}_{s,\widehat{s}}\right|^p\right]^{1/p} 
= \left[\mathbb{E}\left[\Big|\mu (x)(s-\widetilde{s})+\sigma(x)(W_s-W_{\widetilde{s}}) +  \gamma(x)(Z_{s}-Z_{\widetilde{s}})\Big|^p\right]\right]^{1/p} \notag \\
& \le \left[\mathbb{E}\left[\Big| \int_s^{\widetilde{s}} \mu (x)dr\Big|^p\right]\right]^{1/p} + \left[\mathbb{E}\left[\Big| \int_s^{\widetilde{s}} \sigma (x)dW_r\Big|^p\right]\right]^{1/p} + \left[\mathbb{E} \left[\left| \int_s^{\widetilde{s}} \int _{\bR_0}\gamma(x) z \widetilde{N}(dr, dz) \right|^p\right]\right]^{1/p}\notag \\
 & \leq C_p |\widetilde{s} -s|^{1/p} (e^{2.5\overline{c}T}V(x))^{1/p},
\end{align*}   
 % If $\widetilde{s}$ is not a grid point, then there exists $\widehat{s}\in [\widetilde{s}, T]$ such that there is no grid points on $(s, \widehat{s})$.\\
% Let $\delta \in \widetilde{\mathbb{S}},$ $x \in \bR, s\in [0, T],$ $\widetilde{s} \in [s, T],$ $\widehat{s} \in [\widetilde{s},T],$ $\delta([0, T])\cap [s, \widehat{s}] = \emptyset $.
% Using the fact that there is no grid point on $(s, \widehat{s})$ and equation $\eqref{y01}$, we have
% \begin{align*}
%  X^{\delta, x}_{s,\widehat{s}}-X^{\delta, x}_{\widetilde{s},\widehat{s}} &=   \left[ x +\mu (x)(\widehat{s}-s)+\sigma(x)(W_{\widehat{s}}-W_{s}) +  \gamma(x)(Z_{\widehat{s}}-Z_{s})\right]
%  \notag\\
%  & \quad - \left[x +\mu (x)(\widehat{s}-\widetilde{s})+\sigma(x)(W_{\widehat{s}}-W_{\widetilde{s}}) +  \gamma(x)(Z_{\widehat{s}}-Z_{\widetilde{s}})\right] \notag\\
%  & = \mu (x)(\widetilde{s}-s)+\sigma(x)(W_{\widetilde{s}}-W_s) + \gamma(x)(Z_{\widetilde{s}}-Z_{s}).
% \end{align*}
for a positive constant $C_p=C_p(T,m_2, m_p)$. This shows (i).

(ii) Using the definitions of $X^{\iota, x}_{s,\widehat{s}}, X^{\delta,x}_{s,\widehat{s}}, X^{\iota, x}_{\widetilde{s},\widehat{s}}, X^{\delta, x}_{\widetilde{s},\widehat{s}},$ together with the fact that there is no grid point on $(s, \widehat{s}),$ we decompose
\begin{align*}
 & \left(X^{\iota, x}_{s,\widehat{s}}-X^{\delta,x}_{s,\widehat{s}}\right) - \left(X^{\iota, x}_{\widetilde{s},\widehat{s}}-X^{\delta, x}_{\widetilde{s},\widehat{s}}\right) = \left(X^{\iota, x}_{s,\widehat{s}}-X^{\iota, x}_{\widetilde{s},\widehat{s}} \right) - \left(X^{\delta,x}_{s,\widehat{s}}-X^{\delta, x}_{\widetilde{s},\widehat{s}}\right) \notag\\
 & = \Bigg( \left[x+\int_s^{\widehat{s}} \mu(X_{s, r}^{\iota, x}) dr + \int_s^{\widehat{s}} \sigma(X_{s, r}^{\iota, x}) dW_r+ \int_s^{\widehat{s}} \int_{\bR_0} \gamma (X_{s, r-}^{\iota, x}) z \widetilde{N}(dr, dz) \right] \notag \\
 & \quad  -  \left[x+\int_{\widetilde{s}}^{\widehat{s}} \mu(X_{\widetilde{s}, r}^{\iota, x}) dr + \int_{\widetilde{s}}^{\widehat{s}} \sigma(X_{\widetilde{s}, r}^{\iota, x}) dW_r+ \int_{\widetilde{s}}^{\widehat{s}} \int_{\bR_0} \gamma (X_{\widetilde{s}, r-}^{\iota, x}) z \widetilde{N}(dr, dz) \right]\Bigg ) \notag \\
  & \quad - \Bigg( \left[\int_s^{\widehat{s}} \mu(x) dr + \int_s^{\widehat{s}} \sigma(x) dW_r+ \int_s^{\widehat{s}} \int_{\bR_0} \gamma(x) z \widetilde{N}(dr, dz) \right] \notag \\
    & \quad  -  \left[\int_{\widetilde{s}}^{\widehat{s}} \mu(x) dr + \int_{\widetilde{s}}^{\widehat{s}} \sigma(x) dW_r+ \int_{\widetilde{s}}^{\widehat{s}} \int_{\bR_0} \gamma(x) z \widetilde{N}(dr, dz) \right] \Bigg)  \notag \\
 & = \left[\int_s^{\widehat{s}} \mu(X_{s, r}^{\iota, x}) dr + \int_s^{\widehat{s}} \sigma(X_{s, r}^{\iota, x}) dW_r+ \int_s^{\widehat{s}} \int_{\bR_0} \gamma (X_{s, r-}^{\iota, x}) z \widetilde{N}(dr, dz) \right] \notag \\
 & \quad -  \left[\int_{\widetilde{s}}^{\widehat{s}} \mu(X_{\widetilde{s}, r}^{\iota, x}) dr + \int_{\widetilde{s}}^{\widehat{s}} \sigma(X_{\widetilde{s}, r}^{\iota, x}) dW_r+ \int_{\widetilde{s}}^{\widehat{s}} \int_{\bR_0} \gamma (X_{\widetilde{s}, r-}^{\iota, x}) z \widetilde{N}(dr, dz) \right] \notag \\
    & \quad -  \left[\int_{s}^{\widetilde{s}} \mu(x) dr + \int_{s}^{\widetilde{s}} \sigma(x) dW_r+ \int_{s}^{\widetilde{s}} \int_{\bR_0} \gamma(x) z \widetilde{N}(dr, dz) \right]  \notag \\
  & =  \left[\int_s^{\widetilde{s}}  \left(\mu(X_{s, r}^{\iota, x})-\mu(x)\right) dr + \int_s^{\widetilde{s}} \left(\sigma(X_{s, r}^{\iota, x}) -\sigma(x)\right) dW_r+ \int_s^{\widetilde{s}} \int_{\bR_0} \left(\gamma (X_{s, r-}^{\iota, x})-\gamma(x)\right) z \widetilde{N}(dr, dz) \right] \notag \\
   &  + \left[\int_{\widetilde{s}}^{\widehat{s}} \left( \mu(X_{s, r}^{\iota, x})-\mu(X_{\widetilde{s}, r}^{\iota, x}) \right) dr + \int_{\widetilde{s}}^{\widehat{s}} \left(\sigma(X_{s, r}^{\iota, x}) -\sigma(X_{\widetilde{s}, r}^{\iota, x}) \right) dW_r+ \int_{\widetilde{s}}^{\widehat{s}} \int_{\bR_0} \left(\gamma (X_{s, r-}^{\iota, x})-\gamma (X_{\widetilde{s}, r-}^{\iota, x})\right) z \widetilde{N}(dr, dz) \right]. 
\end{align*}
Thus, applying the triangle inequality and Lemma $\ref{Lem3.2},$ we have for any $p \ge 2,$
\begin{align*}
&\left[\mathbb{E}\left[\left|\left(X^{\iota, x}_{s,\widehat{s}}-X^{\delta,x}_{s,\widehat{s}}\right) - \left(X^{\iota, x}_{\widetilde{s},\widehat{s}}-X^{\delta, x}_{\widetilde{s},\widehat{s}}\right)\right|^p\right]\right]^{1/p} \notag\\ 
  & \le  \left[\mathbb{E}\left[\left|\int_s^{\widetilde{s}}  \left(\mu(X_{s, r}^{\iota, x})-\mu(x)\right) dr + \int_s^{\widetilde{s}} \left(\sigma(X_{s, r}^{\iota, x}) -\sigma(x)\right) dW_r+ \int_s^{\widetilde{s}} \int_{\bR_0} \left(\gamma (X_{s, r-}^{\iota, x})-\gamma(x)\right) z \widetilde{N}(dr, dz) \right|^p\right]\right]^{1/p} \notag \\
   & \quad + \left[\mathbb{E}\left[\left| \int_{\widetilde{s}}^{\widehat{s}} \left( \mu(X_{s, r}^{\iota, x})-\mu(X_{\widetilde{s}, r}^{\iota, x}) \right) dr + \int_{\widetilde{s}}^{\widehat{s}} \left(\sigma(X_{s, r}^{\iota, x}) -\sigma(X_{\widetilde{s}, r}^{\iota, x}) \right) dW_r \right. \right. \right. \notag \\
   & \quad + \left. \left. \left. \int_{\widetilde{s}}^{\widehat{s}} \int_{\bR_0} \left(\gamma (X_{s, r-}^{\iota, x})-\gamma (X_{\widetilde{s}, r-}^{\iota, x})\right) z \widetilde{N}(dr, dz) \right|^p\right]\right]^{1/p} 
 \notag \\
 & \leq c C_p |\widetilde{s}-s|^{1/p} \sup_{r\in [s, \widetilde{s}]} \left[ \mathbb{E} \left[\left| X^{\iota,x}_{s, r}-x \right|^p \right]\right]^{1/p} + c C_p |\widehat{s}-\widetilde{s}|^{1/p} \sup_{r\in [ \widetilde{s}, \widehat{s}]} \left[ \mathbb{E} \left[\left| X^{\iota,x}_{s, r}-X^{\iota,x}_{\widetilde{s}, r} \right|^p \right]\right]^{1/p}  \notag\\
 & \leq  c C_p |\widetilde{s}-s|^{1/p} \sup_{r\in [s, \widetilde{s}]} C_p |r-s|^{1/p} (e^{2.5\overline{c}T}V(x))^{1/p} + c C_p |\widehat{s}-\widetilde{s}|^{1/p} C_p |\widetilde{s} - s|^{1/p} (e^{2.5\overline{c}T}V(x))^{1/p} \notag\\
 & \leq  C_p |\widetilde{s}-s|^{1/p}  |\delta|^{1/p} (e^{2.5\overline{c}T}V(x))^{1/p} +  C_p |\delta|^{1/p}  |\widetilde{s} - s|^{1/p} (e^{2.5\overline{c}T}V(x))^{1/p} \notag\\
 % & =  2C_p |\delta|^{1/p} |\widetilde{s} -s|^{1/p} (e^{2.5\overline{c}T}V(x))^{1/p} \notag \\
  & \leq C_p |\delta|^{1/p} |\widetilde{s} -s|^{1/p} (e^{2.5\overline{c}T}V(x))^{1/p}.
  % \label{vi1}
\end{align*} 
Here, to estimate $\left[ \mathbb{E} \left[\left| X^{\iota,x}_{s, r}-x \right|^p \right]\right]^{1/p},$ we  proceed as $\eqref{Lem3.3.1}$-$\eqref{Lem3.3.4}$ in the proof of Lemma $\ref{Lem3.3}$. To estimate $\left[ \mathbb{E} \left[\left| X^{\iota,x}_{s, r}-X^{\iota,x}_{\widetilde{s}, r} \right|^p \right]\right]^{1/p},$ we use the statement (i) above. This finishes the desired proof.
\end{proof}

% \begin{Rem}
%    [Extend the result of $(i)$.]\\
%    Let $W = V^m$, extending the result of $(i)$, we get
%    $$\mathbb{E}\left[V(X^{\delta, x}_{s,t})\right] \leq e^{2\overline{c}|t-s|}V(x).$$
% \end{Rem}
Now, we state the main result which establishes the strong convergence rate of the Euler-Maruyama approximations in temporal-spatial H\"older-norms.
\begin{Thm} \label{Thm} 
%\textcolor{red}{Bo sung thong tin nhu cac Bo de o tren khi gioi thieu $X^{\delta,x}_{s,t}$}
For $ \delta \in \widetilde{\mathbb{S}},$ $s\in [0, T],$ $t \in [s, T]$  and $ x\in \bR,$ let $(X^{\delta,x}_{s,t})_{t\in[s,T]}:[s,T]  \times\Omega\to\mathbb{R} $ 
be an $(\mf_t)_{t\in[s,T]}$-adapted stochastic process with c\`adl\`ag  sample paths satisfying equation \eqref{y01}. 
%Let $p \ge 2$ and $ \mu, \sigma, \gamma, V$ satisfy conditon {\bf A1} and statements {\bf A2-A5}. 
Then, 
\begin{enumerate}
 %   \item [\textup{(i)}] it holds for all $\delta \in \widetilde{\mathbb{S}},$ $ s\in [0, T],$ $t\in [s, T],$ $ x \in \bR$ that
 %$$\mathbb{E}\left[V(X^{\delta, x}_{s,t})\right] \leq e^{2\overline{c}|t-s|}V(x).$$
    \item [\textup{(i)}] for all $\delta \in \widetilde{\mathbb{S}},$ $ s\in [0, T],$ $t\in [s, T],$ $ x \in \bR$, there exists a positive constant $C_p$ such that $$\left[\mathbb{E}\left[\left|X^{\delta, x}_{s,t}-X^{\iota, x}_{s,t}\right|^p\right]\right]^{1/p}\leq C_p |t-s|^{1/p} |\delta|^{1/p} (e^{2.5\overline{c}T}V(x))^{1/p}.$$
    \item [\textup{(ii)}] for all $\delta \in \widetilde{\mathbb{S}},$ $ s, \widetilde{s} \in [0, T],$ $t\in [s, T],$ $\widetilde{t} \in [\widetilde{s}, T],$ $ x, \widetilde{x} \in \bR$, there exists a positive constant $C_p$ such that 
    \begin{align*}
&\left[\mathbb{E}\left[\left|X^{\delta, x}_{s,t}-X^{\delta, \widetilde{x}}_{\widetilde{s},\widetilde{t}}\right|^p\right]\right]^{1/p}  \leq C_p e^{2.5\overline{c}T/p} \dfrac{(V(x))^{1/p} + (V(\widetilde{x}))^{1/p}}{2} \left[|s-\widetilde{s} |^{1/p} + |t-\widetilde{t} |^{1/p}\right] + C_p|x-\widetilde{x}|.\end{align*}
    \item [\textup{(iii)}]  for all $\delta \in \widetilde{\mathbb{S}},$ $ s \in [0, T],$ $t, \widetilde{t} \in [s, T],$ $ x, \widetilde{x} \in \bR$, there exists a positive constant $C_p$ such that \begin{align*}
\left[\mathbb{E}\left[\left|\left(X^{\delta, x}_{s,\widetilde{t}}-X^{\delta,x}_{s,t}\right) - \left(X^{\delta, \widetilde{x}}_{s,\widetilde{t}}-X^{\delta, \widetilde{x}}_{s,t}\right)\right|^p\right]\right]^{1/p}  \le C_p |t-\widetilde{t}|^{1/p} |x-\widetilde{x}|.
    \end{align*} 
    \item [\textup{(iv)}] for all $\delta \in \widetilde{\mathbb{S}},$ $ s \in [0, T],$ $t\in [s, T],$ $m>1,$ $x, \widetilde{x}, y, \widetilde{y} \in \bR$, there exist a positive constant $C_p$ such that
   \begin{align*}
&\left[\mathbb{E}\left[\left|\left(X^{\iota, x}_{s,t}-X^{\delta,y}_{s,t}\right) - \left(X^{\iota, \widetilde{x}}_{s,t}-X^{\delta, \widetilde{y}}_{s,t}\right)\right|^p\right]\right]^{1/p} \notag\\
& \leq 2^{1-\frac{1}{p}} e^{C_p T} |(x-y)-(\widetilde{x}-\widetilde{y})|  + C_p e^{C_p T} |t-s|^{1/p} |x-\widetilde{x}| |\delta|^{1/p}  \notag\\
& \quad + C_p  e^{C_p T}  |t-s|^{1/p} |x-\widetilde{x}| |\delta|^{1/pm} (e^{2.5\overline{c}T})^{1/pm} \frac{  (V(x))^{1/pm}+(V(\widetilde{x}))^{1/pm}}{2}  \notag\\   
& \quad + C_p e^{C_p T}  |t-s|^{1/p} |x-\widetilde{x}|
\dfrac{|x-y|+ |\widetilde{x}-\widetilde{y}|}{2}.
% & \leq  e^{C_p T^{(1/p)+1}} \Bigg\{|(x-y)-(\widetilde{x}-\widetilde{y})| \notag\\
% & \quad + C_p |t-s|^{1/p} |x-\widetilde{x}|  \Bigg[   |\delta |^{1/p}  + |\delta|^{1/pm} (e^{2.5\overline{c}T})^{1/pm} \frac{  (V(x))^{1/pm}+(V(\widetilde{x}))^{1/pm}}{2}  +  \dfrac{|x-y|+ |\widetilde{x}-\widetilde{y}|}{2}    \Bigg] \Bigg\}.
\end{align*}
    \item [\textup{(v)}] for all $\delta \in \widetilde{\mathbb{S}},$ $ s, \widetilde{s} \in [0, T],$ $t\in [s, T],$ $\widetilde{t} \in [\widetilde{s}, T],$ $m>1,$ H\"older conjugates $\kappa_1$ and $\kappa_2,$ $ x, \widetilde{x} \in \bR,$ there exists a positive constant $C_p$ such that\begin{align*}
&\left[\mathbb{E}\left[\left|\left(X^{\iota, x}_{s,t}-X^{\iota,\widetilde{x}}_{\widetilde{s},\widetilde{t}}\right) - \left(X^{\delta, x}_{s,t}-X^{\delta, \widetilde{x}}_{\widetilde{s},\widetilde{t}}\right)\right|^p\right]\right]^{1/p} \notag\\
  & \leq  C_p e^{C_p T} (e^{5 \overline{c}T})^{1/p}         \frac{(V(x))^{2/p}+(V(\widetilde{x}))^{2/p}}{2} \left[  |s-\widetilde{s}|^{1/p\kappa_1} + |t-\widetilde{t}|^{1/p} +|x-\widetilde{x}| \right]  |\delta|^{\frac{1}{p(m \vee \kappa_2)}}. 
    \end{align*}
\end{enumerate}
Here, the constant $\overline{c}$ is given in {\bf A3, A4, A5}.
\end{Thm}

\begin{proof}
 $\textup{(i)}$ 
%We start by estimating $\left[\mathbb{E}\left[\left|X^{\delta, x}_{s,t}-X^{\iota, x}_{s,t}\right|^p\right]\right]^{1/p}.$ 
First, using equation $\eqref{y01}$, we write $X^{\delta, x}_{s,t}-X^{\iota, x}_{s,t} = P_1 + P_2,$ where
\begin{align*}
  P_1 = &\int_s^t \left(\mu(X_{s,\max\{s,\delta(r)\}}^{\delta,x}) - \mu(X_{s,\max\{s,\delta(r)\}}^{\iota,x})\right) dr + \int_s^t \left( \sigma(X_{s,\max\{s,\delta(r)\}}^{\delta,x}) - \sigma(X_{s,\max\{s,\delta(r)\}}^{\iota,x}) \right) dW_r \notag \\
   & \quad + \int_s^t \int_{\bR_0} \left(\gamma(X_{s,\max\{s,\delta(r-)\}}^{\delta,x})-\gamma(X_{s,\max\{s,\delta(r-)\}}^{\iota,x}) \right) z \widetilde{N}(dr, dz),  \\
    P_2 =  & \int_s^t \left(\mu(X_{s,\max\{s,\delta(r)\}}^{\iota,x}) - \mu(X_{s,r}^{\iota,x}) \right) dr + \int_s^t \left( \sigma(X_{s,\max\{s,\delta(r)\}}^{\iota,x}) - \sigma(X_{s,r}^{\iota,x}) \right) dW_r \notag \\
   & \quad + \int_s^t \int_{\bR_0} \left(\gamma(X_{s,\max\{s,\delta(r-)\}}^{\iota,x})-\gamma(X_{s,r-}^{\iota,x})\right) z \widetilde{N}(dr, dz).
\end{align*}
Then for any $p \ge 2,$ 
\begin{align}
\label{P1,2}
\left[\mathbb{E}\left[\left|X^{\delta, x}_{s,t}-X^{\iota, x}_{s,t}\right|^p\right]\right]^{1/p} \le \left[\mathbb{E}\left[\left|P_1\right|^p\right]\right]^{1/p} + \left[\mathbb{E}\left[\left|P_2\right|^p\right]\right]^{1/p}.    
\end{align}

% For simplicity, we start with the estimate of $\left[\mathbb{E}\left[\left|P_2\right|^p\right]\right]^{1/p}.$
% We have,
% \begin{align*}
% \left[\mathbb{E}\left[\left|P_2\right|^p\right]\right]^{1/p} & =\left[\mathbb{E}\left[\left|\int_s^t \mu(X_{s,\max\{s,\delta(r)\}}^{\iota,x}) - \mu(X_{s,r}^{\iota,x}) dr + \int_s^t \sigma(X_{s,\max\{s,\delta(r)\}}^{\iota,x}) - \sigma(X_{s,r}^{\iota,x}) dW_r \right. \right. \right.\notag \\
%    & \quad + \left. \left. \left. \int_s^t \int_{\bR_0} (\gamma(X_{s,\max\{s,\delta(r-)\}}^{\iota,x})-\gamma(X_{s,r-}^{\iota,x})) z \widetilde{N}(dr, dz)\right|^p\right]\right]^{1/p}.
% \end{align*}
Using the triangle inequality, H\"older's inequality and Burkholder-Davis-Gundy inequalities of the It\^o stochastic integral with respect to Brownian motion and compensated Poisson random measure (see Lemma \ref{BDGjump}), we get
\begin{align}
\left[\mathbb{E}\left[\left|P_1\right|^p\right]\right]^{1/p} 
 & \le \left[\mathbb{E}\left[\left|\int_s^t \left( \mu(X_{s,\max\{s,\delta(r)\}}^{\delta,x}) - \mu(X_{s,\max\{s,\delta(r)\}}^{\iota,x}) \right) dr\right|^p\right]\right]^{1/p} \notag \\
 & \quad + \left[\mathbb{E}\left[\left| \int_s^t \left( \sigma(X_{s,\max\{s,\delta(r)\}}^{\delta,x}) - \sigma(X_{s,\max\{s,\delta(r)\}}^{\iota,x}) \right) dW_r \right|^p\right]\right]^{1/p} \notag \\
   & \quad +\left[\mathbb{E}\left[\left| \int_s^t \int_{\bR_0} \left(\gamma (X_{s,\max\{s,\delta(r-)\}}^{\delta,x})-\gamma (X_{s,\max\{s,\delta(r-)\}}^{\iota,x})\right) z \widetilde{N}(dr, dz)\right|^p\right]\right]^{1/p}\notag \\ 
 %& \le c(t-s)^{\frac{p-1}{p}} \left[ \int_s^t \mathbb{E} \left[ \left| X_{s,\max\{s,\delta(r)\}}^{\delta,x} -X_{s,\max\{s,\delta(r)\}}^{\iota,x} \right|^p\right] dr \right]^{1/p}\notag \\ 
 %& \quad + c (C_{1p})^{1/p} (t-s)^{\frac{p/2-1}{p}} \left[ \int_s^t \mathbb{E} \left[ \left| X_{s,\max\{s,\delta(r)\}}^{\delta,x} -X_{s,\max\{s,\delta(r)\}}^{\iota,x} \right|^p\right] dr \right]^{1/p}\notag \\ 
 %& \quad + c (C_{2p})^{1/p} \left(\int_{\mathbb{R}_0}|z|^p \nu (dz) \right)^{1/p} \left[ \int_s^t \mathbb{E} \left[ \left| X_{s,\max\{s,\delta(r)\}}^{\delta,x} -X_{s,\max\{s,\delta(r)\}}^{\iota,x} \right|^p\right] dr \right]^{1/p}\notag \\
  & \le c |t-s|^{1-\frac{1}{p}} \left[ \int_s^t \mathbb{E} \left[ \left| X_{s,\max\{s,\delta(r)\}}^{\delta,x} -X_{s,\max\{s,\delta(r)\}}^{\iota,x} \right|^p\right] dr \right]^{1/p}\notag \\ 
 & \quad + c C_{1p} |t-s|^{\frac{1}{2}-\frac{1}{p} } \left[ \int_s^t \mathbb{E} \left[ \left| X_{s,\max\{s,\delta(r)\}}^{\delta,x} -X_{s,\max\{s,\delta(r)\}}^{\iota,x} \right|^p\right] dr \right]^{1/p}\notag \\ 
 & \quad + c C_{2p} (|t-s|^{\frac{p}{2}-1}m_2^{\frac{p}{2}}+m_p )^{1/p} \left[ \int_s^t \mathbb{E} \left[ \left| X_{s,\max\{s,\delta(r)\}}^{\delta,x} -X_{s,\max\{s,\delta(r)\}}^{\iota,x} \right|^p\right] dr \right]^{1/p}\notag \\
 & \le c C_p \left[ \int_s^t \mathbb{E} \left[ \left| X_{s,\max\{s,\delta(r)\}}^{\delta,x} -X_{s,\max\{s,\delta(r)\}}^{\iota,x} \right|^p\right] dr \right]^{1/p}.
%& = \left(c_1(t-s)^{\frac{p-1}{p}} + c_1(t-s)^{\frac{p/2-1}{p}} + c_1(m_p)^{1/p}  \right) \left[ \int_s^t \mathbb{E} \left[ \left| X_{s,\max\{s,\delta(r)\}}^{\delta,x} -X_{s,\max\{s,\delta(r)\}}^{i,x} \right|^p\right] dr \right]^{1/p} \notag \\
 \label{P1}
\end{align}

 Next, using the triangle inequality, proceeding as in the proof of 
Lemma $\ref{Lem3.2}$ and the Lipschitz property of the coefficients $\mu$, $\sigma$, $\gamma$ (see  Remark $\ref{Rem1}$), we have
 \begin{align*}
\left[\mathbb{E}\left[\left|P_2\right|^p\right]\right]^{1/p} &  \le \left[\mathbb{E}\left[\left|\int_s^t \left ( \mu(X_{s,\max\{s,\delta(r)\}}^{\iota,x}) - \mu(X_{s,r}^{\iota,x}) \right)dr \right|^p\right]\right]^{1/p} \notag \\
   & \quad  + \left[\mathbb{E}\left[\left| \int_s^t \left( \sigma(X_{s,\max\{s,\delta(r)\}}^{\iota,x}) - \sigma(X_{s,r}^{\iota,x}) \right) dW_r \right|^p\right]\right]^{1/p}\notag \\
    &  \quad + \left[\mathbb{E}\left[\left|\int_s^t \int_{\bR_0} \left(\gamma (X_{s,\max\{s,\delta(r-)\}}^{\iota,x})-\gamma(X_{s,r-}^{\iota,x})\right) z \widetilde{N}(dr, dz)\right|^p\right]\right]^{1/p}\notag\\ 
    & \le  c C_p |t-s|^{1/p} \sup_{r\in [s, t]} \left [\mathbb{E}\left[ \left| X_{s,\max\{s,\delta(r)\}}^{\iota,x} -X_{s,r}^{\iota,x} \right|^p\right] \right]^{1/p}.
\end{align*}
% This completes the estimation of $\left[\mathbb{E}\left[\left(P_2\right)^2\right]\right]^{1/2}$.\\
% Now, we start with the estimate of  $\left[\mathbb{E}\left[\left(P_1\right)^2\right]\right]^{1/2}.$
Applying (ii) of Lemma $\ref{Lem3.3}$ with $X^{\delta, x}_{s,\widetilde{t}}=X_{s,\max\{s,\delta(r)\}}^{\iota,x}$ and $X^{\delta, x}_{s,t} = X_{s,r}^{\iota,x}$, we obtain 
\begin{align}
\left[\mathbb{E}\left[\left|P_2\right|^p\right]\right]^{1/p} & \leq c C_p |t-s|^{1/p} \sup_{r\in [s, t]} C_p | \max\{s,\delta(r)\} - r |^{1/p} (e^{2.5\overline{c}T}V(x))^{1/p}     \notag\\
&\le  c C_p |t-s|^{1/p} |\delta|^{1/p} (e^{2.5\overline{c}T}V(x))^{1/p}.
%\notag  \\
 %& = c C_p |t-s|^{1/p} |\delta|^{1/p} (e^{2.5\overline{c}T}V(x))^{1/p}.
 \label{P2}
\end{align}
 %This completes the estimation of 
%$\left[\mathbb{E}\left[\left|P_2\right|^p\right]\right]^{1/p}$.\\
%  Now, we start with the estimate of  $\left[\mathbb{E}\left[\left|P_1\right|^p\right]\right]^{1/p}.$
% \begin{align*}
% &\left[\mathbb{E}\left[\left|P_1\right|^p\right]\right]^{1/p} \notag \\
%  &= \left[\mathbb{E}\left[\left|\int_s^t \mu(X_{s,\max\{s,\delta(r)\}}^{\delta,x}) - \mu(X_{s,\max\{s,\delta(r)\}}^{\iota,x}) dr + \int_s^t \sigma(X_{s,\max\{s,\delta(r)\}}^{\delta,x}) - \sigma(X_{s,\max\{s,\delta(r)\}}^{\iota,x}) dW_r \right. \right. \right. \notag \\
%    & \quad + \left. \left. \left. \int_s^t \int_{\bR_0} (\gamma (X_{s,\max\{s,\delta(r-)\}}^{\delta,x})-\gamma (X_{s,\max\{s,\delta(r-)\}}^{\iota,x})) z \widetilde{N}(dr, dz)\right|^p\right]\right]^{1/p}.
% \end{align*}
Hence, combining $\eqref{P1,2},$ $\eqref{P1}$ and $\eqref{P2}$, we obtain
\begin{align*}
\left[\mathbb{E}\left[\left|X^{\delta, x}_{s,t}-X^{\iota, x}_{s,t}\right|^p\right]\right]^{1/p} 
& \le c C_p \left[ \int_s^t \mathbb{E} \left[ \left| X_{s,\max\{s,\delta(r)\}}^{\delta,x} -X_{s,\max\{s,\delta(r)\}}^{\iota,x} \right|^p\right] dr \right]^{1/p}\notag \\
& \quad  + c C_p |t-s|^{1/p} |\delta|^{1/p} (e^{2.5\overline{c}T}V(x))^{1/p}.
\end{align*}
Now, applying Lemma $\ref{Cor2.1}$ for $x(t)= \left[\mathbb{E}\left[\left|X^{\delta, x}_{s,t}-X^{\iota, x}_{s,t}\right|^p\right]\right]^{1/p},$ $ a(t)= c C_p |t-s|^{1/p} |\delta|^{1/p} (e^{2.5\overline{c}T}V(x))^{1/p} $ and $c_{\star} = c C_p,$ we get 
\begin{align*}
\left[\mathbb{E}\left[\left|X^{\delta, x}_{s,t}-X^{\iota, x}_{s,t}\right|^p\right]\right]^{1/p} & \leq 2^{1-\frac{1}{p}}  e^{\frac{2^{p-1} (c C_p )^p T}{p}} c C_p |t-s|^{1/p} |\delta|^{1/p} (e^{2.5\overline{c}T}V(x))^{1/p} \notag \\
  & \leq  C_p |t-s|^{1/p} |\delta|^{1/p} (e^{2.5\overline{c}T}V(x))^{1/p},
  % \label{x, wx}
 \end{align*}
for a positive constant $C_p=C_p(c, m_2, m_p, T).$ 
This shows $\textup{(i)}.$\\
 $\textup{(ii)}$ Without loss of generality, we suppose that $s< \widetilde{s}$ and $t = \max\{s, \tilde{s}, t, \tilde{t}\}$.

% Here, we need to estimate $\left[\mathbb{E}\left[\left|X^{\delta, x}_{s,t}-X^{\delta, \widetilde{x}}_{\widetilde{s},\widetilde{t}}\right|^p\right]\right]^{1/p}$ 
%For $\delta \in \widetilde{\mathbb{S}},$ $ s, \widetilde{s}\in [0, T], t\in [s, T], \widetilde{t} \in [\widetilde{s}, T] $, $ x, \widetilde{x} \in \bR$ and $p \ge 2$, 
Using the triangle inequality, we get
\begin{align}
\left[\mathbb{E}\left[\left|X^{\delta, x}_{s,t}-X^{\delta, \widetilde{x}}_{\widetilde{s},\widetilde{t}}\right|^p\right]\right]^{1/p}
%= \left[\mathbb{E}\left[\left|X^{\delta, x}_{s,\max\{s, t\}}-X^{\delta, \widetilde{x}}_{\widetilde{s},\max\{\widetilde{s}, \widetilde{t}\}}\right|^p\right]\right]^{1/p} 
%& \leq  \left[\mathbb{E}\left[\left|X^{\delta, x}_{s,\max\{s, t\}}-X^{\delta, x}_{s,\max\{s, \widetilde{s}, t, \widetilde{t}\}}\right|^p\right]\right]^{1/p} +  \left[\mathbb{E}\left[\left|X^{\delta, x}_{s,\max\{s, \widetilde{s}, t, \widetilde{t}\}} - X^{\delta, x}_{\widetilde{s},\max\{s, \widetilde{s}, t, \widetilde{t}\}}\right|^p\right]\right]^{1/p}   \notag\\
%& \quad +\left[\mathbb{E}\left[\left| X^{\delta, x}_{\widetilde{s},\max\{s, \widetilde{s}, t, \widetilde{t}\}} - X^{\delta, x}_{\widetilde{s},\max\{ \widetilde{s},  \widetilde{t}\}}\right|^p\right]\right]^{1/p} + \left[\mathbb{E}\left[\left|  X^{\delta, x}_{\widetilde{s},\max\{ \widetilde{s},  \widetilde{t}\}}- X^{\delta, \widetilde{x}}_{\widetilde{s},\max\{ \widetilde{s},  \widetilde{t}\}}\right|^p\right]\right]^{1/p} \notag\\
 \le  T_1+ T_2+ T_3,
 \label{t1234}
\end{align}
where 
\begin{align*}
%   &T_1:= \left[\mathbb{E}\left[\left|X^{\delta, x}_{s,\max\{s, t\}}-X^{\delta, x}_{s,\max\{s, \widetilde{s}, t, \widetilde{t}\}}\right|^p\right]\right]^{1/p},\\
   &T_1:= \left[\mathbb{E}\left[\left|X^{\delta, x}_{s,t} - X^{\delta, x}_{\widetilde{s},t}\right|^p\right]\right]^{1/p}, \quad 
   &T_2:= \left[\mathbb{E}\left[\left| X^{\delta, x}_{\widetilde{s},t} - X^{\delta, x}_{\widetilde{s},  \widetilde{t}}\right|^p\right]\right]^{1/p}, \quad 
   &T_3:= \left[\mathbb{E}\left[\left|  X^{\delta, x}_{\widetilde{s}, \widetilde{t}}- X^{\delta, \widetilde{x}}_{\widetilde{s},  \widetilde{t}}\right|^p\right]\right]^{1/p}.
\end{align*}
 First, applying (ii) of Lemma $\ref{Lem3.3}$ for $T_2,$ and Lemma $\ref{Lem3.4}$ for $T_3$, we obtain
 \begin{align}
  % &T_1 \le  C_p (e^{2.5\overline{c}T}V(x))^{1/p} |\max \{s, t\} - \max \{s, \widetilde{s}, t, \widetilde{t}\}|^{1/p}, \label{iii1}\\
  &T_2 \le C_p | t- \widetilde{t} |^{1/p} (e^{2.5\overline{c}T}V(x))^{1/p} ,  \label{iii3}\\
     &T_3 \le C_p |x-\widetilde{x}|.\label{iii4}
\end{align}
%Next, since $t=\max\{s, \widetilde{s}, t, \widetilde{t}\}$,  we have  $ T_2=\left[\mathbb{E}\left[\left|X^{\delta, x}_{s,t} - X^{\delta, x}_{\widetilde{s},t}\right|^p\right]\right]^{1/p}$ . 
We next estimate the term $T_1$ by considering the following cases.

 \begin{figure}[ht]
\centering
% (Case1)
\begin{subfigure}{0.45\textwidth}
\captionsetup{labelformat=empty}
\centering
\begin{tikzpicture}
  \draw[thick] (0,0) -- (4,0);
  \node at (0,0) [below] {$s$};
  \node at (2,0) [below] {$\widetilde{s}$};
  \node at (4,0) [below] {$t$};
  \foreach \x in {0,2,4}
    \draw (\x,0.1)--(\x,-0.1);
  \node at (2,0) {$\times$};
\end{tikzpicture}
\caption{Case 1: $\widetilde{s}$ is a grid point.}
\end{subfigure}
\hfill\\
% (Case2)
% \begin{subfigure}{0.45\textwidth}
% \captionsetup{labelformat=empty}
% \centering
% \begin{tikzpicture}
%   % trục
%   \draw[thick] (0,0) -- (4,0);
%   % các điểm
%   \node at (0,0) [below] {$s$};
%   \node at (2,0) [below] {$\widetilde{s}$};
%   \node at (4,0) [below] {$t$};
%   % tick
%   \foreach \x in {0,2,4}
%     \draw (\x,0.1)--(\x,-0.1);
% \end{tikzpicture}
% \caption{Case 2: $\widetilde{s}$ is not a grid point.}
% \end{subfigure}
% \hfill\\
% (b)
\begin{subfigure}{0.45\textwidth}
\captionsetup{labelformat=empty}
\centering
\begin{tikzpicture}
  \draw[thick] (0,0) -- (5,0);
  \node at (0,0) [below] {$s$};
  \node at (1.5,0) [below] {$\widetilde{s}$};
  \node at (3,0) [below] {$\bar{s}$};
  \node at (5,0) [below] {$t$};
  % tick
  \foreach \x in {0,1.5,3,5}
    \draw (\x,0.1)--(\x,-0.1);
  % grid point
  \node at (3,0) {$\times$};
\end{tikzpicture}
\caption{Case 2: $\bar{s}$ is the smallest grid point on $(\widetilde{s},t]$.}
\end{subfigure}
\hfill
% (d)
\begin{subfigure}{0.45\textwidth}
\captionsetup{labelformat=empty}
\centering
\begin{tikzpicture}
  \draw[thick] (0,0) -- (5,0);
  \node at (0,0) [below] {$s$};
  \node at (1.5,0) [below] {$\underline{s}$};
  \node at (3.5,0) [below] {$\widetilde{s}$};
  \node at (5,0) [below] {$t$};
  \foreach \x in {0,1.5,3.5,5}
    \draw (\x,0.1)--(\x,-0.1);
  \node at (1.5,0) {$\times$};
\end{tikzpicture}
\caption{Case 3: $\underline{s}$ is the largest grid point on $[s,\widetilde{s})$.}
\end{subfigure}

\caption{An illustration for the case distinction. A grid point is drawn by $\times$.}
\end{figure}

\noindent\underline{Case 1}: $\widetilde{s}$ \text{ is a grid point.}

%Let $\delta \in \widetilde{\mathbb{S}},$ $x \in \bR,$ $s\in [0, T],$ $\widetilde{s} \in \delta([0, T])\cap [s, T],$ $t\in [\widetilde{s}, T]$. 
Applying the Markov property and Lemma $\ref{Lem3.4},$ we have
\begin{align}
    \label{s, ws}  \left[\mathbb{E}\left[\left|X^{\delta, x}_{s,t}-X^{\delta, x}_{\widetilde{s},t}\right|^p\right]\right]^{1/p} &= \left(\mathbb{E} \left[ \left(\left[\mathbb{E}\left[\left|X^{\delta,\eta}_{\widetilde{s},t}-X^{\delta, x}_{\widetilde{s},t}\right|^p\right]\right]^{1/p}\Big | _{\eta = X^{\delta, x}_{s,\widetilde{s}} }\right)^p\right] \right)^{1/p} \notag \\
    & \leq C_p \left[\mathbb{E}\left[\left|X^{\delta, x}_{s,\widetilde{s}}-x \right|^p\right]\right]^{1/p}.  \end{align}
Using equation $\eqref{y01},$ the triangle inequality and proceeding as in the proof of (ii) of Lemma $\ref{Lem3.3}$, we get
\begin{align}
\label{18}
\left[\mathbb{E}\left[\left|X^{\delta, x}_{s,\widetilde{s}}-x \right|^p\right]\right]^{1/p} 
 %& = \left[\mathbb{E}\left[\left| \int_s^{\widetilde{s}} \mu(X_{s,\max\{s,\delta(r)\}}^{\delta,x}) dr + \int_s^{\widetilde{s}} \sigma(X_{s,\max\{s,\delta(r)\}}^{\delta,x}) dW_r \right. \right. \right. \notag \\  
% & \quad \quad \quad \quad + \left. \left. \left. \int_s^{\widetilde{s}} \int _{\bR_0}\gamma (X_{s-,\max\{s,\delta(r-)\}}^{\delta,x}) z \widetilde{N}(dr, dz)   \right|^p\right]\right]^{1/p}   \notag \\
 &\leq \left[\mathbb{E} \left[\left| \int_s^{\widetilde{s}} \mu(X_{s,\max\{s,\delta(r)\}}^{\delta,x}) dr \right|^p \right]\right]^{1/p} + \left[\mathbb{E} \left[\left| \int_s^{\widetilde{s}} \sigma(X_{s,\max\{s,\delta(r)\}}^{\delta,x}) dW_r \right|^p \right]\right]^{1/p} \notag \\
& \quad  + \left[\mathbb{E} \left[\left| \int_s^{\widetilde{s}} \int _{\bR_0}\gamma (X_{s,\max\{s,\delta(r-)\}}^{\delta,x}) z \widetilde{N}(dr, dz) \right|^p\right]\right]^{1/p} \notag \\
&  \leq  C_p |\widetilde{s} -s|^{1/p} (e^{2.5\overline{c}T}V(x))^{1/p}. \end{align}
Therefore, combining $\eqref{s, ws}$ and $\eqref{18}$, we obtain
    \begin{align}
T_1 = \left[\mathbb{E}\left[\left|X^{\delta, x}_{s,t}-X^{\delta, x}_{\widetilde{s},t}\right|^p\right]\right]^{1/p} &\leq C_p |\widetilde{s} -s|^{1/p} (e^{2.5\overline{c}T}V(x))^{1/p},
\label{cA2}
\end{align}
for a positive constant $C_p=C_p(c, m_2, m_p, T).$ 

\noindent\underline{Case 2}: \text{$\widetilde{s}$ is not a grid point and there is no grid point on $[s, \widetilde{s}]$}.

Let $\overline{s}$ be the smallest grid point on $[\widetilde{s}, t].$
%Let $\delta \in \widetilde{\mathbb{S}},$ $x \in \bR,$ $s\in [0, T],$ $\widetilde{s} \in [s, T],$ $\overline{s} \in [\widetilde{s},T],$ $\delta([0, T])\cap (s, \overline{s}) = \emptyset $, $\delta([0, T])\cap (s, \widetilde{s}) = \emptyset $ and  $\overline{s}= \delta([0, T])\cap [\widetilde{s}, T] $. This implies that 
Thus, there is no grid point on $[s, \overline{s})$. Applying the Markov property, Lemma $\ref{Lem3.4}$ 
%for $x=X^{\delta, x}_{s,\overline{s}}$, $ \widetilde{x}=X^{\delta, x}_{\widetilde{s},\overline{s}}$ 
and (i) of Lemma $\ref{Lem3.5},$  we have
\begin{align} T_1 =\left[\mathbb{E}\left[\left|X^{\delta, x}_{s,t}-X^{\delta, x}_{\widetilde{s},t}\right|^p\right]\right]^{1/p} &= \left(\mathbb{E}\left[\left(\left[\mathbb{E}\left[\left|X^{\delta,\eta}_{\overline{s},t}-X^{\delta, \widetilde{\eta}}_{\overline{s},t}\right|^p\right]\right]^{1/p}\Big | _{\eta = X^{\delta, x}_{s,\overline{s}}, \widetilde{\eta} = X^{\delta, x}_{\widetilde{s},\overline{s}} } \right)^p\right]\right)^{1/p} \notag   \\
 & \leq \left(\mathbb{E}\left[\left( C_p \left[\mathbb{E}\left[\left|\eta-\widetilde{\eta}\right|^p\right]\right]^{1/p}\Big | _{\eta = X^{\delta, x}_{s,\overline{s}}, \widetilde{\eta} = X^{\delta, x}_{\widetilde{s},\overline{s}} } \right)^p\right]\right)^{1/p} \notag   \\
 & = C_p \left[\mathbb{E}\left[\left|X^{\delta, x}_{s,\overline{s}}-X^{\delta, x}_{\widetilde{s},\overline{s}}\right|^p\right]\right]^{1/p}  \notag\\
% &\le  C_p C_p |\widetilde{s} -s|^{1/p} (e^{2.5\overline{c}T}V(x))^{1/p} \notag   \\
 & \le C_p |\widetilde{s} -s|^{1/p} (e^{2.5\overline{c}T}V(x))^{1/p},
 \label{stiles}
\end{align}
for a positive constant $C_p=C_p(c, m_2, m_p, T).$ 

\noindent\underline{Case 3}: \text{$\widetilde{s}$ is not a grid point and there is a grid point on $[s, \widetilde{s}]$}.

Let $\underline{s}$ be the largest grid point on $[s,\widetilde{s}].$ Then, using the fact that $\underline{s}$  is the grid point and proceeding as in the proof of $\eqref{cA2}$ of Case $1$, we get
 \begin{align*} \left[\mathbb{E}\left[\left|X^{\delta, x}_{s,t}-X^{\delta, x}_{\underline{s},t}\right|^p\right]\right]^{1/p} \leq C_p |\underline{s} -s|^{1/p} (e^{2.5\overline{c}T}V(x))^{1/p} \leq C_p |\widetilde{s} -s|^{1/p} (e^{2.5\overline{c}T}V(x))^{1/p}.
\end{align*}
Next, using the fact that there is no grid point on  $(\underline{s}, \widetilde{s}]$ and proceeding as in the proof of $\eqref{stiles}$ of Case 2, we get
\begin{align*}
\left[\mathbb{E}\left[\left|X^{\delta, x}_{\underline{s},t}-X^{\delta, x}_{\widetilde{s},t}\right|^p\right]\right]^{1/p}\leq  C_p |\widetilde{s} -\underline{s}|^{1/p} (e^{2.5\overline{c}T}V(x))^{1/p}\leq  C_p |\widetilde{s} -s|^{1/p} (e^{2.5\overline{c}T}V(x))^{1/p}.
\end{align*}
Therefore, combining all the computations above yields 
\begin{align*}
T_1 = \left[\mathbb{E}\left[\left|X^{\delta, x}_{s,t}-X^{\delta, x}_{\widetilde{s},t}\right|^p\right]\right]^{1/p} & \leq  \left[\mathbb{E}\left[\left|X^{\delta, x}_{s,t}-X^{\delta, x}_{\underline{s},t}\right|^p\right]\right]^{1/p} + \left[\mathbb{E}\left[\left|X^{\delta, x}_{\underline{s},t}-X^{\delta, x}_{\widetilde{s},t}\right|^p\right]\right]^{1/p}\\
%&\leq 2C_p |\widetilde{s} -s|^{1/p} (e^{2.5\overline{c}T}V(x))^{1/p}\\
&\leq C_p |\widetilde{s} -s|^{1/p} (e^{2.5\overline{c}T}V(x))^{1/p}.
\end{align*}
Consequently, we obtain the following estimate of $T_1$ in all cases
 \begin{align} T_1 \leq C_p |\widetilde{s} -s|^{1/p} (e^{2.5\overline{c}T}V(x))^{1/p}.
 \label{allc}
 \end{align} 
% We recall 
% \begin{align*}
% &\left[\mathbb{E}\left[\left|X^{\delta, x}_{s,t}-X^{\delta, \widetilde{x}}_{\widetilde{s},\widetilde{t}}\right|^p\right]\right]^{1/p}= \left[\mathbb{E}\left[\left|X^{\delta, x}_{s,\max\{s, t\}}-X^{\delta, \widetilde{x}}_{\widetilde{s},\max\{\widetilde{s}, \widetilde{t}\}}\right|^p\right]\right]^{1/p} \notag\\
% & \leq  T_1+ T_2+T_3+T_4.
% \end{align*}
From $\eqref{t1234}$,  $\eqref{iii3}$, $\eqref{iii4}$ and $\eqref{allc}$, we obtain
\begin{align*}
\left[\mathbb{E}\left[\left|X^{\delta, x}_{s,t}-X^{\delta, \widetilde{x}}_{\widetilde{s}, \widetilde{t}}\right|^p\right]\right]^{1/p} 
& \leq 
 C_p |\widetilde{s} -s|^{1/p} (e^{2.5\overline{c}T}V(x))^{1/p}  + C_p | t -  \widetilde{t} |^{1/p}  (e^{2.5\overline{c}T}V(x))^{1/p}   + C_p |x-\widetilde{x}|,
\end{align*}
which implies $\textup{(ii)}.$\\
(iii) 
%Now, we estimate $\left[\mathbb{E}\left[\left|\left(X^{\delta, x}_{s,\widetilde{t}}-X^{\delta,x}_{s,t}\right) - \left(X^{\delta, \widetilde{x}}_{s,\widetilde{t}}-X^{\delta, \widetilde{x}}_{s,t}\right)\right|^p\right]\right]^{1/p}.$ 
Without loss of generality, we suppose that $t<\widetilde{t}.$ Using equation $\eqref{y01}$ and applying Lemma $\ref{Lem3.2},$ together with Lemma $\ref{Lem3.4}$, we get
\begin{align*}
   & \left[\mathbb{E}\left[\left|\left(X^{\delta, x}_{s,\widetilde{t}}-X^{\delta,x}_{s,t}\right) - \left(X^{\delta, \widetilde{x}}_{s,\widetilde{t}}-X^{\delta, \widetilde{x}}_{s,t}\right)\right|^p\right]\right]^{1/p} \notag\\
  & =  \left[\mathbb{E} \left[\left| \int_t^{\widetilde{t} } \left ( \mu(X_{s,\max\{s,\delta(r)\}}^{\delta,x}) -\mu(X_{s,\max\{s,\delta(r)\}}^{\delta,\widetilde{x}}) \right) dr + \int_t^{\widetilde{t}} \left ( \sigma(X_{s,\max\{s,\delta(r)\}}^{\delta,x}) -\sigma(X_{s,\max\{s,\delta(r)\}}^{\delta,\widetilde{x}}) \right) dW_r \right. \right. \right.  \notag\\
  & \quad  \left. \left. \left. + \int_t^{\widetilde{t}} \int_{\bR_0} \left (\gamma (X_{s,\max\{s,\delta(r-)\}}^{\delta,x})-\gamma (X_{s,\max\{s,\delta(r-)\}}^{\delta,\widetilde{x}})\right) z \widetilde{N}(dr, dz) \right|^p\right] \right]^{1/p}   \notag\\
  & \leq  \left[\mathbb{E} \left[\left| \int_t^{\widetilde{t} } \left( \mu(X_{s,\max\{s,\delta(r)\}}^{\delta,x}) -\mu(X_{s,\max\{s,\delta(r)\}}^{\delta,\widetilde{x}}) \right) dr \right|^p\right] \right]^{1/p}  \notag\\
  & \quad + \left[\mathbb{E} \left[\left| \int_t^{\widetilde{t} } \left( \sigma(X_{s,\max\{s,\delta(r)\}}^{\delta,x}) -\sigma(X_{s,\max\{s,\delta(r)\}}^{\delta,\widetilde{x}})\right) dW_r \right|^p\right] \right]^{1/p}  \notag\\
   & \quad + \left[\mathbb{E} \left[\left| \int_t^{\widetilde{t}} \int_{\bR_0} \left(\gamma (X_{s,\max\{s,\delta(r-)\}}^{\delta,x})-\gamma (X_{s,\max\{s,\delta(r-)\}}^{\delta,\widetilde{x}})\right) z \widetilde{N}(dr, dz) \right|^p\right] \right]^{1/p}  \notag\\
    & \leq c  C_p  |\widetilde{t} -t|^{1/p} \sup_{r\in [t, \widetilde{t}]} \left[ \mathbb{E} \left[\left| X_{s,\max\{s,\delta(r)\}}^{\delta,x}-X_{s,\max\{s,\delta(r)\}}^{\delta,\widetilde{x}} \right|^p \right]\right]^{1/p} \notag\\
%  & \leq c  C_p  |t-\widetilde{t}|^{1/p} C_p |x-\widetilde{x}| \notag\\
 % & \leq c  C_p  |t-\widetilde{t}|^{1/p} C_p |x-\widetilde{x}|\notag\\
  & \le C_p |t-\widetilde{t}|^{1/p} |x-\widetilde{x}|.
\end{align*}
This shows $\textup{(iii)}.$\\
$\textup{(iv)}$ 
%Here, we need to estimate $\left[\mathbb{E}\left[\left|\left(X^{\iota, x}_{s,t}-X^{\delta,y}_{s,t}\right) - \left(X^{\iota, \widetilde{x}}_{s,t}-X^{\delta, \widetilde{y}}_{s,t}\right)\right|^p\right]\right]^{1/p}.$ 
First, using equation $\eqref{y01},$ we rewrite
\begin{align*}
   & \left(X^{\iota, x}_{s,t}-X^{\delta,y}_{s,t}\right) - \left(X^{\iota, \widetilde{x}}_{s,t}-X^{\delta, \widetilde{y}}_{s,t}\right) \notag\\
   & = (x-y)-(\widetilde{x}-\widetilde{y}) + \int_s^t \left[ \left(\mu(X_{s,r}^{\iota,x}) -\mu(X_{s,\max\{s,\delta(r)\}}^{\delta,y}) \right) - \left(\mu(X_{s,r}^{\iota,\widetilde{x}}) -\mu(X_{s,\max\{s,\delta(r)\}}^{\delta,\widetilde{y}}) \right)\right] dr  \notag\\
   & \quad  + \int_s^t \left[ \left(\sigma(X_{s,r}^{\iota,x}) -\sigma(X_{s,\max\{s,\delta(r)\}}^{\delta,y}) \right) - \left(\sigma(X_{s,r}^{\iota,\widetilde{x}}) -\sigma(X_{s,\max\{s,\delta(r)\}}^{\delta,\widetilde{y}}) \right)\right] dW_r  \notag\\
    & \quad + \int_s^t \int_{\bR_0} \left[ \left(\gamma(X_{s,r-}^{\iota,x}) -\gamma (X_{s,\max\{s,\delta(r-)\}}^{\delta,y}) \right) - \left(\gamma(X_{s,r-}^{\iota,\widetilde{x}}) -\gamma (X_{s,\max\{s,\delta(r-)\}}^{\delta,\widetilde{y}}) \right) \right] z \widetilde{N}(dr, dz).  
\end{align*}
Then, applying the triangle inequality and proceeding as in the proof of Lemma $\ref{Lem3.4}$, we get
\begin{align}
   &\left[\mathbb{E}\left[\left| \left(X^{\iota, x}_{s,t}-X^{\delta,y}_{s,t}\right) - \left(X^{\iota, \widetilde{x}}_{s,t}-X^{\delta, \widetilde{y}}_{s,t}\right)  
  \right|^p\right]\right]^{1/p} \notag\\
   & \le |(x-y)-(\widetilde{x}-\widetilde{y})|  \notag\\
   &\quad +\left[\mathbb{E}\left[\left|  \int_s^t \left[ \left( \mu(X_{s,r}^{\iota,x}) -\mu(X_{s,\max\{s,\delta(r)\}}^{\delta,y}) \right) - \left(\mu(X_{s,r}^{\iota,\widetilde{x}}) -\mu(X_{s,\max\{s,\delta(r)\}}^{\delta,\widetilde{y}}) \right) \right] dr  \right|^p\right]\right]^{1/p}   \notag\\
   & \quad+ \left[\mathbb{E}\left[\left| \int_s^t \left[ \left(\sigma(X_{s,r}^{\iota,x}) -\sigma(X_{s,\max\{s,\delta(r)\}}^{\delta,y}) \right) - \left(\sigma(X_{s,r}^{\iota,\widetilde{x}}) -\sigma(X_{s,\max\{s,\delta(r)\}}^{\delta,\widetilde{y}}) \right) \right] dW_r   \right|^p\right]\right]^{1/p} 
 \notag\\
    &\quad + \left[\mathbb{E}\left[\left| \int_s^t \int_{\bR_0} \left[ \left(\gamma(X_{s,r-}^{\iota,x}) -\gamma (X_{s,\max\{s,\delta(r-)\}}^{\delta,y}) \right) - \left(\gamma(X_{s,r-}^{\iota,\widetilde{x}}) -\gamma (X_{s,\max\{s,\delta(r-)\}}^{\delta,\widetilde{y}}) \right) \right] z \widetilde{N}(dr, dz)\right|^p\right]\right]^{1/p}\notag\\ 
    & \leq |(x-y)-(\widetilde{x}-\widetilde{y})|   \notag\\
& \quad + C_p   \left[ \int_s^t \max_{\xi \in \{\mu, \sigma, \gamma\}} \mathbb{E}\left[\left|\left(\xi(X^{\iota, x}_{s,r})-\xi(X^{\delta,y}_{s,\max \{s, \delta(r)\}})\right) - \left(\xi(X^{\iota, \widetilde{x}}_{s,r})-\xi(X^{\delta, \widetilde{y}}_{s,\max \{s, \delta(r)\}})\right)\right|^p\right] dr\right]^{1/p} \notag\\
    & = |(x-y)-(\widetilde{x}-\widetilde{y})|   \notag\\
 & \quad + C_p  \left[ \int_s^t \max_{\xi \in \{\mu, \sigma, \gamma\}} \left [\left [\mathbb{E}\left[\left|\left(\xi(X^{\iota, x}_{s,r})-\xi(X^{\delta,y}_{s,\max \{s, \delta(r)\}})\right) - \left(\xi(X^{\iota, \widetilde{x}}_{s,r})-\xi(X^{\delta, \widetilde{y}}_{s,\max \{s, \delta(r)\}})\right)\right|^p\right]\right]^{1/p}\right]^p dr\right]^{1/p}. 
\label{idelxy}
\end{align}
% Applying Lemma $\ref{Lem3.1}$, we obtain that
% \begin{align}
% &\left[\mathbb{E}\left[\left|\left(X^{\iota, x}_{s,t}-X^{\delta,y}_{s,t}\right) - \left(X^{\iota, \widetilde{x}}_{s,t}-X^{\delta, \widetilde{y}}_{s,t}\right)\right|^p\right]\right]^{1/p} \notag\\
%& \leq |(x-y)-(\widetilde{x}-\widetilde{y})|  + C_p  |t-s|^{1/p}  \notag\\
% & \quad \quad \times \sup_{r\in [s, t]} \max_{\xi \in \{\mu, \sigma, \gamma\}} \left[ \mathbb{E}\left[\left|\left(\xi(X^{\iota, x}_{s,r})-\xi(X^{\delta,y}_{s,\max \{s, \delta(r)\}})\right) - \left(\xi(X^{\iota, \widetilde{x}}_{s,t})-\xi(X^{\delta, \widetilde{y}}_{s,\max \{s, \delta(r)\}})\right)\right|^p\right] \right]^{1/p}.
% \label{v1}
% \end{align}
% We consider $\left[\mathbb{E}\left[\left|X^{\iota, x}_{s,t}-X^{\delta, y}_{s,\max\{s, \delta(t)\}}\right|^p\right]\right]^{1/p}$. 
Moreover, for $\xi \in \{\mu, \sigma, \gamma\}$, using  $\bf{A1}$, we have
\begin{align*}
& \left|\left(\xi(X^{\iota, x}_{s,r})-\xi(X^{\delta,y}_{s,\max\{s, \delta(r)\}})\right) - \left(\xi(X^{\iota, \widetilde{x}}_{s,r})-\xi(X^{\delta, \widetilde{y}}_{s,\max\{s, \delta(r)\}})\right)\right| \notag\\
& \leq c \left[\left(X^{\iota, x}_{s,r}-X^{\delta,y}_{s,\max\{s, \delta(r)\}}\right) - \left(X^{\iota, \widetilde{x}}_{s,r}-X^{\delta, \widetilde{y}}_{s,\max\{s, \delta(r)\}}\right)\right] \notag\\
& \quad + b \dfrac{\left|X^{\iota, x}_{s,r}-X^{\delta,y}_{s,\max\{s, \delta(r)\}}\right| +  \left|X^{\iota, \widetilde{x}}_{s,r}-X^{\delta, \widetilde{y}}_{s,\max\{s, \delta(r)\}}\right|}{2} \left|X^{\iota, x}_{s,r} - X^{\iota, \widetilde{x} }_{s,r}\right|.  
\end{align*}
Therefore, using the triangle inequality and H\"older inequality with $\frac{1}{m}+\frac{1}{n}=1$, we get
\begin{align}
& \left[\mathbb{E}\left[\left|\left(\xi(X^{\iota, x}_{s,r})-\xi(X^{\delta,y}_{s,\max\{s, \delta(r)\}})\right) - \left(\xi(X^{\iota, \widetilde{x}}_{s,r})-\xi(X^{\delta, \widetilde{y}}_{s,\max\{s, \delta(r)\}})\right)\right|^p\right]\right]^{1/p} \notag\\
& \leq c \left[\mathbb{E}\left[\left|\left(X^{\iota, x}_{s,r}-X^{\delta,y}_{s,\max\{s, \delta(r)\}}\right) - \left(X^{\iota, \widetilde{x}}_{s,r}-X^{\delta, \widetilde{y}}_{s,\max\{s, \delta(r)\}}\right)\right|^p\right]\right]^{1/p}   \notag\\
& \quad + b \dfrac{\left[\mathbb{E}\left[\left|X^{\iota, x}_{s,r}-X^{\delta,y}_{s,\max\{s, \delta(r)\}}\right|^{pm}\right]\right]^{1/pm} + \left[\mathbb{E} \left[\left|X^{\iota, \widetilde{x}}_{s,r}-X^{\delta, \widetilde{y}}_{s,\max\{s, \delta(r)\}}\right|^{pm}\right]\right]^{1/pm}}{2} \left[\mathbb{E}\left[\left|X^{\iota, x}_{s,r} - X^{\iota, \widetilde{x} }_{s,r}\right|^{pn}\right]\right]^{1/pn}    \notag\\
& \leq c \left( \left[\mathbb{E}\left[\left|\left(X^{\iota, x}_{s,\max\{s, \delta(r)\}}-X^{\delta,y}_{s,\max\{s, \delta(r)\}}\right) - \left(X^{\iota, \widetilde{x}}_{s,\max\{s, \delta(r)\}}-X^{\delta, \widetilde{y}}_{s,\max\{s, \delta(r)\}}\right)\right|^p\right]\right]^{1/p} \right.\notag\\
& \quad +  \left. \left[\mathbb{E}\left[\left|\left(X^{\iota, x}_{s,r}-X^{\iota,x}_{s,\max\{s, \delta(r)\}}\right) - \left(X^{\iota, \widetilde{x}}_{s,r}-X^{\iota, \widetilde{x}}_{s,\max\{s, \delta(r)\}}\right)\right|^p\right]\right]^{1/p} \right)  \notag\\ 
  & \quad + b \dfrac{\left[\mathbb{E}\left[\left|X^{\iota, x}_{s,r}-X^{\delta,y}_{s,\max\{s, \delta(r)\}}\right|^{pm}\right]\right]^{1/pm} + \left[\mathbb{E} \left[\left|X^{\iota, \widetilde{x}}_{s,r}-X^{\delta, \widetilde{y}}_{s,\max\{s, \delta(r)\}}\right|^{pm}\right]\right]^{1/pm}}{2} \left[\mathbb{E}\left[\left|X^{\iota, x}_{s,r} - X^{\iota, \widetilde{x} }_{s,r}\right|^{pn}\right]\right]^{1/pn}.
  %... & \leq c \left( \left[\mathbb{E}\left[\left|\left(X^{\iota, x}_{s,\max\{s, \delta(r)\}}-X^{\delta,y}_{s,\max\{s, \delta(r)\}}\right) - \left(X^{\iota, \widetilde{x}}_{s,\max\{s, \delta(r)\}}-X^{\delta, \widetilde{y}}_{s,\max\{s, \delta(r)\}}\right)\right|^p\right] \right]^{1/p} \right. \notag\\
  % &\quad  + \left. C_p  |r-\max\{s, \delta(r)\}|^{1/p} |x-\widetilde{x}| \right) \notag\\
  % & \quad +b \Biggl[ C_p |\delta|^{1/pm} (e^{2.5\overline{c}T})^{1/pm} \frac{ (V(x))^{1/pm}+(V(\widetilde{x}))^{1/pm}}{2}  + C_p \dfrac{|x-y|+ |\widetilde{x}-\widetilde{y}|}{2}  \Biggr] C_p |x-\widetilde{x}|...
  %& \leq c \left[\mathbb{E}\left[\left|\left(X^{\iota, x}_{s,\widetilde{t}}-X^{\delta,y}_{s,\widetilde{t}}\right) - \left(X^{i, \widetilde{x}}_{s,\widetilde{t}}-X^{\delta, \widetilde{y}}_{s,\widetilde{t}}\right)\right|^p\right]\right]^{1/p}  \notag\\
 % &\quad  +2(c^2+bc+b) \left[\sqrt{T} + 1+ \sqrt{m_2} \right]^3 e^{c^2\left[\sqrt{T} + 2+ \sqrt{m_2} \right]^2T} e^{\overline{c}T} \dfrac{\sqrt{V(x)}+\sqrt{V(\widetilde{x})}}{2}|\delta|^{1/2} |x-\widetilde{x}| \notag\\
  % &\quad + 2b e^{2c^2\left[\sqrt{T} + 1+ \sqrt{m_2} \right]^2T} \dfrac{(|x-y|+|\widetilde{x}-\widetilde{y}|)|x-\widetilde{x}|}{2}.
    \label{v2.1}
\end{align}
Now, applying the triangle inequality, (ii) of Lemma $\ref{Lem3.3}$, the statement (i) above, and Lemma $\ref{Lem3.4}$, we have for $q \ge 2,$ 
\begin{align}
&\left[\mathbb{E}\left[\left|X^{\iota, x}_{s,r}-X^{\delta, y}_{s,\max\{s, \delta(r)\}}\right|^q\right]\right]^{1/q} \notag\\ 
& \leq \left[\mathbb{E}\left[\left|X^{\iota, x}_{s,r}-X^{\iota, x}_{s,\max\{s, \delta(r)\}}\right|^q\right]\right]^{1/q} + \left[\mathbb{E}\left[\left|X^{\iota, x}_{s,\max\{s, \delta(r)\}} - X^{\delta, x}_{s,\max\{s, \delta(r)\}}\right|^q\right]\right]^{1/q} \notag\\ 
& \quad + \left[\mathbb{E}\left[\left|X^{\delta, x}_{s,\max\{s, \delta(r)\}} - X^{\delta, y }_{s,\max\{s, \delta(r)\}}\right|^q\right]\right]^{1/q} \notag\\ 
& \leq C_q |\max\{s, \delta(r)\}-r|^{1/q} (e^{2.5\overline{c}T}V(x))^{1/q} + C_q |\max\{s, \delta(r)\}-s|^{1/q} |\delta|^{1/q} (e^{2.5\overline{c}T}V(x))^{1/q} + C_q |x-y| \notag\\
& \leq C_q |\delta|^{1/q} (e^{2.5\overline{c}T}V(x))^{1/q} + C_q T^{1/q} |\delta|^{1/q} (e^{2.5\overline{c}T}V(x))^{1/q} + C_q |x-y| \notag\\ 
& \leq C_q |\delta|^{1/q} (e^{2.5\overline{c}T}V(x))^{1/q} + C_q |x-y|,
\label{vm}
%& \leq \left[ \sqrt{2}c+1 \right]\left[\sqrt{T} + 1+ \sqrt{m_2} \right]^3 e^{c^2\left[\sqrt{T} + 1+ \sqrt{m_2} \right]^2T} e^{\overline{c}T} \sqrt{V(x)}|\delta|^{1/2} + \sqrt{2} e^{c^2\left[\sqrt{T} + 1+ \sqrt{m_2} \right]^2T } |x-y|.
\end{align}
for a positive constant $C_q$.

Consequently, using 
$\eqref{v2.1},$ the aforementioned statement (iii), the estimate $\eqref{vm}$ above with $q=pm$, and Lemma $\ref{Lem3.4}$, we get
\begin{align}
  \label{v2}
& \left[\mathbb{E}\left[\left|\left(\xi(X^{\iota, x}_{s,r})-\xi(X^{\delta,y}_{s,\max\{s, \delta(r)\}})\right) - \left(\xi(X^{\iota, \widetilde{x}}_{s,r})-\xi(X^{\delta, \widetilde{y}}_{s,\max\{s, \delta(r)\}})\right)\right|^p\right]\right]^{1/p} \notag\\
    & \leq c \left( \left[\mathbb{E}\left[\left|\left(X^{\iota, x}_{s,\max\{s, \delta(r)\}}-X^{\delta,y}_{s,\max\{s, \delta(r)\}}\right) - \left(X^{\iota, \widetilde{x}}_{s,\max\{s, \delta(r)\}}-X^{\delta, \widetilde{y}}_{s,\max\{s, \delta(r)\}}\right)\right|^p\right] \right]^{1/p} \right. \notag\\
   &\quad  + \left. C_p  |r-\max\{s, \delta(r)\}|^{1/p} |x-\widetilde{x}| \right) \notag\\
   & \quad +b \Biggl[ C_{pm} |\delta|^{1/pm} (e^{2.5\overline{c}T})^{1/pm} \frac{ (V(x))^{1/pm}+(V(\widetilde{x}))^{1/pm}}{2}  + C_{pm} \dfrac{|x-y|+ |\widetilde{x}-\widetilde{y}|}{2}  \Biggr] C_{pn} |x-\widetilde{x}|\notag\\
   & \leq c \left( \left[\mathbb{E}\left[\left|\left(X^{\iota, x}_{s,\max\{s, \delta(r)\}}-X^{\delta,y}_{s,\max\{s, \delta(r)\}}\right) - \left(X^{\iota, \widetilde{x}}_{s,\max\{s, \delta(r)\}}-X^{\delta, \widetilde{y}}_{s,\max\{s, \delta(r)\}}\right)\right|^p\right] \right]^{1/p} +  C_p  |\delta|^{1/p} |x-\widetilde{x}| \right) \notag\\
   & \quad +b \Biggl[ C_p |\delta|^{1/pm} (e^{2.5\overline{c}T})^{1/pm} \frac{ (V(x))^{1/pm}+(V(\widetilde{x}))^{1/pm}}{2}  + C_p \dfrac{|x-y|+ |\widetilde{x}-\widetilde{y}|}{2}  \Biggr] C_p |x-\widetilde{x}|.
\end{align}
Therefore, using $\eqref{idelxy}$, $\eqref{v2}$ and Minkowski's inequality  $(\int_s^t |h(r)+g(r)|^p dr)^{1/p} \leq (\int_s^t |h(r)|^p dr)^{1/p} + (\int_s^t |g(r)|^p dr)^{1/p}$, valid for all $h, g \in L^p([s, t], \mathbb{R}),$ $p \geq 2$, we have
\begin{align*}
&\left[\mathbb{E}\left[\left|\left(X^{\iota, x}_{s,t}-X^{\delta,y}_{s,t}\right) - \left(X^{\iota, \widetilde{x}}_{s,t}-X^{\delta, \widetilde{y}}_{s,t}\right)\right|^p\right]\right]^{1/p} \notag\\
 & \leq |(x-y)-(\widetilde{x}-\widetilde{y})|   \notag\\
 &\quad + C_p   \Biggr[ \int_s^t  \Biggr[  c \left( \left[\mathbb{E}\left[\left|\left(X^{\iota, x}_{s,  \max\{s, \delta(r)\}  }-X^{\delta,y}_{s,\max \{s, \delta(r)\}}\right) - \left(X^{\iota, \widetilde{x}}_{s, \max\{s, \delta(r)\}}-X^{\delta, \widetilde{y}}_{s,\max \{s, \delta(r)\}}\right)\right|^p\right]\right]^{1/p} \right. \notag\\
  & \quad +\left. C_p  |\delta|^{1/p} |x-\widetilde{x}| \right) \notag\\
  & \quad + b \Biggl[ C_p |\delta|^{1/pm} (e^{2.5\overline{c}T})^{1/pm} \frac{ (V(x))^{1/pm}+(V(\widetilde{x}))^{1/pm}}{2}  + C_p \dfrac{|x-y|+ |\widetilde{x}-\widetilde{y}|}{2}  \Biggr] C_p |x-\widetilde{x}|    \Biggl ]^p dr \Biggl ] ^{1/p}\notag\\
& \leq |(x-y)-(\widetilde{x}-\widetilde{y})|  \notag\\
&\quad  + c C_p  \left[ \int_s^t  \mathbb{E}\left[\left|\left(X^{\iota, x}_{s,\max\{s, \delta(r)\}}-X^{\delta,y}_{s,\max \{s, \delta(r)\}}\right) - \left(X^{\iota, \widetilde{x}}_{s,\max\{s, \delta(r)\}}-X^{\delta, \widetilde{y}}_{s,\max \{s, \delta(r)\}}\right)\right|^p\right] dr\right]^{1/p} \notag\\ 
& \quad + C_p |t-s|^{1/p}  \Bigg[  C_p  |\delta |^{1/p} |x-\widetilde{x}| +b \Biggl(  C_p |\delta|^{1/pm} (e^{2.5\overline{c}T})^{1/pm} \frac{ (V(x))^{1/pm}+(V(\widetilde{x}))^{1/pm}}{2} \notag \\
& \quad  + C_p \dfrac{|x-y|+ |\widetilde{x}-\widetilde{y}|}{2}  \Biggr) C_p |x-\widetilde{x}|\Bigg]\notag\\
& \leq |(x-y)-(\widetilde{x}-\widetilde{y})|  \notag\\ 
& \quad + C_p |t-s|^{1/p} |x-\widetilde{x}| \Bigg[  |\delta |^{1/p}  +   |\delta|^{1/pm} (e^{2.5\overline{c}T})^{1/pm} \frac{ (V(x))^{1/pm}+(V(\widetilde{x}))^{1/pm}}{2}  +  \dfrac{|x-y|+ |\widetilde{x}-\widetilde{y}|}{2}    \Bigg]\notag\\
& \quad +   C_p  \left[ \int_s^t  \mathbb{E}\left[\left|\left(X^{\iota, x}_{s,\max \{s, \delta(r)\}}-X^{\delta,y}_{s,\max \{s, \delta(r)\}}\right) - \left(X^{\iota, \widetilde{x}}_{s,\max \{s, \delta(r)\}}-X^{\delta, \widetilde{y}}_{s,\max \{s, \delta(r)\}}\right)\right|^p\right] dr\right]^{1/p}.
 % &\leq |(x-y)-(\widetilde{x}-\widetilde{y})| +  C_p |t-s|^{1/p} \left[ \int_s^t  \mathbb{E}\left[\left|\left(X^{\iota, x}_{s,t}-X^{\delta,y}_{s,\max \{s, \delta(r)\}}\right) - \left(X^{\iota, \widetilde{x}}_{s,t}-X^{\delta, \widetilde{y}}_{s,\max \{s, \delta(r)\}}\right)\right|^p\right] dr\right]^{1/p} \notag\\ 
 % & \quad + C_p |t-s|^{1/p} |x-\widetilde{x}| |\delta |^{1/p} 
%   & \leq  c C_p |t-s|^{1/p} \left[ \int_s^t  \mathbb{E}\left[\left|\left(X^{\iota, x}_{s,t}-X^{\delta,y}_{s,\max \{s, \delta(r)\}}\right) - \left(X^{\iota, \widetilde{x}}_{s,t}-X^{\delta, \widetilde{y}}_{s,\max \{s, \delta(r)\}}\right)\right|^p\right] dr\right]^{1/p} \notag\\
%  &\quad + |(x-y)-(\widetilde{x}-\widetilde{y})| \notag\\
% & \quad + C_p |t-s|^{1/p}  \notag\\ 
% & \quad  \quad \times \Bigg[  c^2  C_p  |\delta|^{1/p} |x-\widetilde{x}|e^{c C_pT}  +b \Biggl[ \dfrac{ C^{\star}_{pm} |\delta|^{1/pm} (e^{2.5\overline{c}T})^{1/pm}((V(x))^{1/pm}+V(\widetilde{x})^{1/pm})}{2}\notag\\
%   & \quad \quad \quad+ \dfrac{e^{[c(t-s)^{\frac{pm-1}{pm}} + c(t-s)^{\frac{p/2-1}{p}} + c(m_p)^{1/pm}]T} c (C^{\star}_{pm})^{pm} (T)^{1/pm} |\delta|^{1/pm} (e^{2.5\overline{c}T})^{1/pm}\frac{(V(x))^{1/pm}+(V(\widetilde{x}))^{1/pm}}{2}}{2}     \notag\\
%   & \quad \quad\quad   + e^{c C^{\star}_{pm}T} \dfrac{|x-y|+ |\widetilde{x}-\widetilde{y}|}{2}  \Biggr] |x-\widetilde{x}|e^{c C_pT} \Bigg]. 
\end{align*}
Finally, applying Lemma $\ref{Cor2.1}$, we get
\begin{align*}
&\left[\mathbb{E}\left[\left|\left(X^{\iota, x}_{s,t}-X^{\delta,y}_{s,t}\right) - \left(X^{\iota, \widetilde{x}}_{s,t}-X^{\delta, \widetilde{y}}_{s,t}\right)\right|^p\right]\right]^{1/p} \notag\\
% & \leq 2^{1-\frac{1}{p}} e^{\frac{C_p^p |t-s| }{p}} \Bigg\{|(x-y)-(\widetilde{x}-\widetilde{y})| \notag\\
% & \quad + C_p |t-s|^{1/p} |x-\widetilde{x}| \Bigg[  |\delta|^{1/p}  +   |\delta|^{1/pm} (e^{2.5\overline{c}T})^{1/pm} \frac{  (V(x))^{1/pm}+(V(\widetilde{x}))^{1/pm}}{2}    +  \dfrac{|x-y|+ |\widetilde{x}-\widetilde{y}|}{2}  \Bigg] \Bigg\} \notag\\
& \leq 2^{1-\frac{1}{p}} e^{C_p |t-s|} \Bigg\{|(x-y)-(\widetilde{x}-\widetilde{y})| \notag\\
& \quad + C_p |t-s|^{1/p} |x-\widetilde{x}| \Bigg[  |\delta|^{1/p}  +   |\delta|^{1/pm} (e^{2.5\overline{c}T})^{1/pm} \frac{  (V(x))^{1/pm}+(V(\widetilde{x}))^{1/pm}}{2}    +  \dfrac{|x-y|+ |\widetilde{x}-\widetilde{y}|}{2}   \Bigg] \Bigg\},
% & \leq 2^{1-\frac{1}{p}} e^{C_p |t-s|} |(x-y)-(\widetilde{x}-\widetilde{y})|  + C_p e^{C_p |t-s|} |t-s|^{1/p} |x-\widetilde{x}| |\delta|^{1/p}  \notag\\
% & \quad + C_p e^{C_p |t-s|}  |t-s|^{1/p} |x-\widetilde{x}| |\delta|^{1/pm} (e^{2.5\overline{c}T})^{1/pm} \frac{  (V(x))^{1/pm}+(V(\widetilde{x}))^{1/pm}}{2}  \notag\\   
% & \quad + C_p e^{C_p |t-s|}  |t-s|^{1/p} |x-\widetilde{x}|
% \dfrac{|x-y|+ |\widetilde{x}-\widetilde{y}|}{2} .
% & \leq \sqrt{2} e^{2c^2 \left[\sqrt{T} + 1+ \sqrt{m_2} \right]^2T} |(x-y)-(\widetilde{x}-\widetilde{y})| \notag\\
% & \quad + 2\sqrt{2}(c^2+bc+b) \left[\sqrt{T} + 1+ \sqrt{m_2} \right]^4 e^{3c^2\left[\sqrt{T} + 1+ \sqrt{m_2} \right]^2T} e^{\overline{c}T} \dfrac{\sqrt{V(x)}+\sqrt{V(\widetilde{x})}}{2} |\delta|^{1/2} |x-\widetilde{x}| |t-s|^{1/2} \notag\\
% & \quad + 2\sqrt{2} b \left[\sqrt{T} + 1+ \sqrt{m_2} \right] e^{3c^2\left[\sqrt{T} + 1+ \sqrt{m_2} \right]^2T} \dfrac{(|x-y|+|\widetilde{x}-\widetilde{y}|)|x-\widetilde{x}|}{2} |t-s|^{1/2}.
\end{align*}
which implies the desired result.\\
% This shows $\textup{(iv)}.$\\
 $\textup{(v)}$ Without loss of generality, we suppose that $s<\tilde{s}$ and $t = \max\{s, \tilde{s}, t, \tilde{t}\}$.
% For $\delta \in \widetilde{\mathbb{S}},$ $ s, \widetilde{s}\in [0, T], t\in [s, T],$ $\widetilde{t} \in [\widetilde{s}, T] $ and $ x, \widetilde{x} \in \bR$, 
Using the triangle inequality, we have 
\begin{align}
&\left[\mathbb{E}\left[\left|\left(X^{\iota, x}_{s,t}-X^{\iota,\widetilde{x}}_{\widetilde{s},\widetilde{t}}\right) - \left(X^{\delta, x}_{s,t}-X^{\delta, \widetilde{x}}_{\widetilde{s},\widetilde{t}}\right)\right|^p\right]\right]^{1/p}\leq  R_1+R_2+R_3,
\label{R123}
\end{align}
where
\begin{align*}
&R_1:= \left[\mathbb{E}\left[\left| \left(X^{\iota, x}_{s,t}-X^{\delta, x}_{s,t}\right) - \left(X^{\iota, x}_{\widetilde{s},t}-X^{\delta, x}_{\widetilde{s},t}\right) \right| ^p\right]\right]^{1/p},\\
&R_2:= \left[\mathbb{E}\left[\left| \left(X^{\iota, x}_{\widetilde{s},t}-X^{\delta, x}_{\widetilde{s},t}\right)- \left(X^{\iota, x}_{\widetilde{s},\widetilde{t}}-X^{\delta, x}_{\widetilde{s},\widetilde{t}}\right) \right| ^p\right]\right]^{1/p},\\
&R_3:= \left[\mathbb{E}\left[\left| \left(X^{\iota, x}_{\widetilde{s},\widetilde{t}}-X^{\delta, x}_{\widetilde{s},\widetilde{t}}\right) - \left(X^{\iota, \widetilde{x}}_{\widetilde{s},\widetilde{t}}-X^{\delta, \widetilde{x}}_{\widetilde{s},\widetilde{t}}\right) \right| ^p\right]\right]^{1/p}.
\end{align*}
{\it Estimation of  $R_1$}:
We estimate $R_1$ by dividing into 3 cases as in the proof of the term $T_1$ above.

\noindent\underline{Case 1}: $\widetilde{s}$ \text{ is a grid point.}
% Let $\delta \in \widetilde{\mathbb{S}},$ $x \in \bR,$ $s\in [0, T],$ $\widetilde{s} \in \delta([0, T])\cap [s, T],$ $t\in [s, T], \widetilde{t} \in [\widetilde{s}, T]$.

Applying the Markov property, the statement of (iv) above, we have
\begin{align*}
R_1&=\left[\mathbb{E}\left[\left|\left(X^{\iota, x}_{s,t}-X^{\delta,x}_{s,t}\right) - \left(X^{\iota, x}_{\widetilde{s},t}-X^{\delta, x}_{\widetilde{s},t}\right)\right|^p\right]\right]^{1/p} \notag\\ 
 & = \left(\mathbb{E}\left[\left(\left[\mathbb{E}\left[\left|\left(X^{\iota, \mathbf{x}}_{\widetilde{s},t}-X^{\delta,\mathbf{y}}_{\widetilde{s},t}\right) - \left(X^{\iota, x}_{\widetilde{s},t}-X^{\delta, x}_{\widetilde{s},t}\right)\right|^p\right]\right] ^{1/p}\Bigg| _{\mathbf{x}= X_{s, \widetilde{s}}^{\iota, x}, \mathbf{y}= X_{s, \widetilde{s}}^{\delta, x}} \right)^p\right]\right)^{1/p} \notag\\
 &\leq \Bigg(\mathbb{E}\Bigg[ \Bigg(  2^{1-\frac{1}{p}} e^{C_p T} |(\mathbf{x} - \mathbf{y}) - (x-x)|   + C_p e^{C_p T} |t-\widetilde{s}|^{1/p} |\mathbf{x}-x| |\delta |^{1/p} \notag\\
& \quad + C_p e^{C_p T} |t-\widetilde{s}|^{1/p} |\mathbf{x}-x||\delta|^{1/pm} (e^{2.5\overline{c}T})^{1/pm} \frac{   (V(\mathbf{x}))^{1/pm}+(V(x))^{1/pm}}{2} \notag \\
& \quad   + C_p e^{C_p T} |t-\widetilde{s}|^{1/p} |\mathbf{x}-x|  \dfrac{|\mathbf{x}-\mathbf{y}|+ |x-x|}{2}    \Bigg| _{\mathbf{x}= X_{s, \widetilde{s}}^{\iota, x}, 
\mathbf{y}= X_{s, \widetilde{s}}^{\delta, x}} \Bigg)^p\Bigg]\Bigg)^{1/p} \notag\\
% & = \Bigg(\mathbb{E}\Bigg[\Bigg| \Bigg(   2^{1-\frac{1}{p}} e^{C_p T} |\mathbf{x} - \mathbf{y}|  + C_p e^{C_p T} |t-\widetilde{s}|^{1/p} |\mathbf{x}-x| |\delta |^{1/p} \notag\\ 
%  & \quad + C_p e^{C_p T} |t-\widetilde{s}|^{1/p} |\mathbf{x}-x||\delta|^{1/pm} (e^{2.5\overline{c}T})^{1/pm} \frac{   (V(\mathbf{x}))^{1/pm}+(V(x))^{1/pm}}{2} \notag \\
% & \quad  + C_p e^{C_p T} |t-\widetilde{s}|^{1/p} |\mathbf{x}-x| \dfrac{|\mathbf{x}-\mathbf{y}|}{2}  \Bigg] \Bigg) \Bigg| _{\mathbf{x}= X_{s, \widetilde{s}}^{\iota, x}, 
% \mathbf{y}= X_{s, \widetilde{s}}^{\delta, x}}\Bigg|^p\Bigg]\Bigg)^{1/p}
% \notag\\
& = \Bigg(\mathbb{E}\Bigg[\Bigg|    2^{1-\frac{1}{p}} e^{C_p T} \left|X_{s, \widetilde{s}}^{\iota, x} - X_{s, \widetilde{s}}^{\delta, x}\right|  + C_p e^{C_p T} |t-\widetilde{s}|^{1/p} \left|X_{s, \widetilde{s}}^{\iota, x}-x \right| |\delta |^{1/p} \notag\\ 
 & \quad + C_p e^{C_p T} |t-\widetilde{s}|^{1/p} |X_{s, \widetilde{s}}^{\iota, x}-x||\delta|^{1/pm} (e^{2.5\overline{c}T})^{1/pm} \frac{   (V(X_{s, \widetilde{s}}^{\iota, x}))^{1/pm}+(V(x))^{1/pm}}{2} \notag \\
& \quad  + C_p e^{C_p T} |t-\widetilde{s}|^{1/p} |X_{s, \widetilde{s}}^{\iota, x}-x| \dfrac{|X_{s, \widetilde{s}}^{\iota, x}-X_{s, \widetilde{s}}^{\delta, x}|}{2}   \Bigg|^p\Bigg]\Bigg)^{1/p}.
 \end{align*}
Using the triangle inequality and H\"older's inequality with $\frac{1}{\kappa_1}+\frac{1}{\kappa_2}=1,$
\begin{align*}
R_1 & \leq 2^{1-\frac{1}{p}} e^{C_p T}  \left[\mathbb{E}\left[\left|X^{\iota, x}_{s,\widetilde{s}}-X^{\delta,x}_{s,\widetilde{s}}\right|^p\right]\right]^{1/p}  + C_p e^{C_p T} |t-\widetilde{s}|^{1/p}   |\delta|^{1/p} \left[\mathbb{E}\left[\left|X^{\iota, x}_{s,\widetilde{s}} - x \right|^p\right]\right]^{1/p}  \notag\\
& \quad + C_p e^{C_p T} |t-\widetilde{s}|^{1/p} |\delta|^{1/pm} (e^{2.5\overline{c}T})^{1/pm} \left[\mathbb{E}\left[ \left|X^{\iota, x}_{s,\widetilde{s}} - x \right|^{p\kappa_1}\right]\right]^{1/p\kappa_1} \dfrac{   \left[\mathbb{E}\left[ \left|V(X^{\iota, x}_{s,\widetilde{s}}) \right|^{\kappa_2 /m}\right]\right]^{1/p\kappa_2}}{2} \notag\\
&\quad + C_p e^{C_p T} |t-\widetilde{s}|^{1/p} |\delta|^{1/pm} (e^{2.5\overline{c}T})^{1/pm} \left[\mathbb{E}\left[\left|X^{\iota, x}_{s,\widetilde{s}} - x \right|^{p}\right]\right]^{1/p} \dfrac{   (V(x))^{1/pm}}{2}\notag\\
  & \quad   +  C_p e^{C_p T} |t-\widetilde{s}|^{1/p}  \left[\mathbb{E}\left[\left|X^{\iota, x}_{s,\widetilde{s}} - x\right|^{p\kappa_1}\right]\right]^{1/p\kappa_1}\dfrac{\left[\mathbb{E}\left[ \left|X^{\iota, x}_{s,\widetilde{s}} - X^{\delta, x}_{s,\widetilde{s}}\right|^{p\kappa_2}\right]\right]^{1/p\kappa_2}}{2} .
  \end{align*}
%   Next, using equation $\eqref{y01}$ and triangle inequality,  for $q\in \{ p, p \kappa_1\}$, we have
% \begin{align}
% \left[\mathbb{E}\left[\left|X^{\iota, x}_{s,\widetilde{s}} - x \right|^q\right]\right]^{1/q} & \leq \left[\mathbb{E} \left[\left| \int_s^{\widetilde{s}} \mu(X_{s,\max\{s,\delta(r)\}}^{\iota,x}) dr \right|^q \right]\right]^{1/q} + \left[\mathbb{E} \left[\left| \int_s^{\widetilde{s}} \sigma(X_{s,\max\{s,\delta(r)\}}^{\iota,x}) dW_r \right|^q \right]\right]^{1/q} \notag \\
% & \quad + \left[\mathbb{E} \left[\left| \int_s^{\widetilde{s}} \int _{\bR_0}\gamma (X_{s,\max\{s,\delta(r-)\}}^{\iota,x}) z \widetilde{N}(dr, dz) \right|^q\right]\right]^{1/q}.
% \end{align}
Then, applying the aforementioned statement (i), proceeding as in the proof of $\eqref{18}$, and proceeding as in the proof (i) of Lemma $\ref{Lem3.3}$, we get
  \begin{align}
  \label{vi3}
R_1
 & \le 2^{1-\frac{1}{p}} e^{C_p T}  C_p |\widetilde{s}-s|^{1/p} |\delta|^{1/p} (e^{2.5\overline{c}T}V(x))^{1/p}  + C_p e^{C_p T} |t-\widetilde{s}|^{1/p}  |\delta|^{1/p} C_p|\widetilde{s} -s|^{1/p} (e^{2.5\overline{c}T}V(x))^{1/p}   \notag\\
& \quad  +C_p e^{C_p T} |t-\widetilde{s}|^{1/p} |\delta|^{1/pm} (e^{2.5\overline{c}T})^{1/pm} C_p |\widetilde{s} -s|^{1/p\kappa_1} (e^{2.5\overline{c}T}V(x))^{1/p\kappa_1}  \dfrac{(e^{2.5\overline{c} |\widetilde{s}-s|})^{1/p\kappa_2}((V(x))^{1/pm}}{2} \notag\\
& \quad  +C_p e^{C_p T} |t-\widetilde{s}|^{1/p} |\delta|^{1/pm} (e^{2.5\overline{c}T})^{1/pm} C_p |\widetilde{s} -s|^{1/p} (e^{2.5\overline{c}T}V(x))^{1/p}   \dfrac{(V(x))^{1/pm}}{2} \notag\\
& \quad  +C_p e^{C_p T} |t-\widetilde{s}|^{1/p}  C_p |\widetilde{s} -s|^{1/p\kappa_1} (e^{2.5\overline{c}T}V(x))^{1/p\kappa_1}  C_p |\widetilde{s} -s|^{1/p\kappa_2} |\delta|^{1/p \kappa_2} (e^{2.5\overline{c}T}V(x))^{1/p\kappa_2}  \notag\\
&\leq C_p e^{C_p T}   |\widetilde{s}-s|^{1/p} |\delta|^{1/p} (e^{2.5\overline{c}T}V(x))^{1/p}  + C_p e^{C_p T} T^{1/p}  |\delta|^{1/p} |\widetilde{s} -s|^{1/p} (e^{2.5\overline{c}T}V(x))^{1/p}   \notag\\
& \quad  + C_p e^{C_p T} T^{1/p} |\delta|^{1/pm} (e^{2.5\overline{c}T})^{1/pm}  |\widetilde{s} -s|^{1/p\kappa_1} (e^{2.5\overline{c}T}V(x))^{1/p\kappa_1} \dfrac{(e^{2.5\overline{c}T})^{1/p\kappa_2}(V(x))^{1/pm}}{2} \notag\\
& \quad  + C_p e^{C_p T} T^{1/p} |\delta|^{1/pm} (e^{2.5\overline{c}T})^{1/pm}  |\widetilde{s} -s|^{1/p} (e^{2.5\overline{c}T}V(x))^{1/p}  \dfrac{(V(x))^{1/pm}}{2} \notag\\
& \quad  +C_p e^{C_p T} T^{1/p}  |\widetilde{s} -s|^{1/p\kappa_1} (e^{2.5\overline{c}T}V(x))^{1/p\kappa_1}   |\widetilde{s} -s|^{1/p\kappa_2} |\delta|^{1/p \kappa_2} (e^{2.5\overline{c}T}V(x))^{1/p\kappa_2}  \notag\\
  & \leq  C_p e^{C_p T}  (e^{5\overline{c}T})^{1/p } |\delta|^{\frac{1}{p(m \vee \kappa_2)}}  |\widetilde{s}-s|^{1/p\kappa_1}(V(x))^{2/p}.
 \end{align}
 %without loss of generality, we suppose that $m=\max\{m, \kappa_2\}.$

\noindent\underline{Case 2}: \text{$\widetilde{s}$ is not a grid point and there is no grid point on $[s, \widetilde{s}]$}.

Let $\overline{s}$ be the smallest grid point on $[\widetilde{s}, t].$
% Let $\delta \in \widetilde{\mathbb{S}},$ $x \in \bR,$ $s\in [0, T],$ $\widetilde{s} \in [s, T],$ $t\in [\widetilde{s},T],$ $ \overline{s} \in [\widetilde{s},T],$ $\delta([0, T])\cap (s, \overline{s}) = \emptyset $, $\delta([0, T])\cap (s, \widetilde{s}) = \emptyset $ and  $\overline{s} \in \delta([0, T])\cap [\widetilde{s}, T] $. 
Thus, there is no grid point on $[s, \overline{s}).$ Applying the Markov property, the statement (iv) above, the triangle inequality and H\"older's inequality with  $\frac{1}{\kappa_1}+\frac{1}{\kappa_2}=1,$ we get
\begin{align*}
R_1&=\left(\mathbb{E}\left[\left(\left[\mathbb{E}\left[\left|\left(X^{\iota, \mathbf{x}}_{\overline{s},t}-X^{\delta,\mathbf{y}}_{\overline{s},t}\right) - \left(X^{\iota, \mathbf{\widetilde{x}}}_{\overline{s},t}-X^{\delta, \mathbf{\widetilde{y}}}_{\overline{s},t}\right)\right|^p\right]\right]^{1/p}\Bigg| _{\mathbf{x}= X_{s, \overline{s}}^{\iota, x}, \mathbf{y}= X_{s, \overline{s}}^{\delta, x}, \mathbf{\widetilde{x}}= X_{\widetilde{s}, \overline{s}}^{\iota, x}, \mathbf{\widetilde{y}}= X_{\widetilde{s}, \overline{s}}^{\delta, x}}\right)^p\right]\right)^{1/p} \notag\\
& \leq  \Bigg( \mathbb{E}\Bigg[\Bigg(  2^{1-\frac{1}{p}} e^{C_p T} |(\mathbf{x} - \mathbf{y}) - (\mathbf{\widetilde{x}}-\mathbf{\widetilde{y}})| + C_p e^{C_p T} |t-\overline{s}|^{1/p} |\mathbf{x}-\mathbf{\widetilde{x}}| |\delta|^{1/p}  \notag\\ 
& \quad + C_p e^{C_p T} |t-\overline{s}|^{1/p} |\mathbf{x}-\mathbf{\widetilde{x}}|   |\delta|^{1/pm} (e^{2.5\overline{c}T})^{1/pm} \frac{  (V(\mathbf{x}))^{1/pm}+(V(\mathbf{\widetilde{x}}))^{1/pm}}{2} \notag\\
& \quad + C_p e^{C_p T} |t-\overline{s}|^{1/p} |\mathbf{x}-\mathbf{\widetilde{x}}| \dfrac{|\mathbf{x}-\mathbf{y}|+ |\mathbf{\widetilde{x}}-\mathbf{\widetilde{y}}|}{2}   \Bigg|_{\mathbf{x}= X_{s, \overline{s}}^{\iota, x}, \mathbf{y}= X_{s, \overline{s}}^{\delta, x}, \mathbf{\widetilde{x}}= X_{\widetilde{s}, \overline{s}}^{\iota, x}, \mathbf{\widetilde{y}}= X_{\widetilde{s}, \overline{s}}^{\delta, x}}\Bigg)^p\Bigg]\Bigg)^{1/p} \notag\\
 %& \leq  \Bigg( \mathbb{E}\Bigg[\Bigg(  2^{1-\frac{1}{p}} e^{C_p T} |(\mathbf{x} - \mathbf{y}) - (\mathbf{\widetilde{x}}-\mathbf{\widetilde{y}})| + C_p e^{C_p T} |t-s|^{1/p} |\mathbf{x}-\mathbf{\widetilde{x}}| |\delta|^{1/p}  \notag\\
%& \quad + C_p e^{C_p T} |t-s|^{1/p} |\mathbf{x}-\mathbf{\widetilde{x}}|   |\delta|^{1/pm} (e^{2.5\overline{c}T})^{1/pm} \frac{  (V(\mathbf{x}))^{1/pm}+(V(\mathbf{\widetilde{x}}))^{1/pm}}{2} \notag\\
%& \quad + C_p e^{C_p T} |t-s|^{1/p} |\mathbf{x}-\mathbf{\widetilde{x}}| \dfrac{|\mathbf{x}-\mathbf{y}|+ |\mathbf{\widetilde{x}}-\mathbf{\widetilde{y}}|}{2}   \Bigg|_{\mathbf{x}= X_{s, \overline{s}}^{\iota, x}, \mathbf{y}= X_{s, \overline{s}}^{\delta, x}, \mathbf{\widetilde{x}}= X_{\widetilde{s}, \overline{s}}^{\iota, x}, \mathbf{\widetilde{y}}= X_{\widetilde{s}, \overline{s}}^{\delta, x}}\Bigg)^p\Bigg]\Bigg)^{1/p} \notag\\
 & \leq   2^{1-\frac{1}{p}} e^{C_p T} \left[\mathbb{E}\left[\left|\left(X^{\iota, x}_{s,\overline{s}}-X^{\delta,x}_{s,\overline{s}}\right) - \left(X^{\iota, x}_{\widetilde{s},\overline{s}}-X^{\delta, x}_{\widetilde{s},\overline{s}}\right)\right|^p\right]\right]^{1/p}   + C_p e^{C_p T} |t-s|^{1/p} |\delta|^{1/p} \left[\mathbb{E}\left[\left|X^{\iota, x}_{s,\overline{s}}-X^{\iota, x}_{\widetilde{s},\overline{s}}\right|^p\right]\right]^{1/p} \notag\\
& \quad + C_p e^{C_p T} |t-s|^{1/p} |\delta|^{1/pm} (e^{2.5\overline{c}T})^{1/pm} \left[\mathbb{E}\left[|X^{\iota, x}_{s,\overline{s}} - X^{\iota, x}_{\widetilde{s},\overline{s}}|^{p\kappa_1}\right]\right]^{1/p\kappa_1} \notag\\
& \qquad \times  \dfrac{\left[\mathbb{E}\left[|V(X^{\iota, x}_{s,\overline{s}})|^{\kappa_2/m}\right]\right]^{1/p\kappa_2}+\left[\mathbb{E}\left[|V(X^{\iota, x}_{\widetilde{s},\overline{s}})|^{\kappa_2/m}\right]\right]^{1/p\kappa_2}}{2}\notag\\
& \quad + C_p e^{C_p T} |t-s|^{1/p} \left[\mathbb{E}\left[|X^{\iota, x}_{s,\overline{s}} - X^{\iota, x}_{\widetilde{s},\overline{s}}|^{p\kappa_1}\right]\right]^{1/p\kappa_1} \notag\\
  & \quad \quad  \times \dfrac{\left[\mathbb{E}\left[|X^{\iota, x}_{s,\overline{s}} - X^{\delta, x}_{s,\overline{s}}|^{p\kappa_2}\right]\right]^{1/p\kappa_2}+ \left[\mathbb{E}\left[|X^{\iota, x}_{\widetilde{s},\overline{s}} - X^{\delta, x}_{\widetilde{s},\overline{s}}|^{p\kappa_2}\right]\right]^{1/p\kappa_2}}{2}   .
  \end{align*}
Next, using (ii) of Lemma $\ref{Lem3.5},$ then (i) of Lemma $\ref{Lem3.5},$ proceeding as in the
proof (i) of Lemma $\ref{Lem3.3},$ and the aforementioned statement (i), we obtain  
  \begin{align}
R_1 
  & \leq   2^{1-\frac{1}{p}} e^{C_p T}  C_p |\delta|^{1/p} |\widetilde{s} -s|^{1/p} (e^{2.5\overline{c}T}V(x))^{1/p}  + C_p e^{C_p T}   |t-s|^{1/p} |\delta|^{1/p} C_p |\widetilde{s}-s|^{1/p} (e^{2.5\overline{c}T}V(x))^{1/p}  \notag\\
& \quad + C_p e^{C_p T} |t-s|^{1/p} |\delta|^{1/pm} (e^{2.5\overline{c}T})^{1/pm}  C_p |\widetilde{s}-s|^{1/p\kappa_1} (e^{2.5\overline{c}T}V(x))^{1/p\kappa_1}    \notag\\
& \qquad \quad \times \dfrac{  (e^{2.5\overline{c}|\overline{s}-s|})^{1/p\kappa_2} (V(x))^{1/pm} +(e^{2.5\overline{c}|\overline{s}-\widetilde{s}|})^{1/p\kappa_2} (V(x))^{1/pm} }{2}\notag\\
& \quad + C_p e^{C_p T} |t-s|^{1/p}   C_p |\widetilde{s}-s|^{1/p\kappa_1} (e^{2.5\overline{c}T}V(x))^{1/p\kappa_1} \notag\\ 
  &\qquad \quad \times  \dfrac{ C_p |\overline{s}-s|^{1/p\kappa_2} |\delta|^{1/p\kappa_2} (e^{2.5\overline{c}T}V(x))^{1/p\kappa_2} + C_p|\overline{s}-\widetilde{s}|^{1/p\kappa_2} |\delta|^{1/p\kappa_2} (e^{2.5\overline{c}T}V(x))^{1/p\kappa_2}}{2}    \notag\\
  & \leq  C_p e^{C_p T}   |\delta|^{1/p} |\widetilde{s} -s|^{1/p} (e^{2.5\overline{c}T}V(x))^{1/p}  + C_p e^{C_p T}   T^{1/p} |\delta|^{1/p}  |\widetilde{s}-s|^{1/p} (e^{2.5\overline{c}T}V(x))^{1/p}  \notag\\
& \quad + C_p e^{C_p T} T^{1/p} |\delta|^{1/pm} (e^{2.5\overline{c}T})^{1/pm}   |\widetilde{s}-s|^{1/p\kappa_1} (e^{2.5\overline{c}T}V(x))^{1/p\kappa_1}    \notag\\
& \qquad \quad \times \dfrac{  (e^{2.5\overline{c}T})^{1/p\kappa_2} (V(x))^{1/pm} +(e^{2.5\overline{c}T})^{1/p\kappa_2} (V(x))^{1/pm} }{2}\notag\\
& \quad + C_p e^{C_p T} T^{1/p}  |\widetilde{s}-s|^{1/p\kappa_1} (e^{2.5\overline{c}T}V(x))^{1/p\kappa_1} \notag\\ 
  &\qquad \quad \times  \dfrac{  T^{1/p\kappa_2} |\delta|^{1/p\kappa_2} (e^{2.5\overline{c}T}V(x))^{1/p\kappa_2} +  T^{1/p\kappa_2} |\delta|^{1/p\kappa_2} (e^{2.5\overline{c}T}V(x))^{1/p\kappa_2}}{2}    \notag\\
 % & \leq C_p e^{C_p T} (e^{5\overline{c}T})^{1/p} |t-s|^{1/p}  |s-\widetilde{s}|^{1/p\kappa_1}|\delta|^{1/pm} (V(x))^{2/p} \notag\\
  & \leq C_p e^{C_p T}  (e^{5\overline{c}T})^{1/p} |\delta|^{\frac{1}{p(m \vee \kappa_2)}} |\widetilde{s}-s|^{1/p\kappa_1}  (V(x))^{2/p}.
  \label{vi2}
\end{align}
% without loss of generality, we suppose that $m=\max\{m, \kappa_2\}.$

\noindent\underline{Case 3}: \text{$\widetilde{s}$ is not a grid point and there is a grid point on $[s, \widetilde{s}]$}.

Let $\underline{s}$ be the largest grid point on $[s,\widetilde{s}].$ 
Then, using the fact that $\underline{s}$  is the grid point and proceeding as in the proof of $\eqref{vi3}$ of Case $1$, we get
 \begin{align*}  \left[\mathbb{E}\left[\left|\left(X^{\iota, x}_{s,t}-X^{\delta,x}_{s,t}\right) - \left(X^{\iota, x}_{\underline{s},t}-X^{\delta, x}_{\underline{s},t}\right)\right|^p\right]\right]^{1/p} &\leq C_p e^{C_p T} (e^{5\overline{c}T})^{1/p} |\delta|^{\frac{1}{p(m \vee \kappa_2)}}  |\underline{s}-s|^{1/p\kappa_1}(V(x))^{2/p}\notag\\
 &\leq C_p e^{C_p T} (e^{5\overline{c}T})^{1/p} |\delta|^{\frac{1}{p(m \vee \kappa_2)}}  |\widetilde{s} - s|^{1/p\kappa_1}(V(x))^{2/p}.
\end{align*}
Next, using the fact that there is no grid point on  $(\underline{s}, \widetilde{s}]$ and proceeding as in the proof of $\eqref{vi2}$ of Case 2, we get
\begin{align*}
\left[\mathbb{E}\left[\left|\left(X^{\iota, x}_{\underline{s},t}-X^{\delta,x}_{\underline{s},t}\right) - \left(X^{\iota, x}_{\widetilde{s},t}-X^{\delta, x}_{\widetilde{s},t}\right)\right|^p\right]\right]^{1/p} & \leq  C_p e^{C_p T}(e^{5\overline{c}T})^{1/p} |\delta|^{\frac{1}{p(m \vee \kappa_2)}}  |\widetilde{s} -\underline{s}|^{1/p\kappa_1}  (V(x))^{2/p}\notag\\
& \leq  C_p e^{C_p T}(e^{5\overline{c}T})^{1/p} |\delta|^{\frac{1}{p(m \vee \kappa_2)}}  |\widetilde{s} -s|^{1/p\kappa_1}  (V(x))^{2/p}.
\end{align*}
Therefore, combining all the computations above yields
\begin{align*}
R_1&=\left[\mathbb{E}\left[\left|\left(X^{\iota, x}_{s,t}-X^{\delta,x}_{s,t}\right) - \left(X^{\iota, x}_{\widetilde{s},t}-X^{\delta, x}_{\widetilde{s},t}\right)\right|^p\right]\right]^{1/p} \notag\\ 
 & \leq \left[\mathbb{E}\left[\left|\left(X^{\iota, x}_{s,t}-X^{\delta,x}_{s,t}\right) - \left(X^{\iota, x}_{\underline{s},t}-X^{\delta, x}_{\underline{s},t}\right)\right|^p\right]\right]^{1/p} + \left[\mathbb{E}\left[\left|\left(X^{\iota, x}_{\underline{s},t}-X^{\delta,x}_{\underline{s},t}\right) - \left(X^{\iota, x}_{\widetilde{s},t}-X^{\delta, x}_{\widetilde{s},t}\right)\right|^p\right]\right]^{1/p} \notag\\ 
& \leq C_p e^{C_p T}  (e^{5\overline{c}T})^{1/p} |\delta|^{\frac{1}{p(m \vee \kappa_2)}} |\widetilde{s}-s|^{1/p\kappa_1}  (V(x))^{2/p}  . 
\end{align*}
% Since $\underline{s}$ is the grid point and there is no grid point on $(\underline{s},\widetilde{s}]$,
% applying the triangle inequality, $\eqref{vi3}$ and $\eqref{vi2}$ we get
% \begin{align*}
% R_1&=\left[\mathbb{E}\left[\left|\left(X^{\iota, x}_{s,t}-X^{\delta,x}_{s,t}\right) - \left(X^{\iota, x}_{\widetilde{s},t}-X^{\delta, x}_{\widetilde{s},t}\right)\right|^p\right]\right]^{1/p} \notag\\ 
%     & \leq \left[\mathbb{E}\left[\left|\left(X^{\iota, x}_{s,t}-X^{\delta,x}_{s,t}\right) - \left(X^{\iota, x}_{\underline{s},t}-X^{\delta, x}_{\underline{s},t}\right)\right|^p\right]\right]^{1/p} + \left[\mathbb{E}\left[\left|\left(X^{\iota, x}_{\underline{s},t}-X^{\delta,x}_{\underline{s},t}\right) - \left(X^{\iota, x}_{\widetilde{s},t}-X^{\delta, x}_{\widetilde{s},t}\right)\right|^p\right]\right]^{1/p} \notag\\ 
%     & \leq  C_p e^{C_p T} (e^{5\overline{c}T})^{1/p} |\delta|^{1/pm}  |s-\underline{s}|^{1/p\kappa_1}(V(x))^{2/p}\notag \\
%     & \quad + C_p e^{C_p T}  |s-\widetilde{s}|^{1/p\kappa_1} |\delta|^{1/pm} (V(x))^{2/p}\notag\\
%     & \leq C_p e^{C_p T}  (e^{5\overline{c}T})^{1/p} |s-\widetilde{s}|^{1/p\kappa_1} |\delta|^{1/pm} (V(x))^{2/p}  . 
% \end{align*}
Consequently, we obtain in all cases
 \begin{align} R_1 \leq C_p e^{C_p T}  (e^{5\overline{c}T})^{1/p} |\delta|^{\frac{1}{p(m \vee \kappa_2)}} |\widetilde{s}-s|^{1/p\kappa_1}  (V(x))^{2/p}.
 \label{R1}
 \end{align} 
{\it Estimation of  $R_2$}: 
% We consider $$\left[\mathbb{E}\left[\left|\left(X^{\iota, x}_{s,\max\{s, t\}}-X^{\delta, x}_{s,\max\{s, t\}}\right)- \left(X^{\iota, x}_{s,\max\{s,\widetilde{s}, t,\widetilde{t}\}}-X^{\delta, x}_{s,\max\{s,\widetilde{s}, t, \widetilde{t}\}}\right) \right| ^p\right]\right]^{1/p}.$$ 
Using equation $\eqref{y01}$, the triangle inequality, Lemma $\ref{Lem3.2}$, (ii) of Lemma $\ref{Lem3.3}$, and the statement (i) above,
%with $\widetilde{s} <t,$ $ \widetilde{s} < \widetilde{t}$
 we get
\begin{align}
R_2& = \left[\mathbb{E}\left[\left|\left(X^{\iota, x}_{\widetilde{s},t}-X^{\delta,x}_{\widetilde{s},t}\right) - \left(X^{\iota, x}_{\widetilde{s},\widetilde{t}}-X^{\delta, x}_{\widetilde{s},\widetilde{t}}\right)\right|^p\right]\right]^{1/p} \notag\\
& = \left[\mathbb{E}\left[ \left|\int^t_{\widetilde{t}}  \left( \mu(X_{\widetilde{s}, r}^{\iota, x})-\mu(X_{\widetilde{s}, \max \{\widetilde{s}, \delta(r) \}}^{\delta, x}) \right) dr + \int^t_{\widetilde{t}} \left( \sigma(X_{\widetilde{s}, r}^{\iota, x}) -\sigma(X_{\widetilde{s}, \max \{\widetilde{s}, \delta(r) \}}^{\delta, x})\right) dW_r \right. \right. \right. \notag\\
& \quad \quad \left. \left. \left. + \int^t_{\widetilde{t}} \int_{\bR_0} \left(\gamma(X_{\widetilde{s}, r-}^{\iota, x})-\gamma(X_{\widetilde{s}, \max \{\widetilde{s}, \delta(r-) \}}^{\delta, x}) \right) z \widetilde{N}(dr, dz) \right|^p \right]\right]^{1/p} \notag\\
& \leq \left[\mathbb{E}\left[ \left|\int^t_{\widetilde{t}}  \left( \mu(X_{\widetilde{s}, r}^{\iota, x})-\mu(X_{\widetilde{s}, \max \{\widetilde{s}, \delta(r) \}}^{\delta, x}) \right) dr\right|^p \right]\right]^{1/p} + \left[\mathbb{E}\left[ \left|\int^t_{\widetilde{t}} \left( \sigma(X_{\widetilde{s}, r}^{\iota, x}) -\sigma(X_{\widetilde{s}, \max \{\widetilde{s}, \delta(r) \}}^{\delta, x})\right) dW_r \right|^p \right]\right]^{1/p} \notag\\
& \quad \quad  + \left[\mathbb{E}\left[ \left|\int^t_{\widetilde{t}} \int_{\bR_0} \left(\gamma(X_{\widetilde{s}, r-}^{\iota, x})-\gamma(X_{\widetilde{s}, \max \{\widetilde{s}, \delta(r-) \}}^{\delta, x}) \right) z \widetilde{N}(dr, dz) \right|^p \right]\right]^{1/p} \notag\\
&\leq c C_p |t-\widetilde{t}|^{1/p} \sup_{r\in [ \widetilde{t}, t]} \left[ \mathbb{E} \left[\left| X_{\widetilde{s}, r}^{\iota, x}-X_{\widetilde{s}, \max \{\widetilde{s}, \delta(r) \}}^{\delta, x} \right|^p \right]\right]^{1/p} \notag\\
&\leq c C_p |t-\widetilde{t}|^{1/p} \sup_{r\in [ \widetilde{t}, t]} \Bigg[ \left[ \mathbb{E} \left[\left| X_{\widetilde{s}, r}^{\iota, x}-X_{\widetilde{s}, \max \{\widetilde{s}, \delta(r) \}}^{\iota, x} \right|^p \right]\right]^{1/p} + \left[ \mathbb{E} \left[\left| X_{\widetilde{s}, \max \{\widetilde{s}, \delta(r) \}}^{\iota, x}-X_{\widetilde{s}, \max \{\widetilde{s}, \delta(r) \}}^{\delta, x} \right|^p \right]\right]^{1/p} \Bigg] \notag\\
&\leq  c C_p |t-\widetilde{t}|^{1/p} \sup_{r\in [ \widetilde{t}, t]} \Bigg[C_p | \max \{\widetilde{s}, \delta(r) \} -r |^{1/p} (e^{2.5\overline{c}T} V(x))^{1/p} \notag\\ 
& \quad +  C_p  |\max \{\widetilde{s}, \delta(r) \} -\widetilde{s} |^{1/p} |\delta|^{1/p}(e^{2.5\overline{c}T}V(x))^{1/p} \Bigg] \notag\\
&\leq  c C_p |t-\widetilde{t}|^{1/p} \Bigg[C_p |\delta|^{1/p} (e^{2.5\overline{c}T} V(x))^{1/p}  +  C_p T^{1/p}  |\delta|^{1/p} (e^{2.5\overline{c}T}V(x))^{1/p} \Bigg] \notag\\
& \leq C_p |t-\widetilde{t}|^{1/p} |\delta|^{1/p} (e^{2.5\overline{c}T}V(x))^{1/p}. 
%& \leq [\sqrt{2} c^2+c] \left[\sqrt{T} + 2+ \sqrt{m_2} \right]^4 e^{c^2 \left[\sqrt{T} + 2+ \sqrt{m_2} \right]^2T} e^{\overline{c}T} \sqrt{V(x)} |\delta|^{1/2} |t-\widetilde{t}|^{1/2}. 
\label{R2}
\end{align}
% This implies that
% \begin{align}
% R_1 &\le   C_p |\max\{s,\widetilde{s}, t, \widetilde{t}\} - \max\{s, t\}|^{1/p} |\delta|^{1/p} (e^{2.5\overline{c}T}V(x))^{1/p}, \\
%   R_3 &\le C_p |\max\{s,\widetilde{s}, t, \widetilde{t}\} - \max\{\widetilde{s}, \widetilde{t}\}|^{1/p} |\delta|^{1/p} (e^{2.5\overline{c}T}V(x))^{1/p}.
% \end{align}
{\it Estimation of $R_3$}: Applying the aforementioned statement (iv), we have
\begin{align}
R_3 & = \left[\mathbb{E}\left[\left| \left(X^{\iota, x}_{\widetilde{s},\widetilde{t}}-X^{\delta, x}_{\widetilde{s},\widetilde{t}}\right) - \left(X^{i, \widetilde{x}}_{\widetilde{s},\widetilde{t}}-X^{\delta, \widetilde{x}}_{\widetilde{s},\widetilde{t}}\right) \right| ^p\right]\right]^{1/p} \notag\\
& \leq C_p e^{C_p T} |\widetilde{t} -\widetilde{s}|^{1/p}|x-\widetilde{x}| |\delta|^{1/p} \notag\\
& \quad + C_p e^{C_p T} |\widetilde{t}-\widetilde{s}|^{1/p}|x-\widetilde{x}|  |\delta|^{1/pm} (e^{2.5 \overline{c}T})^{1/pm}         \frac{(V(x))^{1/pm}+(V(\widetilde{x}))^{1/pm}}{2}. 
\label{R3}
\end{align}
Therefore, combining $\eqref{R123},$ $\eqref{R1}$, $\eqref{R2}$ and $\eqref{R3}$, we obtain  
\begin{align*}
&\left[\mathbb{E}\left[\left|\left(X^{\iota, x}_{s,t}-X^{i,\widetilde{x}}_{\widetilde{s},\widetilde{t}}\right) - \left(X^{\delta, x}_{s,t}-X^{\delta, \widetilde{x}}_{\widetilde{s},\widetilde{t}}\right)\right|^p\right]\right]^{1/p} \notag\\
& \leq C_p e^{C_p T}  (e^{5\overline{c}T})^{1/p}|\delta|^{\frac{1}{p(m \vee \kappa_2)}} |\widetilde{s}-s |^{1/p\kappa_1}  (V(x))^{2/p} + C_p |t-\widetilde{t}|^{1/p} |\delta|^{1/p} (e^{2.5\overline{c}T}V(x))^{1/p} \notag\\
&\quad + C_p e^{C_p T} |\widetilde{t} -\widetilde{s}|^{1/p}|x-\widetilde{x}| |\delta|^{1/p}  \notag\\
& \quad + C_p e^{C_p T} |\widetilde{t}-\widetilde{s}|^{1/p}|x-\widetilde{x}|  |\delta|^{1/pm} (e^{2.5 \overline{c}T})^{1/pm}         \frac{(V(x))^{1/pm}+(V(\widetilde{x}))^{1/pm}}{2} \notag\\
 & \leq C_p e^{C_p T}   (e^{5 \overline{c}T})^{1/p}         \frac{(V(x))^{2/p}+(V(\widetilde{x}))^{2/p}}{2} \left[  |\widetilde{s}-s|^{1/p\kappa_1} + |t-\widetilde{t}|^{1/p} +|x-\widetilde{x}| \right]   |\delta|^{\frac{1}{p(m \vee \kappa_2)}}. 
\end{align*}
The proof of Theorem $\ref{Thm}$ is thus completed.
\end{proof}
Applying Theorem \ref{Thm} for the uniform grids, we have the following result. 
\begin{Cor}
 \label{Thm1.1}
    Let $p \ge 2,$ 
    %$m, \kappa_1\ge 1,$ 
    $T\in (0, \infty),$ $\mu, \sigma, \gamma  \in C^2(\mathbb{R}, \mathbb{R})$, $(\Om, \mf, \pr, (\mathcal{F})_{t\in [0, T]})$ be a filtered probability space which satisfies the usual conditions. %Let $W= (W_t)_{t\in [0,T]}: [0, T] \times \Om \to \mathbb{R}$ be a standard $(\mathcal{F})_{t\in [0,T]}$-Brownian motion with continuous sample paths, and 
    For every $n \in \mathbb{N}, x\in \mathbb{R}$, let $Y^{n,x} = \left(Y^{n, x}_t \right)_{t\in [0, T]}: [0, T] \times \Om \to \mathbb{R}$ satisfy for all $k \in \{0, \dots, n-1\},$ $t\in \left( \frac{kT}{n}, \frac{(k+1)T}{n} \right]$ that $Y_0^{n, x}=x$ and 
    $$Y_{t}^{n, x} = Y_{\frac{kT}{n}}^{n, x}+\mu \left(Y_\frac{kT}{n}^{n, x}\right) \left (t- \frac{kT}{n} \right) +\sigma \left(Y_\frac{kT}{n}^{n, x}\right) \left(W_{t}- W_{\frac{kT}{n}}\right) +  \gamma (Y_\frac{kT}{n}^{n, x})\left(Z_{t}- Z_{\frac{kT}{n}}\right).$$ 
    Then for 
    every $x\in \mathbb{R}$, there exists an uniqueness adapted stochastic process with c\`adl\`ag sample paths $(X_t^x)_{t\in [0, T]}: [0, T]\times \Om \to \mathbb{R}$ such that for all $t\in [0, T]$, it holds a.s that $$X_t^x=x+\int^t_0\mu(X^x_s)ds+\int^t_0\sigma(X^x_s)dW_s + \int^t_0 \int_{\bR_0} \gamma (X^x_{s-})z \widetilde{N}(d s, d z),$$ and 
    for any $p \geq 2$ and $m > 1$, there exists a constant $C> 0$ such that for all $n\in \mathbb{N},$  $x \in \mathbb{R},$ $y\in \mathbb{R} \setminus \{x\}$, $t\in [0, T]$, it holds that 
        \begin{align*}
        &  \dfrac{\left(\mathbb{E}\left[\left|X^x_{t}-X^{y}_{t}\right|^p\right]\right)^{\frac{1}{p}}}{|x-y| n^{\frac{1}{mp}}} 
        + \frac{\left(\mathbb{E}\left[\left|Y^{n,x}_{t}-Y^{n,y}_{t}\right|^p\right]\right)^{\frac{1}{p}}}{|x-y| n^{\frac{1}{mp}}}\\
&+\dfrac{\left(\mathbb{E}\left[\left|X^x_{t}-Y^{n, x}_t\right|^p\right]\right)^{\frac{1}{p}}}{1+|x|}+\dfrac{\left(\mathbb{E}\left[\left|\left(X^x_{t}-Y^{n, x}_t\right)-\left(X^y_{t}-Y^{n, y}_t\right)\right|^p\right]\right)^{\frac{1}{p}}}{|x-y|(1+|x|+|y|)} \leq \dfrac{C}{n^{\frac{1}{mp}}}.
        \end{align*}
\end{Cor}

\subsection*{Acknowledgements} 
The authors would like to thank Professor Arturo Kohatsu-Higa for fruitful discussions on the subject.

\end{document}